\documentclass[pdflatex,sn-mathphys-ay]{sn-jnl}

\usepackage{graphicx}%
\usepackage{multirow}%
\usepackage{amsmath,amssymb,amsfonts}%
\usepackage{amsthm}%
\usepackage{mathrsfs}%
\usepackage[title]{appendix}%
\usepackage{xcolor}%
\usepackage{textcomp}%
\usepackage{manyfoot}%
\usepackage{booktabs}%
\usepackage{algorithm}%
\usepackage{algorithmicx}%
\usepackage{algpseudocode}%
\usepackage{listings}%

\usepackage{bm}         

\newtheorem{theorem}{Theorem}[section]
\newtheorem{corollary}{Corollary}[theorem]
\newtheorem{lemma}[theorem]{Lemma}
\newtheorem{definition}{Definition}
\newtheorem{proposition}{Proposition}[section]
\newtheorem{remark}{Remark}

\newtheorem*{assertion*}{Assertion}

\usepackage{enumitem}
\usepackage{bbm}
\usepackage{amssymb}    
\usepackage{amsthm}     
\usepackage{textcomp}   
\usepackage{gensymb}    
\usepackage{dsfont}

\newcommand{\spa}{X}
\newcommand{\reeb}{\mathrm{R}}
\newcommand{\man}{\mathrm{M}}
\newcommand{\sample}{\mathrm{X}}
\newcommand{\cluster}{\mathcal{C}}
\newcommand{\mes}{\mathrm{m}}
\newcommand{\dis}{\mathrm{d}}
\newcommand{\supp}{\mathrm{supp}}
\newcommand{\var}{\mathrm{Var}}
\newcommand{\wasser}{\mathrm{W}}
\newcommand{\gromov}{\mathrm{GW}}
\newcommand{\Reeb}{{\mathrm R}}
\newcommand{\critic}{{\mathrm {Crit}}}
\newcommand{\comp}{{\mathrm {C}}}
\newcommand{\appr}{{\mathrm {A}}}
\newcommand{\estim}{{\mathrm {E}}}
\newcommand{\haus}{{\mathrm {d_H}}}
\newcommand{\ball}{{\mathrm {B}}}
\newcommand{\converge}{\underset{n\rightarrow\infty}{\longrightarrow}}
\newcommand{\Ys}{Y}
\newcommand{\interv}{{\mathrm {I}}}
\newcommand{\mapper}{{\mathbb{M}_n}}
\newcommand{\simp}{\Delta}
\newcommand{\DD}{\Delta\!\!\!\!\Delta}
\newcommand{\dirac}{\delta}
\newcommand{\diracn}{\mathrm{m}_n}
\newcommand{\diam}{\mathrm{Diam}}

\newcommand{\covn}{\mathcal{N}}
\newcommand{\vol}{\mathrm{Vol}}
\newcommand{\ind}{\mathds{1}}

\newcommand{\coverwidth}{\change{\delta_r}}
\newcommand{\g}{\mathrm{g}}
\newcommand{\G}{\mathrm{G}}
\newcommand{\deter}{\mathrm{det}}
\newcommand{\tangent}{\mathrm{T}}
\newcommand{\ricci}{\mathrm{Ric}}
\newcommand{\spaceform}{\mathrm{S}}

\newcommand{\R}{\mathbb{R}}
\newcommand{\N}{\mathbb{N}}
\newcommand{\length}{\mathrm{length}}
\newcommand{\eps}{\varepsilon}
\newcommand{\preim}[2]{#1^{-1}(#2)}
\newcommand{\torus}{\mathbb{T}^2}
\newcommand{\embd}{T}
\newcommand{\mub}{{\mu_{\text{max}}}}
\newcommand{\crib}{{n_{\text{crit}}}}
\newcommand{\class}{\mathcal{F}_{\mub,\crib}}
\newcommand{\inorm}[1]{\Vert #1\Vert_\infty}

\newtheorem{hypothesis}{Assumption}

\newcommand{\change}[1]{\noindent #1}
\newcommand{\ff}{h}
\theoremstyle{thmstyleone}%

\begin{document}

\title[Gromov-Wasserstein bound between Reeb and Mapper graphs]{Gromov-Wasserstein bound between Reeb and Mapper graphs}


\author*[1]{\fnm{Ziyad} \sur{Oulhaj}}\email{ziyad.oulhaj@ec-nantes.fr}

\author[2]{\fnm{Mathieu} \sur{Carri\`ere}}

\author[1]{\fnm{Bertrand} \sur{Michel}}

\affil*[1]{\orgdiv{Laboratoire de Math\'ematiques Jean Leray, CNRS UMR 6629}, \orgname{Nantes Universit\'e, \'Ecole Centrale Nantes}, \orgaddress{ \city{Nantes}, \country{France}}}

\affil[2]{\orgdiv{DataShape}, \orgname{Centre Inria d'Universit\'e C\^ote d'Azur}, \orgaddress{ \city{Sophia Antipolis}, \country{France}}}



\abstract{Since its introduction as a computable approximation of the Reeb graph, the Mapper graph has become one of the most popular tools from topological data analysis for performing data visualization and inference. However, finding an appropriate metric (that is, a tractable metric with theoretical guarantees) for comparing Reeb and Mapper graphs, in order to, e.g., quantify the rate of convergence of the Mapper graph to the Reeb graph, is a difficult problem. While several metrics have been proposed in the literature, none is able to incorporate measure information, when data points are sampled according to an underlying probability measure. The resulting Reeb and Mapper graphs are therefore purely deterministic and combinatorial, and substantial effort is thus required to ensure their statistical validity. In this article, we handle this issue by treating Reeb and Mapper graphs as metric measure spaces. This allows us to use Gromov-Wasserstein metrics to compare these graphs directly in order to better incorporate the probability measures that data points are sampled from. Then, we describe the geometry that arises from this perspective, and we derive rates of convergence of the Mapper graph to the Reeb graph in this context. Finally, we showcase the usefulness of such metrics for Reeb and Mapper graphs in a few numerical experiments. }




\maketitle

\section{Introduction} \label{sec:disc}

The Mapper algorithm is a popular method from topological data analysis (TDA) for data visualization and inference. It provides a synthetic representation of a given dataset based on the topological variations of a continuous function defined on the data, often referred to as {\em filter function}. This representation typically takes the form of a graph, which facilitates data visualization and exploration. Indeed, the topological features of the Mapper graph (connected components, branches, loops, etc) are representatives of the topological features of the dataset, and can be used to identify its structures and subpopulations of interest. See Figure~\ref{fig:mapper_example} for an illustation. As such, Mapper graphs have been successfully used in numerous applications, including, but not limited to, 3D meshes~\citep{wang2020exploration,rosen2018inferring}, single-cell sequencing~\citep{wang2018topological,zechel2014topographical},  machine learning~\citep{Bruel-Gabrielsson2018, Naitzat2018}, or neural network architectures~\citep{Mitra21,joseph2021topological}. 

More fundamentally, the Mapper graph can be seen as a discrete version of the Reeb graph, which is a topological quotient space obtained by identifying the connected components of the level sets of a filter function, see Figure~\ref{fig:reeb_graph_example} for an illustration. Since the introduction of the Mapper algorithm, several metrics have been proposed to compare the resulting Mapper graph with its target Reeb graph, and more generally to compare several Mapper graphs and Reeb graphs together. As both objects are intrinsically combinatorial when computed over smooth enough filter functions (such as Morse type functions), 
these metrics focus on the combinatorial and algebraic aspects of Reeb and Mapper graphs. This point of view is also reinforced by the fact that the Mapper was initially introduced and studied in a deterministic setting for finite metric spaces. 

\change{One of the first lines of work concerning Reeb/Mapper graph metrics was initiated by \citet{de2016categorified}, which studies Reeb graphs from an algebraic perspective by viewing them as cosheaves, a viewpoint extended to Mapper graphs by \citet{munch_convergence_2016}, who show that the interleaving distance between them is well-defined and bounded by the cover resolution. Another approach by \citet{carriere2018structure} introduces pseudometrics based on extended persistence diagrams, which yields stability guarantees via the usual stability theorem of persistent homology combined with Mapper-specific bottleneck distances. Additionally, alternative metrics such as the edit distance~\citep{bauer_edit_2016, bauer_reeb_2021}, which counts graph transformations, and the functional distortion distance~\citep{bauer_measuring_2014}, which measures distortion of path lengths under continuous mappings, have been proposed and also enjoy theoretical guarantees (universality and stability, respectively).}

\change{Conceptually, the Mapper graph can be viewed as a stochastic object built from a random set of data points sampled in a metric space according to some probability measure (which is, in most applications, unknown).} It is thus natural to study the convergence of the Mapper graph to the Reeb graph with a {\em statistical} perspective, using the metrics mentioned above. For instance, an alternative Mapper construction called the \emph{enhanced Mapper} is given by~\citet{brown2021probabilistic}, in which it is proved to approximate the Reeb graph (up to the resolution of the cover) w.r.t. the interleaving distance with high probability, when computed on a large enough sample.  Based on the extended persistent homology approach developed by~\citet{carriere2018structure}, \citet{carriere2018statistical} provide rates of convergence of Mapper graphs to the Reeb graph in expected value using an appropriate variation of the bottleneck distance.  
Finally, the statistical convergence of the Mapper complex (in the case of multivariate and stochastic filters) has been proposed by~\citet{carriere2022statistical} using the Gromov-Hausdorff distance.

However, in all of these approaches, the target Reeb graph remains a fully deterministic object: although the sampling measure is taken into account in the construction of the Mapper graph or in the necessary hypotheses to ensure convergence, the Reeb graph does not come with a measure-oriented description. {\em In this article, we propose a new perspective on Mapper and Reeb graphs by considering these objects as metric measure spaces.} This approach is motivated by theoretical considerations - we want to enrich the mathematical descriptions of these objects - and also by practical applications of Mapper graphs, in which they are frequently visualized with node sizes corresponding to their respective masses, i.e., how many points they contain. 

\paragraph{Contributions.} In this article, we study the Reeb and Mapper graphs of Morse filter functions for data points sampled on probability measures supported on Riemannian manifolds. {\bf Our contributions are two-fold:}

\begin{itemize}
    \item {\bf We endow Reeb and Mapper graphs with metric measure space structures}, and then rigorously introduce the Gromov-Wassertein metric between these spaces. This point of view can be seen as the measure-aware version of the Gromov-Hausdorff metric proposed by~\citet{carriere2022statistical}. Overall, this allows us to incorporate measure information in the computation of distances between Reeb and Mapper graphs, an information that is usually lost or hidden in other approaches, due to their combinatorial nature. We moreover prove that the Gromov-Wasserstein distance between Reeb graphs is stable w.r.t. the change of filter function.
    
    \item {\bf We study the convergence of the Mapper to the Reeb graph in this framework.} In our main result (Theorem~\ref{thm:main_result}), we provide an upper bound on the expectation of the $p$-Gromov-Wasserstein  distance (for any $p \geq1$) between the Mapper graph, chosen with an appropriate resolution, and the Reeb graph, under a sampling generative model. This bound is of the order of  $n^{-\frac{\nu}{d+\alpha}} $,
    where $\nu=\min\left\{\frac{1}{2},\frac{d}{p(d+1)}\right\}$ for any $\alpha >0$, and where $d$ is the metric dimension of the measure. This upper bound relies in part on results on the convergence of the empirical measure in mean Wasserstein distance in a Polish metric space~\citep{weed2019sharp}. 
\end{itemize}

\paragraph{Related work.} Since their introduction by \change{\citet{sturm2006geometry} and subsequently, under a different formulation, by \citet{memoli2007use}} for object matching, Gromov-Wasserstein metrics have been widely used in many applications for comparing 
heterogeneous data, including, e.g., shape and graph matching \citep{memoli2009spectral,xu2019gromov,koehl2023computing}, which makes them particularly well-suited for comparing 
Reeb and Mapper graphs. 
While applications and computational aspects of the Gromov-Wasserstein metrics have been studied extensively, the statistical aspects have not been carefully studied until very recently \citep{han2023covariance,zhang2024gromov}. 
In particular, \citet{zhang2024gromov} show that the empirical quadratic Gromov-Wasserstein convergence rate over Euclidean spaces of different dimensions $d_x$ and $d_y$ is less than $n^{-2/ \max(\min(d_x,d_y),4)}$. Note that this result does not apply straightforwardly to the study of Reeb and Mapper graphs (which cannot be seen as Euclidean spaces), 
and we follow a different route for deriving our rates of convergence. \change{Note also that the Gromov-Wasserstein metric formulation that we focus on in this article is the one given by \citet{sturm2006geometry}, even though the upper bound we give in our main result is valid for both formulations.}\\
\citet{wangMeasureTheoreticReebGraphs2024} also aim at adapting Reeb graphs (and Reeb spaces) for metric measure space inputs. The crucial difference with our method is in the treatment of the measure. The goal of~\citet{wangMeasureTheoreticReebGraphs2024} is to 
produce a new measure-aware space 
and filter 
for the computation of Reeb graphs,
and to study their stability properties. However, our approach directly incorporates the measure in the Reeb and Mapper graphs by considering them as metric measure spaces.\\ Similarly, \citet{ruscitti2024improving} have proposed a modification of the Mapper algorithm to account for the data distribution in the filter space, in order to define better suited open covers, and then provide rates of convergence with respect to the bottleneck distance, in a similar vein than \citet{carriere2018statistical}. Again, our approach differs from this one in that the measure is not used for a better tuning of the Mapper parameters, but is rather directly incorporated in both Reeb and Mapper graphs, and in the distance between them.

%
%
Finally, \citet{liFlexibleProbabilisticTopology2025} provide Gromov-Wasserstein stability bounds for \textit{merge trees}---which are similar yet quite different from Reeb graphs as they encode the connectivity of sublevel sets (rather than level sets) of filter functions---computed from simplicial complexes equipped with probability measures on their vertices. Further differences are discussed below Theorem~\ref{thm:stability}.


\paragraph{Summary.} Section~\ref{sec:background} provides necessary background on Reeb and Mapper graphs, as well as some elements of Morse theory. In Section~\ref{sec:reeb_mms}, we define a metric structure on the Reeb graph with the Hausdorff distance, study its induced topology and prove stability w.r.t. the change of filter function for the associated Gromov-Wasserstein distance. Similarly, we introduce a metric measure space structure for the Mapper graph in Section~\ref{sec:mapper_mms}. Section~\ref
{sec:mapper} presents our main result about the rate of convergence of Mapper graphs to Reeb graphs in terms of the Gromov-Wasserstein distance, along with its proof. Finally, numerical experiments that illustrate the practical use of Gromov-Wasserstein metrics to compare Mapper graphs are provided in Section~\ref{sec:appl}.

\section{Background}\label{sec:background}
\subsection{Reeb and Mapper graphs}
\label{subsec:reeb_mapper_def}

We start by introducing Reeb graphs defined on general topological spaces.
\begin{definition}[\change{Reeb graph}]
    Let $\spa$ be a topological space and let $f\,:\,\spa\rightarrow \R$ be a continuous 
function called {\it filter function}. Let $\sim_{f}$ be the equivalence relation between two elements $x$ and $y$ in $\spa$ defined by: $x\sim_{f} y$ if and only if $x$ and $y$ are in the same connected component of $\preim{f}{\{v\}}$ for some $v$ in $f(X)$. 
 The Reeb graph $\Reeb_f(\spa)$ of  $\spa$ 
is then defined as the quotient space $\spa/\sim_{f}$.
\end{definition}

Figure \ref{fig:reeb_graph_example} provides an illustration of a Reeb graph computed on a torus with its height as filter function.

\begin{figure}[t]
    \centering
    \includegraphics[width=0.5\textwidth]{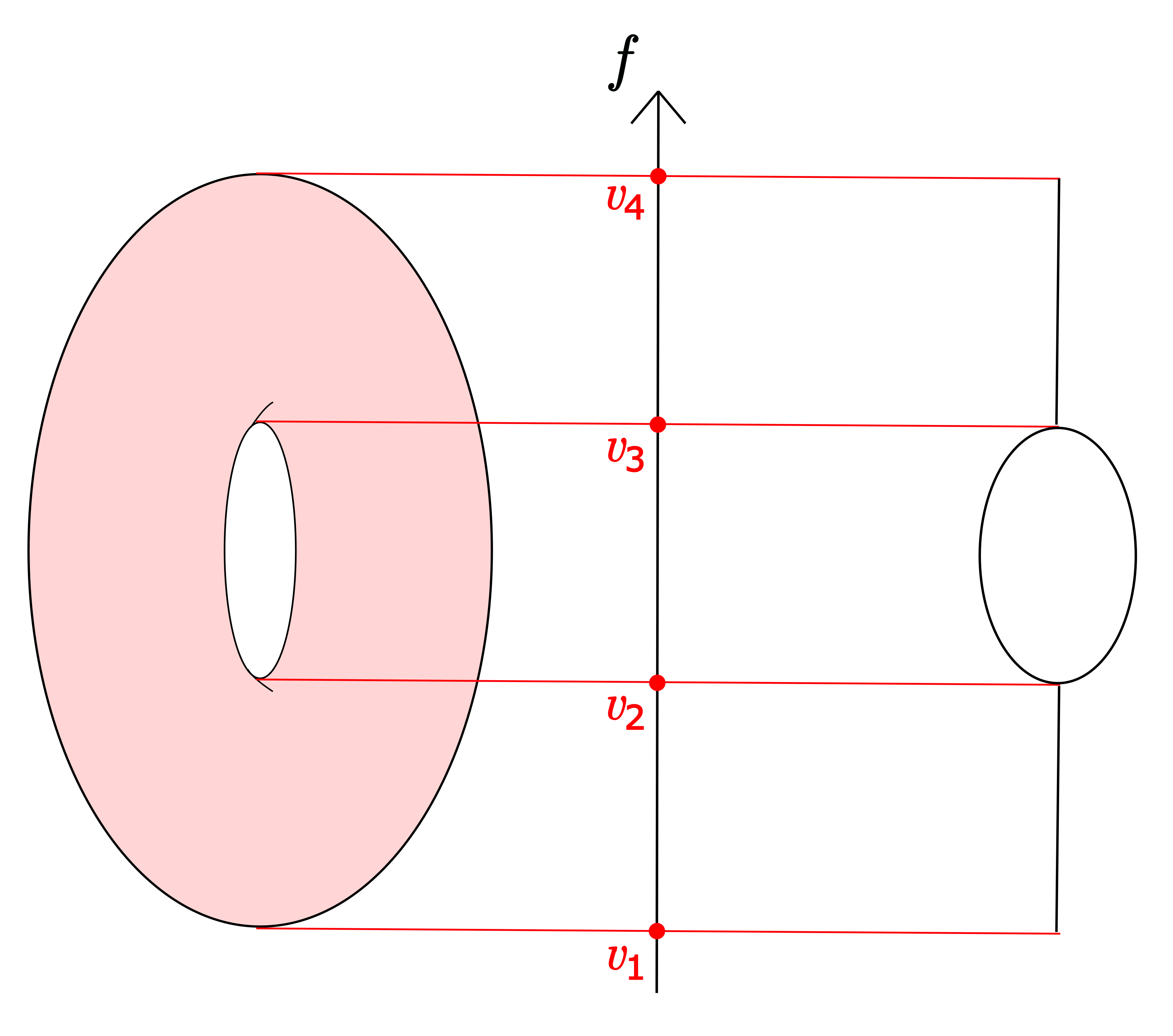}
    \caption{Example of a Reeb graph computed on a torus using height as filter function. The critical values $\{v_1,v_2,v_3,v_4\}$ of the filter function $f$ are represented in the middle. The Reeb graph is represented on the right.}
    \label{fig:reeb_graph_example}
\end{figure}

Assume now that $(\spa,\dis)$ is  a metric space and let $f\colon\spa\rightarrow\R$ be a continuous function. Mapper was introduced by \citet{singh2007topological} as a discrete and computable version of the Reeb graph $\Reeb_f(\spa)$. 
Given a point cloud 
$\sample_n= \{x_1,\dots,x_n\}\subseteq \spa$, the Mapper graph can then be computed with the following algorithm:
\begin{enumerate}
\item Cover the range of values $\Ys_n = f(\sample_n)$ with a set of consecutive intervals $\interv_1,\dots, \interv_r$  that overlap, i.e., one has $\interv_i\cap \interv_{i+1}\neq \varnothing$ for all $1\leq i \leq r-1$.
\item Group the points that fall in the same connected component of each preimage $\preim{f}{\interv_j}$, $j\in\{1,...,r\}$. This defines a {\it pullback cover}
$\mathcal{C}=\{\cluster_{1,1},\dots,\cluster_{1,k_1},\dots,\cluster_{r,1},\dots,\cluster_{r,k_r}\}$ of 
$\sample_n$. 
\item The Mapper graph is defined as the {\it nerve} of $\mathcal{C}$. Each node $v_{j,k}$ of the Mapper graph corresponds to an element $\cluster_{j,k}$ of $\mathcal C$, 
and two nodes $v_{j,k}$ and $v_{j',k'}$ are connected by an edge if and only if  $\cluster_{j,k} \cap \cluster_{j',k'} \neq \varnothing$.
\end{enumerate} 

Figure~\ref{fig:mapper_example} provides a Mapper graph computed on points sampled from a torus with its height as filter function, and a cover of the filter image with three intervals. 

\begin{figure}[t]
    \centering
    \includegraphics[width=\textwidth]{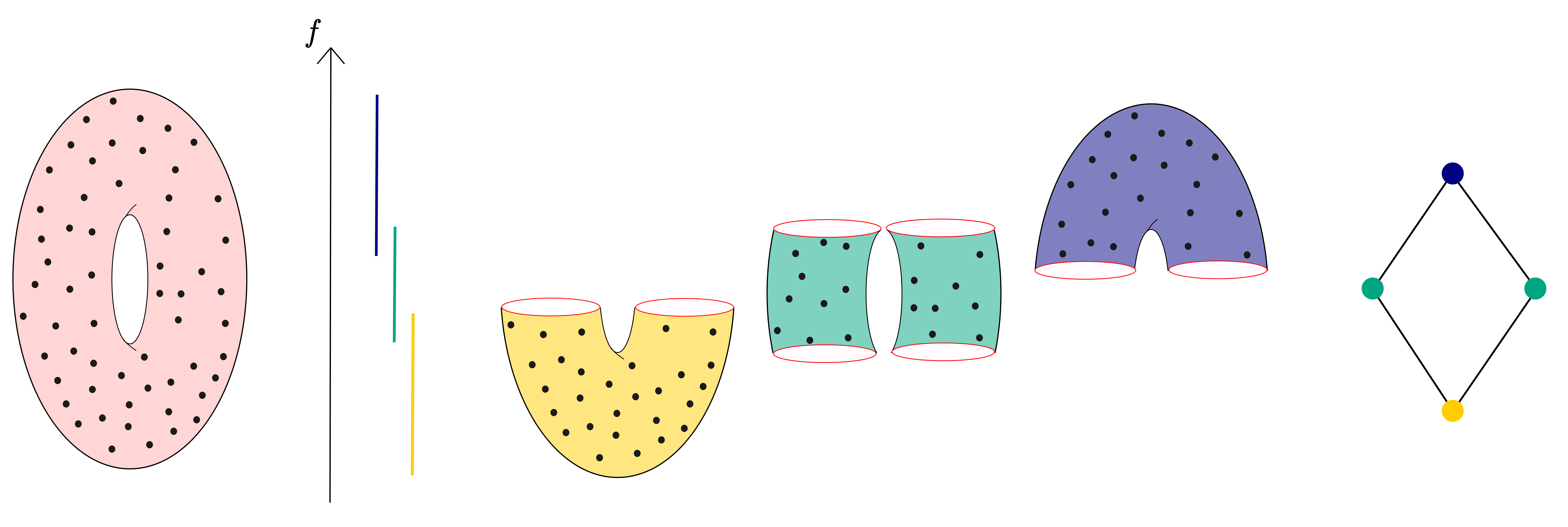}
    \caption{Example of a Mapper graph computed on points sampled from a torus with its height as filter function. The three intervals used for covering its image are represented in different colors.}
    \label{fig:mapper_example}
\end{figure}

In practice, the second step of the above algorithm is performed using a clustering algorithm, but we leave this consideration aside 
as it is more suitable for a theoretical discussion. Note that some clustering algorithms come with guarantees as to whether they are able to successfully estimate the connected components of the original space $\spa$. See for instance~\citet{arias2011spectral} for guarantees on spectral clustering.


\subsection{Elements of Morse theory}
 We will assume further in this article that we observe data sampled from a Riemannian manifold, as well as the values of a continuous filter function defined on the points. We introduce in Appendix~\ref{apdx:elemRiem} some elements of Riemannian geometry that will be used throughout this work. Although these are standard results, we also recall some proof elements for the sake of completeness. A general presentation of Riemannian geometry, and notably a proof of the Bishop-Gromov theorem (see Theorem~\ref{thm:bishop_gromov}), can be found in~\citet{petersen2006riemannian}.\\
 
In the article, we study the convergence of Mapper graphs towards Reeb graphs under the assumption that the filter function is a Morse function. Note that this assumption is not restrictive as the set of Morse functions is a dense open subset of $C^{\infty}(\man)$, for $\man$ a compact manifold. The aim of this section is to recall the cylindrical structures of preimages of intervals (under a Morse filter function) around non-critical points. See for instance~\citet{milnor1963morse} for a comprehensive presentation of Morse theory.

Let $\man$ be a Riemannian manifold. \change{A} smooth map $f\,:\,\man\rightarrow \mathbb{R}$ is called a {\it Morse function} if its critical points are non-degenerate, i.e., the Hessian of $f$ at the points where its gradient vanishes is non-singular.  We first recall the standard Morse Lemma, which describes the behavior of a Morse function in the neighborhood of a critical point.  Let $\critic(f)$ denote the set of critical points of a Morse function $f\,:\,\man\rightarrow \R$ defined on $\man$. We will denote $\Vert\cdot\Vert^2:=\g(\cdot, \cdot)$ where $\g$ is the metric on $\man$.

\begin{lemma}[\change{Morse Lemma}]\label{lem:morse}
Let $f\,:\,\man\rightarrow \mathbb{R}$ be a Morse function and $c\in\critic(f)$. Then there exists a chart $\varphi\,:\,U\subseteq\man\rightarrow \mathbb{R}^d$ containing $c$ such that for every $p\in U$:
    $$f(p)=f(c)-\sum_{j=1}^ix_j^2+\sum_{j=i+1}^dx_j^2,$$
    where $x=\varphi(p)$ and the integer $i$ depends only on the signature of the Hessian at the critical point $c$.
\end{lemma}

In particular, the Morse Lemma implies that the critical points of a Morse function are isolated, and $\critic(f)$ is finite when $f$ is a Morse function defined on a compact manifold.

\change{As $f$ is smooth, whenever $v\in\mathbb{R}$ is a regular value of $f$ (i.e. $\preim{f}{\{v\}}$ is non-empty and does not contain critical points), we know that $\preim{f}{\{v\}}$ is a $(d-1)-$dimensional submanifold of $\man$ (for a proof see Theorem 9.8 in~\citet{tu2011manifolds} for example).}

Next, we recall a major result that relates the topology of a manifold to the analytic properties of a Morse function defined on it.

\begin{lemma}[\change{Gradient flow of a smooth function}]\label{lem:flow}
Let $f\,:\,\man\rightarrow \R$ be a smooth function defined on a compact manifold $\man$ and $a,b\in\R$ such that $a<b$.
If $\preim{f}{[a,b]}\cap \critic(f)=\varnothing$,
then there exists a diffeomorphism $\psi_{b-a}$ between $\preim{f}{\{a\}}$ and $\preim{f}{\{b\}}$.

\end{lemma}
We provide the proof of this well-known result \change{in Appendix~\ref{apdx:proofs}}, as we will make use of this gradient flow several times later in the article. 

This following result is proved implicitly in Theorem 3.1 in~\citet{milnor1963morse}. It is used for example by~\citet{de2016categorified} to show that Reeb graphs are an example of what the authors call {\it constructible spaces}. For the sake of completeness, we give a full proof \change{in Appendix~\ref{apdx:proofs}}. See Figure \ref{fig:reeb_cylindre} for an illustration.

\begin{theorem}[\change{Cylindrical shape around non-critical points}]\label{thm:morse}
Let $f\,:\,\man\rightarrow \R$ be a Morse function defined on a connected compact Riemannian manifold $\man$. Let also $x\in\man$ such that $x\not\in \critic(f)$, 
and let $v=f(x)$. 
For a small enough $\eps>0$ such that $\preim{f}{[v-\eps,v+\eps]}$ does not contain critical points, the following map is a homeomorphism:
\begin{align*}
\psi\colon \preim{f}{\{v\}}\times[-\eps,\eps] &\longrightarrow \preim{f}{[v-\eps,v+\eps]}  \\
(q,t)&\longmapsto \psi_{t}(q)
\end{align*}
\end{theorem}

\begin{figure}[t]
    \centering
    \includegraphics[width=0.7\textwidth]{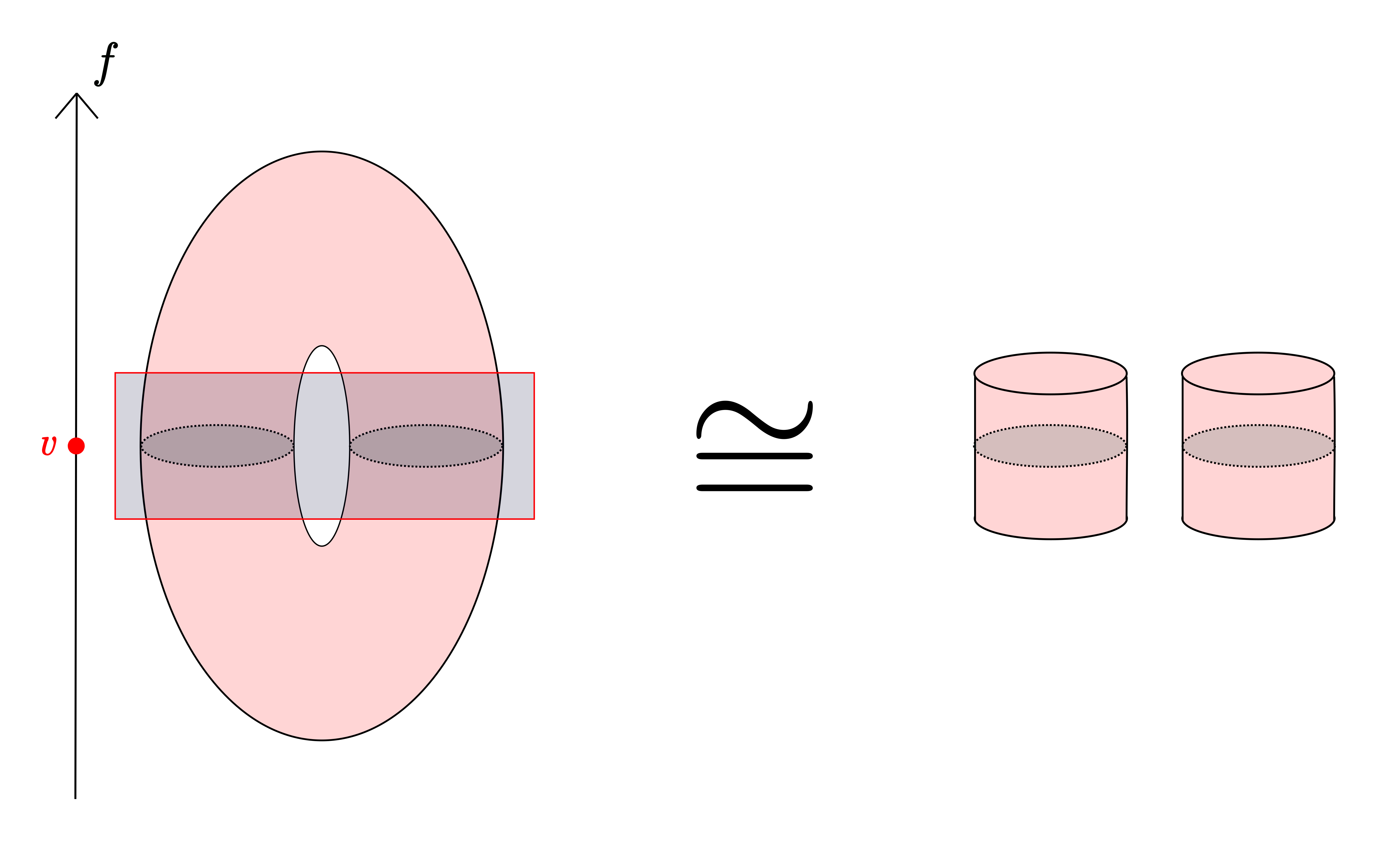}
    \caption{\label{fig:reeb_cylindre} Around a non-critical value $v$, the manifold is homeomorphic to a finite collection of cylinders, whose faces are given by the level set $\preim{f}{\{v\}}$.}
\end{figure}

In the context of studying Reeb graphs built on Morse functions, we will use the following Proposition \change{whose proof is also provided in Appendix~\ref{apdx:proofs}}.
\begin{proposition}\label{prop:local_path_con}
Let $f\,:\,\man\rightarrow \R$ be a Morse function defined on a connected compact Riemannian manifold $\man$. The level sets of $f$, $ \preim{f}{\{v\}}$ for $v\in f(\man)$, are locally path connected.
\end{proposition}

\change{
Proposition~\ref{prop:local_path_con} has two major applications. First, when looking at the level sets of a Morse function defined on a compact Riemannian manifold, connectedness and path connectedness are equivalent. Therefore, in what follows, we will use the term \emph{connected} to mean both. Second, the number of connected components of a level set of a  Morse function defined on a compact Riemannian manifold is always finite, as a level set is always compact and locally path connected.}

\subsection{Geometry of Metric Measure Spaces}
In this section, we present the framework of metric measure space geometry introduced by~\citet{sturm2006geometry}. We first recall that 
connected compact Riemannian manifolds fall under this framework, and then we introduce Gromov-Wasserstein metrics between metric measure spaces.

\subsubsection{Metric Measure Spaces}

A {\it metric measure space} (or mm-space for short) is a triple $(\spa,\dis,\mes)$ where $(\spa,\dis)$ is a Polish metric space (i.e., complete and separable)
and where $\mes$ is a measure on the Borel $\sigma$-algebra of $(\spa,\dis)$ 
    that is locally finite.\footnote{By locally finite, we mean that $\forall x\in \spa,\,\exists r>0$ such that $\mes(\ball(x,r))<\infty$.}
The support of $\mes$
is denoted as $\supp(\mes)$. \change{Like mentioned above, connected compact Riemannian manifolds equipped with the geodesic distance and the volume measure fall under this definition.}

\begin{proposition}
    A connected compact Riemannian manifold $(\man,\g)$ together with its geodesic distance $\dis$ and the volume measure $(\man,\dis,\vol)$ is a mm-space. 
\end{proposition}

\begin{proof}
Since $\man$ is connected and compact, 
it is a complete metric space. It is also separable since 
it is second-countable. $\man$ being compact, the volume measure
$\vol$ is locally finite.
\end{proof}

\change{We now introduce some useful notions concerning metric measure spaces, beginning with isomorphisms.} Let $(\spa_1,\dis_1,\mes_1)$ and $(\spa_2,\dis_2,\mes_2)$ be two mm-spaces. \change{A} map $\Phi\colon\supp(\mes_1)\rightarrow\supp(\mes_2)$ is an isomorphism of mm-spaces between $(\spa_1,\dis_1,\mes_1)$ and $(\spa_2,\dis_2,\mes_2)$ if $\Phi$ is an isometry of metric spaces
    and if $\mes_2$ is the pushforward measure of $\mes_1$ under the map $\Phi$.
As isomorphisms induce an equivalence relation of mm-spaces, we denote the isomorphism equivalence class of a mm-space $(\spa,\dis,\mes)$ as $[\spa,\dis,\mes]$.

Metric measure space isomorphisms also induce the notion of {\it variance} of a mm-space $(\spa,\dis,\mes)$, defined as:
$$\var(\spa,\dis,\mes)=\inf_{\substack{(\spa',\dis',\mes')\in [\spa,\dis,\mes]\\ z\in \spa'}}\int_{\spa'}\dis'(x,z)^2\,d\mes'(x).$$
If $(\spa,\dis,\mes)$ is a compact mm-space then clearly one has $\inf_{z\in \spa}\int_{\spa}\dis(x,z)^2\,d\mes(x) < \infty $ and then $\var(\spa,\dis,\mes)$ is finite. Moreover, the variance is by definition an invariant of mm-space isomorphisms.


\subsubsection{Gromov-Wasserstein Distance}
\change{Let us first introduce the Wasserstein distance between two measures in the same complete and separable metric space.}

\begin{definition}[\change{Wasserstein distance}]
Let $(\spa,\dis)$ be a complete and separable metric space. Let $\mes_1$ and $\mes_2$ be two measures on the Borel $\sigma$-algebra of $(\spa,\dis)$. For $p \geq 1$, the {\it $p$-Wasserstein distance} between $\mes_1$ and $\mes_2$ is defined as:
$$\wasser_p(\mes_1,\mes_2)=\left(\inf_{\mes}\int_{\spa\times\spa} \dis(x,y)^p\,d\mes(x,y)\right)^{\frac{1}{p}},$$
where the infimum is taken over all measures $\mes$ on $\spa\times \spa$ that have marginals $\mes_1$ and $\mes_2$.
\end{definition}

\change{We are now ready to define the Gromov-Wasserstein distance between mm-spaces, which will be the distance that we consider to measure the convergence of the Mapper graph to the Reeb graph. We first define metric couplings of mm-spaces.}

\begin{definition}[\change{Metric coupling}]
Let $(\spa_1,\dis_1,\mes_1)$ and $(\spa_2,\dis_2,\mes_2)$ be two mm-spaces.
A {\it metric coupling} of $\dis_1$ and $\dis_2$ is any pseudometric\footnote{A pseudometric can be zero for non-equal points.} $\dis$ on $\spa_1\sqcup \spa_2$ that satisfies:
$$\forall x,y\in \supp(\mes_1):\,\dis(x,y)=\dis_1(x,y),$$
$$\forall x,y\in \supp(\mes_2):\,\dis(x,y)=\dis_2(x,y).$$
\end{definition}

\begin{definition}[\change{Gromov-Wasserstein distance}]
 For $p \geq 1$, the  {\it $p$-Gromov-Wasserstein distance} between two mm-spaces $(\spa_1,\dis_1,\mes_1)$ and $(\spa_2,\dis_2,\mes_2)$ is defined as:
$$\gromov_p((\spa_1,\dis_1,\mes_1),(\spa_2,\dis_2,\mes_2))=\left(\inf_{\dis,\mes}\int_{\spa_1\times \spa_2} \dis(x,y)^p\,d\mes(x,y)\right)^{\frac{1}{p}},$$
where the infimum is taken over all measures $\mes$ on $\spa_1\times \spa_2$ that have marginals $\mes_1$ and $\mes_2$, and over all metric couplings $\dis$ of $\dis_1$ and $\dis_2$.
\end{definition}

    For the sake of simplicity, we will denote $\gromov_p((\spa_1,\dis_1,\mes_1),(\spa_2,\dis_2,\mes_2))$ simply as $\gromov_p(\spa_1,\spa_2)$ when there is no ambiguity over the metric and the measure associated to each space.\\
\change{The Gromov-Wasserstein distance can be equivalently defined as follows.}

\begin{proposition}\label{prop:gw_def}
The $p$-Gromov-Wasserstein distance between two mm-spaces, $(\spa_1,\dis_1,\mes_1)$ and $(\spa_2,\dis_2,\mes_2)$, can be equivalently defined as:
$$\gromov_p((\spa_1,\dis_1,\mes_1),(\spa_2,\dis_2,\mes_2))=\inf_{\varphi_1,\varphi_2}\wasser_p(\mes_1\circ\varphi_1^{-1},\mes_2\circ\varphi_2^{-1}),$$
where the infimum is taken over all isometric embeddings $\varphi_1\colon\supp(\mes_1)\rightarrow \spa$ and $\varphi_2\colon\supp(\mes_2)\rightarrow \spa$ of the supports of $\mes_1$ and $\mes_2$ into a common metric space $(\spa,\dis)$.
\end{proposition}

\change{As mentioned previously, the definition of the Gromov-Wasserstein distance above is due to~\citet{sturm2006geometry}. Note that therein, this distance is referred to as the $L^p-$transportation distance. \citet{memoli2007use}} provides an alternative definition of the $p$-Gromov-Wasserstein distance, inspired by a reformulation of the Gromov-Hausdorff distance, as:
$$\Hat{\gromov_p}((\spa_1,\dis_1,\mes_1),(\spa_2,\dis_2,\mes_2))=\frac 12 \left(\inf_{\mes}\int_{(\spa_1\times \spa_2)^2} |\dis_1(x,x')-\dis_2(y,y')|^p\,d\mes(x,y)\otimes\mes(x',y')\right)^{\frac{1}{p}},$$
where the infimum is taken over all measures $\mes$ on $\spa_1\times \spa_2$ that have marginals $\mes_1$ and $\mes_2$. Due to its computational advantage for practical applications (such as, e.g., optimization), this formulation is popular in the literature, see for example~\citet{zhang2024gromov,memoli2009spectral}.
Moreover, it is shown by~\citet{memoli2011gromov}  that $\Hat{\gromov_p}\leq\gromov_p$ (see Theorem 5.1(g) therein). In this article, we provide upper bounds for $\gromov_p$ which therefore automatically transfer to $\Hat{\gromov_p}$.\\
\change{Finally, the following result shows that the Gromov-Wasserstein distance characterizes mm-space isomorphism equivalence classes. 
}

\begin{proposition}\label{prop:gw_met}
The $2$-Gromov-Wasserstein distance $\gromov_2$ is a metric on the set of mm-space isomorphism classes of finite variance.    
\end{proposition}

For proofs of Propositions \ref{prop:gw_def} and \ref{prop:gw_met}, see \citet{sturm2006geometry}. \label{sec:mm_space_geometry}
\section{Reeb and Mapper graphs as Metric Measure Spaces}\label{sec:reeb_mms}

In this article, we aim at comparing Reeb graphs and Mapper graphs using tools and distances for metric spaces. In this section, we first define a metric structure on the Reeb graph using the Hausdorff distance, and then we study the induced topology of this corresponding metric space. \change{We then show that the Gromov-Wasserstein distance between Reeb graphs is stable with respect to filter function perturbations, before introducing a similar metric measure space version of Mapper graphs.}

\subsection{Hausdorff distance on Reeb graphs}

We will denote the set of non-empty compact subsets of a metric space $(\spa,\dis)$ as $\comp(\spa)$.  See~\citet{henrikson1999completeness} for a comprehensive introduction of the Hausdorff distance. 

\begin{definition}[\change{Hausdorff distance}]
 Let $A,B$ be two non-empty subsets of a metric space $(\spa,\dis)$, the Hausdorff distance between $A$ and $B$ is defined as:
 $$\haus(A,B)=\max\left(\sup_{x\in A}\dis(x,B),\sup_{y\in B}\dis(y,A)\right),$$
 where $\dis(x,B)=\inf_{v\in B}\dis(x,v)$ and $\dis(y,A)=\inf_{u\in A}\dis(u,y)$.
\end{definition}
With the same notations as before,  
the Hausdorff distance
can be equivalently defined as:
$$\haus(A,B)=\inf\left\{\eps>0,\,A\subseteq B^\eps\;\mathrm{ and }\;B\subseteq A^\eps\right\},$$
where $C^\eps=\bigcup_{z\in C}\ball(z,\eps)$ for every subset $C$ of $X$.

The Hausdorff distance 
$\haus$ is a metric on $\comp(\spa)$. Moreover, the topology of $(\comp(\spa),\haus)$ and $(\spa,\dis)$ are closely related, as illustrated in the following two propositions proved by~\citet{henrikson1999completeness}.

\begin{proposition}\label{prop:haus_limit}
Suppose that $(\spa,\dis)$ is complete. Then $(\comp(\spa),\haus)$ is also complete. Moreover, let $(a_n)_{n\in\N}$ be a Cauchy sequence in $(\comp(\spa),\haus)$, then $a_n\converge a$, where:
$$a=\left\{x \in \spa, \, \textrm{ there exists a sequence } (x_n)_{n\in\N}  \in \prod_n a_n  \textrm{ such that } x_n \converge x \right\}.$$
\end{proposition}

\begin{proposition}
If $(\spa,\dis)$ is compact then $(\comp(\spa),\haus)$ is also compact.
\end{proposition}

Hereafter, $(\man,\dis)$ is assumed to be a connected compact Riemannian manifold together with its geodesic distance. \change{In particular $(\comp(\man),\haus)$  is thus compact.} We also consider a Morse function $f\colon\man\rightarrow\R$. 
In this section, to ease notation, we will denote the Reeb graph $\Reeb_f(\man):=\man/\sim_{f}$ by $\reeb$. \change{In all of the following, we use the fact that an element $a\in\reeb$ is defined as an equivalence class of $\sim_{f}$, and as such is a subset of $\man$. In order to use the Hausdorff distance on $R$, we need to check that every element of $\reeb$ is compact, seen as a subset of $\man$.}
\begin{proposition}\label{prop:reeb_closed}For every $a\in\reeb$, $a$ is closed as a subset of $\man$.
\end{proposition}
\begin{proof}
Let $a\in\reeb$, then $a$ is defined as a connected component of $\preim{f}{\{v\}}$ for some $v$ in $f(\man)$. The set
$a$ is therefore closed in $\preim{f}{\{v\}}$ since it is a connected component, as connected components are maximal connected subsets and the closure of a connected subset is also connected. Moreover, $\preim{f}{\{v\}}$ is closed in $\man$ because $f$ is continuous. As such, $a$ is closed in a closed subset and consequently closed in $\man$.
\end{proof}

The manifold $\man$ being compact, Proposition~\ref{prop:reeb_closed} proves that $\reeb\subseteq\comp(\man)$. Thus, $(\reeb,\haus)$ is a metric space.

\subsection{Topology of \texorpdfstring{$(\reeb,\haus)$}{TEXT}}\label{sec:reeb_top}
In order to use Wasserstein and Gromov-Wasserstein distances, we require that $(\reeb,\haus)$ is complete and separable. However, completeness of $\reeb$ does not hold in general (see Figure \ref{fig:reeb_adherence}). To circumvent this, we consider its closure $\overline{\reeb}$, which is compact and as such complete. We show that $\reeb$ differs from $\overline{\reeb}$ only by a finite number of points. \change{The proofs of Propositions~\ref{prop:reeb_adherence}, \ref{prop:reeb_finite}, \ref{prop:reeb_sep} and Lemma~\ref{lem:haus_cont} are given in Appendix~\ref{apdx:reeb_geom_proofs}.}\\
Let us start with preliminary results. \change{We can first describe the adherent points to $\reeb$, i.e. the elements of $\overline{\reeb}$, as follows.}

\begin{proposition}\label{prop:reeb_adherence} Let $a\in\comp(\man)$ be an adherent point to $\reeb$. Then $a$ is connected and $f$ is constant on $a$.

\end{proposition}

The second point of Proposition~\ref{prop:reeb_adherence} proves that we can define a function $\Tilde{f}\colon \overline{\reeb}\rightarrow\R$ that associates to every adherent point of $\reeb$ the constant value that $f$ takes on it. Proposition~\ref{prop:reeb_adherence} as a whole might seem to indicate that $\overline{\reeb}=\reeb$ could also be proved, 
but this turns out to be false in general. The crucial property that the elements of $\overline{\reeb}\setminus\reeb$ do not satisfy 
is to be associated to {\it maximal} connected subsets. See Figure \ref{fig:reeb_adherence} for an example.

\begin{figure}[t]
    \centering
    \includegraphics[width=0.4\textwidth]{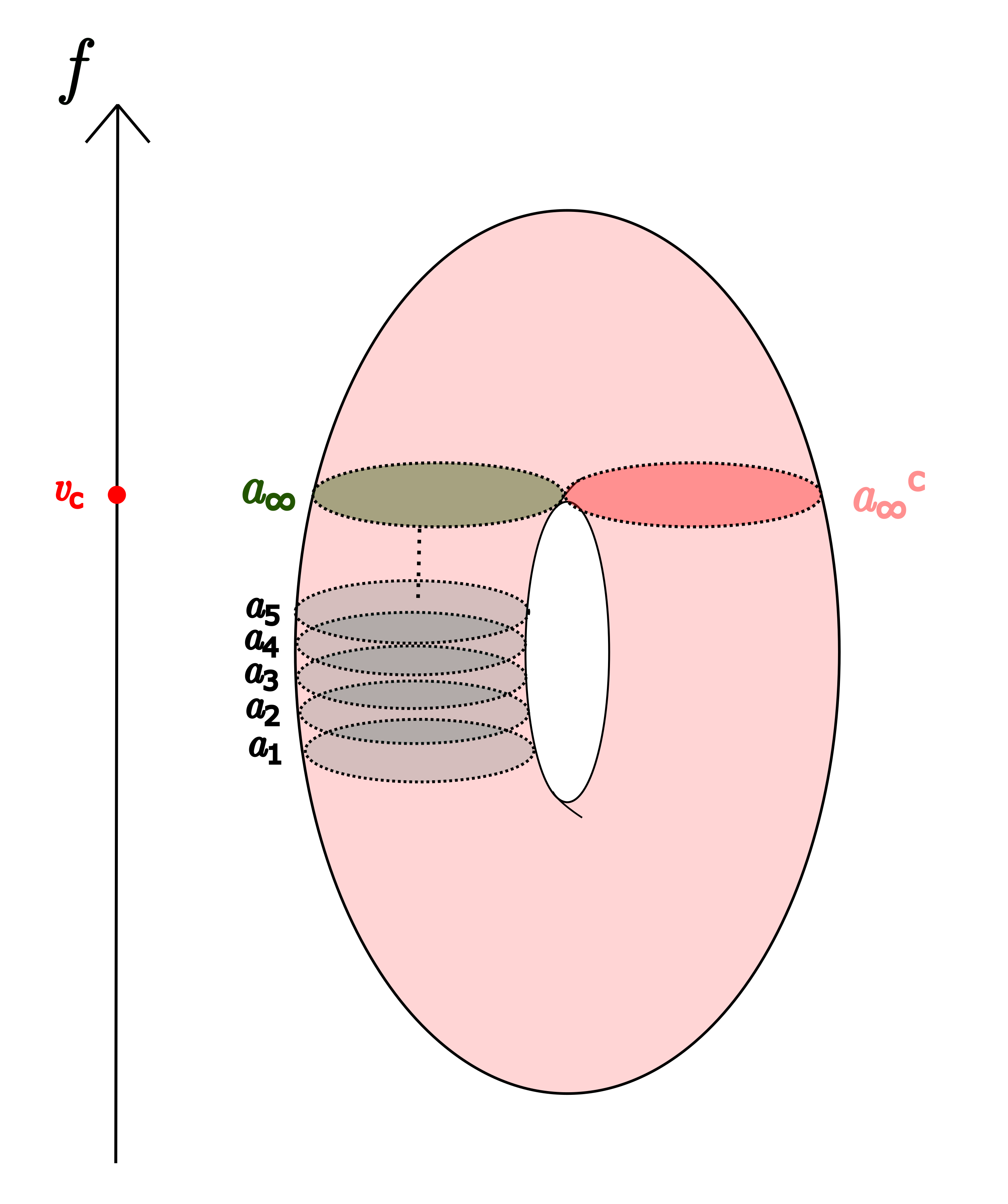}
    \caption{Example of a Reeb graph sequence $(a_n)_{n\in\N}$ that has a Hausdorff limit $a_\infty$ outside of the Reeb graph. Notice that this occurs when $a_\infty$ is a proper subset of the level set of a critical value $v_c$. In this example, $\preim{f}{\{v_c\}}$ has only one connected component. The complement $a_\infty^c$ of $a_\infty$ in this connected component is also represented here.}
    \label{fig:reeb_adherence}
\end{figure}

\change{However, the following result ensures that the elements of $\overline{\reeb}\setminus\reeb$ are associated to only critical values.}

\begin{lemma}\label{lem:haus_cont}
Let $(a_n)_{n\in\N}\in\reeb^{\N}$ and $(x_n)_{n\in\N}\in\man^{\N}$, such that $x_n\in a_n$ 
for every $n\in\N$ and such that $(x_n)_{n\in\N}$ admits a limit $x \in \man$. 
If $f(x)$ is not a critical value, then $a_n\converge [x]_{\sim_f}$.
\end{lemma}

\change{Our next result shows that $\overline{\reeb}$ and $\reeb$ differ only by a finite number of points, as stated above.}

\begin{proposition}\label{prop:reeb_finite}
$\overline{\reeb}\setminus\reeb$ is finite.
\end{proposition}

\change{ Having that $(\overline{\reeb},\haus)$ is a compact metric space, we now check that it is separable.}

\begin{proposition}\label{prop:reeb_sep}$(\overline{\reeb},\haus)$ is separable.
\end{proposition}

Consider the following map:
\begin{align*}
\pi\colon&\man\longrightarrow\overline{\reeb}\\
&x\longmapsto[x]_{\sim_f}.
\end{align*}
$\pi$ consists of the composition of the natural map associated to the quotient of $\man$ by the equivalence relation $\sim_f$, with the inclusion of $\reeb$ in its closure. \change{The following results imply (see Corollary~\ref{cor:pi_mes}) that $\pi$ can be used to push forward measures over $\man$.}

We introduce the notation $\mathcal{F}=\left\{x\in\man,\,f(x)\text{ is a critical value}\right\} $ for the set of points whose image under $f$ is a critical value. Notice that $\critic(f)$ is in general a proper subset of $\mathcal{F}$, as it is possible to have a non-critical point in a critical level set.

\begin{proposition}\label{prop:pi_cont}  For every measure $\mes$ on $(\man,\dis)$ such that $\mes(\mathcal{F})=0$, one has that
$\pi$ is continuous $\mes$-almost everywhere.
\end{proposition}

\begin{proof}
We will prove that the map $\pi$ is continuous at every point $x\in\man$ such that $f(x)$ is not a critical value. Let $x$ be such a point and let $(x_n)_{n\in\N}$ be a sequence in $\man$ such that $x_n\converge x$. By Lemma~\ref{lem:haus_cont}, one has $$\pi(x_n)=[x_n]_{\sim_f}\converge \pi(x)=[x]_{\sim_f}.$$
\end{proof}

\begin{lemma}\label{lem:zero_measure}
If a measure $\mes$ is absolutely continuous with respect to the volume measure, then $\mes(\mathcal{F})=0$.
\end{lemma}
\begin{proof}
The set $\critic(f)$ is finite and hence closed. As such, $\man'=\man\setminus\critic(f)$ is an open submanifold of $\man$. Now,
$$\mathcal{F}=\critic(f)\cup \left[\mathcal{F}\cap \man'\right],$$
and $\mathcal{F}\cap \man'$ is a finite union of $(d-1)-$dimensional submanifolds in $\man$ as the critical values of $f$ in $\man$ are no longer critical in $\man'$. Hence
$$\mes\left( \mathcal{F}\cap \man'\right)=0.$$
Finally, $\mes(\mathcal{F})=0$.
\end{proof}

\begin{corollary}\label{cor:pi_mes}
$\pi$ is Borel measurable.
\end{corollary}

\change{Following Proposition~\ref{prop:reeb_sep} and Corollary~\ref{cor:pi_mes}, we can now consider the metric measure space $(\overline{\reeb},\haus,\mes\circ\pi^{-1})$, for every Borel measure $\mes$ on $\man$.}

\subsection{Stability with respect to the change of filter function}

We show here that the Gromov-Wasserstein distance between Reeb graphs computed on a fixed manifold $\man$ is stable with respect to perturbations made to the filter function $f$, when $f$ is chosen within a certain class of Morse functions.\\
For a Morse function $f$ and an associated critical point $c\in\critic(f)$, we denote $\mu(c)$ as the square root of the spectral radius of $\G(c)$ the metric in the Morse chart described in Lemma~\ref{lem:morse} (see Appendix~\ref{apdx:elemRiem} for a definition).\\
For $\mub,\crib>0$, consider the following class of Morse functions having a uniform bound on the spectral radii of the metric inside a fixed choice of their Morse charts and on the number of their critical points:
$$\class=\left\{f \text{ Morse function},\max_{c\in\critic(f)}\mu(c)\leq\mub, \left|\critic(f)\right|\leq \crib \right\}.$$ 
For $f\in\class$, we denote $\pi_f\colon x\mapsto[x]_{\sim_f}$ and $\reeb_f$ the Reeb graph computed on $\man$ using $f$. \\


The proof of the following \change{theorem} is provided in Appendix~\ref{apdx:stab_proof}. We consider here the volume measure $\vol$ on $\man$. However for a measure $\mes$ that is absolutely continuous with respect to $\vol$ and has a bounded density, the proof follows closely.

\begin{theorem}\label{thm:stability}
Let $f,\ff\in\class$ and consider the two metric measure spaces $\left(\overline{\reeb}_f,\haus,\vol\circ\pi_f^{-1}\right)$ and $\left(\overline{\reeb}_\ff,\haus,\vol\circ\pi_\ff^{-1}\right)$. For $\inorm{f-\ff}$ small enough, we have:
$$\gromov_p^p(\overline{\reeb}_f,\overline{\reeb}_\ff)\leq \tau\cdot\inorm{f-\ff}^{\frac{p}{2}}+\chi\cdot \inorm{f-\ff}^{\frac{d}{d+1}},$$
where
$$\tau=(2\mub)^p,$$
and 
$$\chi=\diam(\man)^p\crib\left[12\mub^2\cdot\sup_{v\notin f(\critic(f))}\vol_{d-1}\left(\preim{f}{\{v\}}\right)+2\alpha_d\right],$$
and where $\vol_{d-1}\left(\preim{f}{\{v\}}\right)$ is the $(d-1)$-dimensional volume of $\preim{f}{\{v\}}$ seen as a submanifold of $\man$ and $\alpha_d$ is the volume of the unit ball in $\R^d$.
\end{theorem}

\begin{remark}
The result of \citet{liFlexibleProbabilisticTopology2025} on merge trees is similar to our Theorem~\ref{thm:stability}, even though we restrict to Reeb graphs, making our base objects quite different. Note also that, contrary to their upper bound (see Theorem 2 therein), our upper bound explicitly incorporates the intrinsic dimension of the underlying manifold, the number of critical points of the filter, as well as the spectral radius $\mub$ that describes the behaviour of the filter around the critical points, whereas their upper bound depends on the number of the vertices in the input simplicial complexes.
\end{remark}

A consequence of Theorem~\ref{thm:stability} is that the Gromov-Wasserstein distance between Reeb graphs is stable with respect to the change of filter function, in the following sense.

\begin{corollary}\label{cor:holder}
On $\class$ the map $f\mapsto\left(\overline{\reeb}_f,\haus,\vol\circ\pi_f^{-1}\right)$ is locally H\"{o}lder continuous. For every $f\in\class$, there exists $\epsilon,K>0$ such that for every $\ff\in\class$ such that $\inorm{f-\ff}\leq\epsilon$:
$$\gromov_p(\overline{\reeb}_f,\overline{\reeb}_\ff)\leq K\cdot\inorm{f-\ff}^\iota,$$

where $$\iota=\min\left(\frac{1}{2},\frac{d}{p(d+1)}\right).$$

\end{corollary}

 \label{sec:reeb_geom}
\subsection{\change{Mapper graphs as Metric Measure Spaces}}\label{sec:mapper_mms}

Let $\sample_n= \{x_1,\dots,x_n\}\subset \man$ be an $n$-sample taken inside $\man$. In this section, we aim at turning the Mapper $\mapper$ built from $\sample_n$, and constructed using a function $f\colon\,\man\rightarrow\R$ and some cover of its image, into a metric measure space. 

In this article, we restrict the focus to Mapper {\em graphs}, i.e., we only consider the $1$-skeleton of the Mapper complexes. 
An element $\simp$ of $\mapper$ is thus either a $0$-dimensional simplex (vertex) or a $1$-dimensional simplex (edge). See Section~\ref{subsec:reeb_mapper_def}. 

The Mapper graph $\mapper$ can be embedded in $\comp(\man)\cup\{\varnothing\}$ in the following way:

 \begin{itemize}
     \item A vertex $\simp_0$ associated to a cluster $\cluster_{j,k}$ is represented by $$\cluster_{j,k}\setminus \left(\bigcup_{\substack{j'\neq j \\k'\neq k}}\cluster_{j',k'}\right),$$
     \item An edge $\simp_1$ associated to two vertices, and hence to the corresponding two clusters $\cluster_{j,k}$ and $\cluster_{j',k'}$ is represented by $$\cluster_{j,k}\cap \cluster_{j',k'}\setminus\left(\bigcup_{\substack{j''\neq j,j' \\k''\neq k,k'}}\cluster_{j'',k''}\right).$$
 \end{itemize}
 
The elements of $\mapper$ belong to the space of all subsets of $\sample_n$, and are therefore compact subsets of $\spa$. It is possible for some simplex to have the empty set $\varnothing$ to be its representative, e.g., when a cluster is included in its intersection with another cluster. 
Moreover, we chose the representatives of the simplices in $\mapper$ so as to have an element $x_i$ of $\spa_n$ be in one and only one $\simp\in\mapper$. This allows to define a map
 $$\DD\colon\sample_n\longrightarrow\mapper,$$
 that associates $x_i$ to the unique simplex that contains it. 
Using the convention
 $ \haus\left(A,\varnothing\right)=\diam\left(\man\right) $ 
 for every non-empty $A\subseteq\man$, and  $\haus\left(\varnothing,\varnothing\right)=0$,  it is clear that $\left(\mapper,\haus\right)$ is a metric space.

 \medskip 
 
 Any measure on $\sample_n$, defined as a mixture of Dirac measures on $\man$ centered on each $x_i$, can be pushed forward onto $\mapper$ using $\DD$. \change{For example, denoting the empirical measure by $$\diracn=\frac{1}{n}\sum_{i=1}^n \dirac_{x_i},$$
for some points $x_i$ sampled according to a measure $\mes$ on $\man$, we have that $(\mapper,\haus,\diracn\circ\DD^{-1})$ is a metric measure space.}

\section{A bound on the Gromov-Wasserstein distance between the Reeb graph and the Mapper graph}\label{sec:mapper}

Let $\man$ be a compact connected Riemannian manifold of dimension $d$. Let  $f\colon\man\rightarrow\R$ be a Morse function and $\mes$ be a Borel probability measure on $\man$.

We make two assumptions:

 \begin{hypothesis} \label{hypRicci}
    The Ricci curvature $\ricci$ (see Appendix~\ref{apdx:elemRiem}) of $\man$ is lower bounded: $$\ricci\geq(d-1)\cdot k  $$
    where $k \in \R$.
\end{hypothesis}

\begin{hypothesis}\label{hypdensity}
	The measure $\mes$ is absolutely continuous with respect to the volume measure $\vol$ and admits an upper bounded density, i.e., $$\sup\frac{d\mes}{d\vol}<+\infty.$$ Furthermore, we assume that $\mes$ is fully supported on $\man$.  
\end{hypothesis}

Note that when $\man$ is a submanifold of an Euclidean space, Assumption~\ref{hypRicci} can be satisfied by assuming that the {\em reach} of $\man$, denoted by $\tau_{\man}$, is positive: $\tau_{\man}>0$ (see Proposition A.1. by~\citet{aamari2019estimating}). Assumptions on the reach of submanifolds are popular in literature (see, e.g., \citet{niyogi2008finding,burgisser2018computing}). 

Let $\sample_n=\{x_1,...,x_n\}$ be an $n$-sample taken in $\man$ with respect to $\mes$ and let $\diracn=\frac{1}{n}\sum_{i=1}^n \dirac_{x_i}$ be the associated empirical measure. In this section, we compare the Reeb graph $\Reeb:=\reeb_f(\man)$ and the Mapper $\mapper$ built on $\sample_n$ according to $\gromov$ metrics.

\subsection{Main result}

First, we specify the Mapper parameters that will be used in our results.
We will assume a {\em homogeneous covering of resolution} $r\in\N$, i.e., the intervals $\mathbb{I}=\{\interv_1,...,\interv_r\}$ that cover the range of the filter $f(\sample_n)$ in the Mapper are all of the same length and consecutive intervals have a percentage $0<g<1/2$ of their length in common. More specifically, we take $\interv_j=[a_j,b_j]$ where:
$$a_j=\min f+(j-1)\cdot\frac{\max f-\min f}{r}-\frac{g}{2-2 g}\cdot\frac{\max f - \min f}{r},$$
$$b_j=\min f+j\cdot\frac{\max f-\min f}{r}+\frac{g}{2-2 g}\cdot\frac{\max f - \min f}{r}.$$

Now, given a set of intervals $\mathbb{I}=\{\interv_1,...,\interv_r\}$ used to cover the range of the filter $f(\sample_n)$ in the Mapper, we define the {\em refinement} of $\mathbb{I}$ as:
$$\mathbb{J}=\left\{\interv_1\setminus\interv_2,\interv_1\cap\interv_2,\interv_2\setminus(\interv_1\cup\interv_3),\interv_2\cap\interv_3,\dots\right\}.$$

The family $\mathbb{J}$ is a refinement of $\mathbb{I}$ in the sense that it corresponds to a collection of elementary blocks of $\mathbb{I}$. Notice also that two elements of $\mathbb{J}$ intersect in at most one point. 

Let $\coverwidth$ denote the maximal length of an element in $\mathbb{J}$. We also call this the {\em maximal width} of the cover $\mathbb{I}$. It follows from the definition of $\interv_j=[a_j,b_j]$ above that
$$\coverwidth=\frac{\max f - \min f}{r}.$$ 
Similarly, the {\em minimal width} of an element in $\mathbb{J}$ is given by 
 $\frac{g}{1-g}\cdot\coverwidth.$ 

We now state the main theorem of this article. 
\change{Its proof is provided later in Section~\ref{subsec:proof_thm}.
The intermediate results and explicit constants are given in Proposition~\ref{prop:reeb_mapper}, Corollary~\ref{cor:estim} and Proposition~\ref{prop:apprx}, stated later in this section.}
\begin{theorem}\label{thm:main_result}
Under 
Assumptions~\ref{hypRicci} and~\ref{hypdensity}, and using the definitions above, 
consider the two metric measure spaces $(\overline{\reeb},\haus,\mes\circ\pi^{-1})$ and $(\mapper,\haus,\diracn\circ\DD^{-1})$ corresponding to the Reeb graph and to the Mapper graph, respectively.  
For any $\alpha>0$, $p \geq 1$ and a resolution $r(n)$ of the Mapper satisfying $r(n)\underset{n\rightarrow\infty}{\sim}n^{\frac{1}{d+\alpha}}$, we then have, for any large enough sample size $n$:
 $$\mathbb{E}\left(\gromov_p(\overline{\reeb},\mapper)\right)\lesssim n^{-\frac{\nu}{d+\alpha}} $$
 where $$\nu=\min\left\{\frac{1}{2},\frac{d}{p(d+1)}\right\}.$$ 
\end{theorem}

\begin{remark}We give a discussion on the optimality of this bound. As can be seen from the proofs of Proposition~\ref{prop:apprx} and Theorem~\ref{thm:main_result}, when $\frac{d}{p(d+1)}\leq\frac{1}{2}$, the upper bound given in Theorem~\ref{thm:main_result} is achieved by the bound on $\max_{\mathrm{J}\in\mathbb{J}}\vol\left( \preim{f}{\mathrm{J}}\right)$ (see Lemma~\ref{lem:meas}). Note moreover that this upper bound is sharp. To see this, consider for instance the case where $\man$ is the graph of the function $x\mapsto x^2$ over the segment $[-1,1]$: $\man=\{(x,x^2),\,x\in[-1,1]\}$, with the metric $\g$ being induced by the Euclidean metric in $\R^2$. The filter function $f$ here is the projection onto the second coordinate. We cover the image of $f$ with a homogeneous covering $\mathbb{I}$ like above and we are interested in $\vol\left( \preim{f}{\interv_1\setminus\interv_2}\right)$. 
We have:
\begin{align*}
\vol\left( \preim{f}{\interv_1\setminus\interv_2}\right)=\vol\left( \preim{f}{[0,a_2]}\right)&=\int_{-\sqrt{a_2}}^{\sqrt{a_2}}\sqrt{1+4t^2}\,dt\\
&=\frac{1}{4}\left[t\sqrt{1+t^2}+\log\left(t+\sqrt{1+t^2}\right)\right]^{2\sqrt{a_2}}_{-2\sqrt{a_2}}\\
&\underset{r\rightarrow\infty}{\sim}2\sqrt{\frac{2-3g}{2-2g}\coverwidth}=2\sqrt{\frac{2-3g}{2-2g}}\coverwidth^{\frac{d}{d+1}}.
\end{align*}
We therefore have that $\max_{\mathrm{J}\in\mathbb{J}}\vol\left( \preim{f}{\mathrm{J}}\right)\gtrsim\coverwidth^{\frac{d}{d+1}}$ for large enough $r$.

Note that in this example $\man$ has a boundary. \change{To circumvent this, we can for example make the same argument on the unit circle with an appropriate metric.}
\end{remark}

\subsection{Approximation-Estimation Decomposition}
\change{To prove Theorem~\ref{thm:main_result}, we begin by giving the following decomposition of the Gromov-Wasserstein distance between $\overline{\reeb}$ and $\mapper$.}
\begin{proposition}\label{prop:reeb_mapper}
 Consider the two metric measure spaces $(\overline{\reeb},\haus,\mes\circ\pi^{-1})$ and $(\mapper,\haus,\diracn\circ\DD^{-1})$ corresponding to the Reeb graph and to the Mapper graph, respectively. \change{We have:
 $$\gromov_p(\overline{\reeb},\mapper)\leq \estim_n+\appr_n,$$
where $\estim_n$ corresponds to the estimation term, defined as
 $$\estim_n=\wasser_p(\mes\circ\pi^{-1},\diracn\circ\pi^{-1}),$$
whereas $\appr_n$ can be interpreted as an approximation term, defined as $$\appr_n=\left(\frac{1}{n}\sum_{i=1}^n\haus([x_i]_{\sim_f},\DD(x_i))^p\right)^{\frac{1}{p}}.$$
 }
 
 \end{proposition}
 
 \begin{proof}
 Consider the metric measure space $\overline{\reeb}_n=(\overline{\reeb},\haus,\diracn\circ\pi^{-1})$ obtained with the empirical measure on the Reeb graph.
 
 On one hand,
$$\gromov_p(\overline{\reeb},\overline{\reeb}_n)\leq\wasser_p(\mes\circ\pi^{-1},\diracn\circ\pi^{-1}).$$
On the other hand, 
$$\gromov_p^p(\overline{\reeb}_n,\mapper)\leq \inf_{\eta}\int_{\overline{\reeb}\times\mapper}\haus(a,\simp)^p\,d\eta(a,\simp),$$
 where the infimum is taken over all coupling measures $\eta$ on $\overline{\reeb}\times\mapper$ that have marginals $\diracn\circ\pi^{-1}$ and $\diracn\circ\DD^{-1}$. Indeed, the Hausdorff distance $\haus$ can be defined on $\overline{\reeb}\sqcup\mapper$ via its inclusion in $\comp(\man)$, and is thus a metric coupling of $\haus$ restricted to $\overline{\reeb}$ and $\haus$ restricted to $\mapper$. Furthermore, consider the map
 \begin{align*}
     \phi\colon\sample_n&\longrightarrow\overline{\reeb}\times\mapper\\
     x_i&\longmapsto([x_i]_{\sim_f},\DD(x_i))
 \end{align*}
 It is clear that $\diracn\circ\phi^{-1}$ is an example of such a coupling measure $\eta$. Hence,
 \begin{align*}
    \gromov_p^p(\overline{\reeb}_n,\mapper)&\leq \int_{\overline{\reeb}\times\mapper}\haus(a,\simp)^p\,d\diracn\circ\phi^{-1}(a,\simp) \\
    &=\int_{\sample_n}\haus([x]_{\sim_f},\DD(x))^p\,d\diracn(x)\\
    &=\frac{1}{n}\sum_{i=1}^n\haus([x_i]_{\sim_f},\DD(x_i))^p.
 \end{align*}
Finally,
  \begin{align*}
    \gromov_p(\overline{\reeb},\mapper)&\leq \gromov_p(\overline{\reeb},\overline{\reeb}_n)+\gromov_p(\overline{\reeb}_n,\mapper)\\
    &\leq\wasser_p(\mes\circ\pi^{-1},\diracn\circ\pi^{-1})+\left(\frac{1}{n}\sum_{i=1}^n\haus([x_i]_{\sim_f},\DD(x_i))^p\right)^{\frac{1}{p}}.
 \end{align*}
 \end{proof}
  The bound given in Proposition~\ref{prop:reeb_mapper} contains two terms. The term $\estim_n$ can be interpreted as an estimation error. It measures how representative the discrete measure associated to the $n$-sample $\sample_n$ is of the measure $\mes$, in terms of the distance in the Reeb graph. The term $\appr_n$ can be interpreted as an approximation error. It measures how well the Mapper graph captures the same stratification than the one in the Reeb graph. It is small if, for most of the sample points $x_i$, the simplex containing $x_i$ is close (in Hausdorff distance) from the Reeb graph element that contains $x_i$.

 \subsection{The estimation error \texorpdfstring{$\estim_n$}{TEXT} }
 
 \subsubsection{Covering numbers for Reeb graphs}
 
 Following the works of \citet{BoissardLeGouic2014} and \citet{weed2019sharp}, in this section, we aim at bounding the estimation error $\estim_n$ using the $\eps$-covering number of $(\reeb,\haus)$.
 
 \begin{definition}
 Given a metric space $\spa$, $S\subseteq\spa$ and $\eps>0$, the $\eps$-covering number of $S$, $\covn_\eps(S)$, is the minimum integer $m$ such that there exist balls $B_1,\dots,B_m$ in $\spa$ of radius $\eps$ such that $S\subseteq\bigcup_{i=1}^m B_i$.
 \end{definition}
 
 Note that we have $\covn_\eps(\reeb)=\covn_\eps(\bar \reeb)$, since $\reeb\subseteq\bar\reeb$ and any $\eps$-covering of $\reeb$ is also an $\eps$-covering of $\bar\reeb$, as $\bar\reeb$ is the minimal closed subset that contains $\reeb$.
 
 \begin{proposition}\label{prop:covn}
There exists $\eps'>0$ such that for every $\eps\leq\eps'$, we have
  $$\covn_\eps(\reeb)\leq N_c+\left(\frac{\beta}{\eps}\right)^{2d},$$
  where $N_c$ is the number of elements in $\reeb$ associated to critical values, $d$ is the dimension of $\man$ and $\beta$ is a constant depending on $f$ and $\man$. 
 \end{proposition}
 
  \begin{proof}
 First, let $\eps>0$ and  take $N_c$ balls $\ball_1,\dots,\ball_{N_c}\subseteq \reeb$ of radius $\eps$, one around each element of $\reeb$ associated to a critical value. 

 Second, let $\{v_1,\dots,v_{k+1}\}$ be the critical values of $f$, and let
 $$\man_\eps :=\man\setminus \bigcup_{c \in \critic(f)}\ball\left(c,\frac{\eps}{4\left|\critic(f)\right|}\right).$$
 For $1\leq j\leq k$, let the sets $\{S_{j,l}\}_{j,l}$ 
 denote the connected components of $\man\cap \preim{f}{(v_j,v_{j+1})}$. Notice that there are no critical points in any of the $S_{j,l}$.
 
 We start with the next technical result that connect balls of $(\man,d)$ with those of $(\reeb,\haus)$.
 
 \begin{assertion*}
 For every $\eps>0$, every $(j,l)$ and every $x\in S_{j,l}$:
 $$\ball\left(x,\frac{\eps\cdot \inf_{p\in\man_\eps}\Vert\nabla f(p)\Vert}{2\sup_{p\in\man}\Vert\nabla f(p)\Vert}\right)\cap S_{j,l}\subseteq \preim{\pi}{\ball\left([x]_{\sim_f},\eps\right)}.$$
 \end{assertion*}
 
 \begin{proof}
 Let $y\in \ball\left(x,\frac{\eps\cdot \inf_{p\in\man_\eps}\Vert\nabla f(p)\Vert}{2\sup_{p\in\man}\Vert\nabla f(p)\Vert}\right)\cap S_{j,l}$. W.l.o.g., we assume that $f(y)>f(x)$.
 
From Lemma~\ref{lem:flow}, we can introduce the gradient flow $(\psi_t)_t$ 
of $f$, such that  $\psi_{f(y)-f(x)}$  is a diffeomorphism between $\preim{f}{\{f(x)\}}$ and $\preim{f}{\{f(y)\}}$. Moreover, since $x$ and $y$ are both in $S_{j,l}$, then $\psi_{f(y)-f(x)}$ is in particular a diffeomorphism between $[x]_{\sim_f}$ and $[y]_{\sim_f}$. Thus, for any point $p$ in $[x]_{\sim_f}$, there exists a point $q$ in $[y]_{\sim_f}$ (and conversely) such that the curve
 \begin{align*}
 \gamma\colon[0,f(y)-f(x)]&\longrightarrow \man\\
 u&\longmapsto\psi_{u+f(x)}(p)
 \end{align*}
connects $p = \gamma(0)$ to $q := \gamma(f(y)-f(x))$.

Now, for every $c\in\critic(f)$, if $\gamma$ enters $\ball\left(c,\frac{\eps}{4\left|\critic(f)\right|}\right)$ consider 
$$u^c_0=\inf\left\{u\in[0,f(y)-f(x)],\,\gamma(u)\in \ball\left(c,\frac{\eps}{4\left|\critic(f)\right|}\right)\right\},$$
and 
$$u^c_1=\sup\left\{u\in[0,f(y)-f(x)],\,\gamma(u)\in \ball\left(c,\frac{\eps}{4\left|\critic(f)\right|}\right)\right\}.$$
Otherwise take $u^c_0=u^c_1=0$.
Notice that
$$\dis(\gamma(u^c_0),\gamma(u^c_1))\leq \frac{\eps}{2\left|\critic(f)\right|}.$$
Denoting
 $$\mathcal{U}=[0,f(y)-f(x)]\setminus\bigcup_{c \in \critic(f)}[u^c_0,u^c_1],$$
 we have 
 $$\gamma\left(\mathcal{U}\right)\in\man_\eps$$
 Furthermore:
\begin{align*}
\dis(p,q)&\leq\int_{u\in \mathcal{U}}\left\Vert\frac{d\psi_{u+f(x)}(p)}{du}\right\Vert\,du + \sum_{c \in \critic(f)} \dis(\gamma(u^c_0),\gamma(u^c_1))\\
&\leq\int_{u\in \mathcal{U}}\frac{du}{\Vert\nabla f(\psi_{u+f(x)}(p))\Vert} + \sum_{c \in \critic(f)} \dis(\gamma(u^c_0),\gamma(u^c_1))\\
&\leq\frac{f(y)-f(x)}{\inf_{p\in\man_\eps}\Vert\nabla f(p)\Vert}+ \sum_{c \in \critic(f)} \frac{\eps}{2\left|\critic(f)\right|}\\
&\leq\frac{\dis(x,y)\cdot \sup_{p\in\man}\Vert\nabla f(p)\Vert}{\inf_{p\in\man_\eps}\Vert\nabla f(p)\Vert}+\frac{\eps}{2}\\
&\leq\eps
\end{align*}
 As such, $\haus([x]_{\sim_f},[y]_{\sim_f})\leq\eps$.
 \end{proof}
 
 Using the previous assertion, we will bound the $\eps$-covering number of $\reeb$.
 
Let $e=\frac{\eps\cdot \inf_{p\in\man_\eps}\Vert\nabla f(p)\Vert}{2\sup_{p\in\man}\Vert\nabla f(p)\Vert}$ and consider minimal $e$-coverings of each $S_{j,l}$.
Now, for every $x$ such that $f(x)$ is not a critical value, $x$ is in one of the $S_{j,l}$. As such, it is in one of the balls of radius $e$ that cover $S_{j,l}$, centered, say, on $x_i$. Accordingly, by the previous assertion $[x]_{\sim_f}\in\ball\left([x_i]_{\sim_f},\eps\right)$.

We proved that

$$\covn_\eps\left(\reeb\right)\leq N_c+\sum_{j,l}\covn_e\left(S_{j,l}\right).$$
 
 Fix one of the $S_{j,l}$. We now bound $\covn_e\left(S_{j,l}\right)$. If $\{\ball(x_i,e)\}_i$ is a minimal covering, the balls $\{\ball(x_i,e/2)\}_i$ are disjoint. Hence, 
 $$\vol\left(\bigcup_{i}\ball(x_i,e/2)\right)=\sum_{i}\vol\left(\ball(x_i,e/2)\right).$$
Also, 
$$\bigcup_{i}\ball(x_i,e/2)\subseteq \man,$$

and as such
$$\sum_{i}\frac{\vol\left(\ball(x_i,e/2)\right)}{\vol\left(\man\right)}\leq 1.$$
We now use Theorem~\ref{thm:bishop_gromov} since we assumed that the Ricci curvature of $\man$ satisfies $\ricci\geq(d-1)\cdot k$. We have:

$$\frac{\vol\left(\ball(x_i,e/2)\right)}{v(d,k,e/2)}\geq\frac{\vol\left(\man\right)}{v(d,k,\diam(\man))},$$
and as such

$$\sum_i \frac{v(d,k,e/2)}{v(d,k,\diam(\man))}\leq\sum_{i}\frac{\vol\left(\ball(x_i,e/2)\right)}{\vol\left(\man\right)}.$$
Hence,

$$\covn_e\left(S_{j,l}\right)\leq\frac{v(d,k,\diam(\man))}{v(d,k,e/2)}.$$

Looking at Proposition~\ref{prop:geo_vol}, we know that 
$$v(d,k,e/2)\underset{e\rightarrow0}{\sim}\alpha_d\cdot\frac{e^d}{2^d},$$
where $\alpha_d$ is the volume of the unit ball in $\mathbb{R}^d$.
For every small enough $e$, and hence small enough $\epsilon$, we have
$$v(d,k,e/2)\geq \frac{\alpha_d}{2}\cdot \frac{e^d}{2^d}.$$

We conclude that
$$\covn_e\left(S_{j,l}\right)\leq \frac{2^{d+1}\cdot v(d,k,\diam(\man))}{\alpha_d\cdot e^d}.$$
Accordingly
$$\covn_\epsilon(S)\leq N_c+\left(\frac{\Tilde{\beta}}{\eps\cdot\Gamma(\eps)}\right)^d,$$
where:
\begin{align*}
    \Tilde{\beta}&=\left(\frac{2\left|\{S_{j,l}\}_{j,l}\right|\cdot v(d,k,\diam(\man))}{\alpha_d}\right)^{\frac{1}{d}}\cdot4\sup_{p\in\man}\Vert\nabla f(p)\Vert,\\
    \Gamma(\eps)&=\inf_{p\in\man_\eps}\Vert\nabla f(p)\Vert.
\end{align*}

It only remains now to investigate the behavior of $\Gamma(\eps)$ when $\eps\rightarrow 0$.

For $\eps$ small enough, $\inf_{p\in\man_\eps}\Vert\nabla f(p)\Vert$ is reached in a neighborhood of a critical point $c$, at a point $p$ such that
$$\dis(p,c)\geq\frac{\eps}{4\left|\critic(f)\right|}.$$

 Now, writing $f$ inside a Morse coordinate neighborhood $U\subseteq\man$ centered at $c$ (see Lemma~\ref{lem:morse}) gives:\\
 $$\sqrt{\sum_{i=1}^d\frac{\partial f(p)^2}{\partial x_i}}=2\sqrt{\sum_{i=1}^d x_i^2},$$
 
 where $x(p)=(x_1,...,x_d)$ is the representation of the point $p\in U$ in local coordinates.\\
Proposition~\ref{prop:grad} gives:
$$\Vert \nabla f(p) \Vert\geq \frac{1}{\mu(p)}\cdot \sqrt{\sum_{i=1}^d\frac{\partial f(p)^2}{\partial x_i}},$$

with $\mu(p)=\max\{\sqrt{\rho},\, \rho\in\mathrm{Sp(G)}\}$ where $\mathrm{Sp(G)}$ is the spectrum of the metric tensor $\G(p)=\left[\g(\partial_i,\partial_j)\right]_{i,j}$. The function $\mu$ is continuous and $\mu(p)\rightarrow\mu(c)<\infty$ as $p\rightarrow c$. Taking $\mu_0=\sup \mu(p)$ in a small neighborhood of $c$, we have $\mu_0<\infty$ and 
$$\Vert \nabla f(p) \Vert\geq \frac{1}{\mu_0}\cdot \sqrt{\sum_{i=1}^d\frac{\partial f(p)^2}{\partial x_i}},$$
 
Moreover, Proposition~\ref{prop:dis_equiv} shows that:
$$\dis_0(p,c)=\sqrt{\sum_{i=1}^d x_i^2}\geq \frac{1}{\mu_0} \cdot \dis(p,c),$$
for small enough $\eps$. Hence,
$$\Vert \nabla f(p) \Vert \geq \frac{2}{\mu_0^2}\cdot \frac{\eps}{4\left|\critic(f)\right|}.$$

Finally,
$$\covn_\eps(\reeb)\leq N_c+\left(\frac{\beta}{\eps}\right)^{2d},$$
where 
$$\beta=\left(2\Tilde{\beta}\cdot \mu_0^2\cdot\left|\critic(f)\right| \right)^{\frac{1}{2}}$$

 \end{proof}
 
\subsubsection{Convergence rate for the estimation error} 
We first restate Proposition 5 from~\citet{weed2019sharp}.
\begin{proposition}\label{prop:weed}
Let $(\spa,d)$ be a Polish metric space with $\diam(\spa)\leq 1$, $\mes$ be a probability measure on $\spa$ and $\diracn$ be an empirical measure associated to a random sample $\sample_n$ taken from $\mes$. \\
Let $p\in[1,\infty)$. Suppose there exists $\eps'>0$ and $s>2p$ such that for all $\eps\leq\eps'$
$$-\frac{\log \covn_{\eps,s}}{\log\eps}\leq s,$$
 where $$\covn_{\eps,s}=\inf\left\{\covn_\eps(B)\, ,\,\mes(B)\geq 1-\eps^{\frac{sp}{s-2p}}\right\}.$$
 Then,
 $$\wasser_p^p(\mes,\diracn)\leq C_1'n^{-\frac{p}{s}}+C_2'n^{-\frac{1}{2}},$$
 where 
 $$C_1'=3^{\frac{3sp}{s-2p}+1}\left(\frac{1}{3^{\frac{s}{2}-p}-1}+3\right),$$
 and
 $$C_2'=\left(27/\eps'\right)^{\frac{s}{2}}.$$

\end{proposition}

We now apply Proposition~\ref{prop:weed} to bound the estimation error $\estim_n$, using our asymptotic control over the covering number of $\reeb$ in Proposition~\ref{prop:covn}.

 \begin{corollary}\label{cor:estim}
For every $s>\max(2d,2p)$, we have
 $$\mathbb{E}\left(\estim_n^p\right)\leq C_1n^{-\frac{p}{s}}+C_2n^{-\frac{1}{2}},$$
 where 
 $$C_1=\diam\left(\man\right)^p\cdot3^{\frac{3sp}{s-2p}+1}\left(\frac{1}{3^{\frac{s}{2}-p}-1}+3\right),$$
 $$C_2=\diam\left(\man\right)^p\cdot\left(27/\eps'\right)^{\frac{s}{2}},$$
 and $\eps'>0$.
 \end{corollary}
 \begin{proof}
 We have $\estim_n=\wasser_p(\mes\circ\pi^{-1},\diracn\circ\pi^{-1})$. Moreover, $\diracn\circ\pi^{-1}$ is an empirical measure associated to $\mes\circ\pi^{-1}$, since for every Borel subset $A$ of $\overline{\reeb}$:
 \begin{align*}
     \diracn\circ\pi^{-1}(A)&=\frac{1}{n}\sum_{i=1}^n\ind_{x_i\in\pi^{-1}(A)} \\
     &=\frac{1}{n}\sum_{i=1}^n\ind_{[x_i]_{\sim_f}\in A},
 \end{align*}
 and 
 $$[x_i]_{\sim_f}\overset{\mathrm{iid}}{\sim}\mes\circ\pi^{-1},$$
 because 
  \begin{align*}
     \mathbb{P}([x_i]_{\sim_f}\in A)&=\mathbb{P}(x_i\in \pi^{-1}(A)) \\
     &=\mes\circ\pi^{-1}(A),
 \end{align*}
 where we used that $\diracn$ is an empirical measure associated to $\mes$.
 
In~\citet{weed2019sharp}, the metric spaces $\spa$ are assumed to verify $\diam(\spa)\leq 1$. This can be achieved in our case by considering a normalized version of $\haus$ in $\bar \reeb$, defined as $\haus/\diam(\man)$. Notice that covering numbers $\covn'_\eps(\cdot)$ of this normalized space satisfy
$$\covn'_\eps(\cdot)=\covn_{(\diam(\man)\cdot\eps)}(\cdot).$$
Moreover, the Wasserstein distance $\wasser'_p$ in the normalized space satisfies:
$$\wasser'_p(\cdot,\cdot)=\frac{1}{\diam(\man)}\wasser_p(\cdot,\cdot).$$
Now, we can use Proposition 5. in~\citet{weed2019sharp} to prove our result. The only assumption made in this proposition is that, for any small enough $\eps$ and $s>2p$: 
 $$-\frac{\log \covn'_{\eps,s}}{\log\eps}\leq s,$$
 where $$\covn'_{\eps,s}=\inf\left\{\covn'_\eps(B)\, ,\,\mes\circ\pi^{-1}(B)\geq 1-\eps^{\frac{sp}{s-2p}}\right\}.$$
 We see that:
 $$\covn'_{\eps,s}\leq\covn'_\eps(\reeb)=\covn_{(\diam(\man)\cdot\eps)}(\reeb).$$
 As such, using Proposition~\ref{prop:covn}, for any small enough $\eps$, we have
 $$-\frac{\log \covn'_{\eps,s}}{\log\eps}\leq-\frac{\log\left( N_c+\left(\frac{\beta}{\diam(\man)\cdot\eps}\right)^{2d}\right)}{\log\eps}\underset{\eps\rightarrow0}{\longrightarrow}2d.$$
 Hence, for every $s>\max(2d,2p)$ there exists $\eps'>0$ such that for every $\eps\leq\eps'$:
 $$-\frac{\log \covn'_{\eps,s}}{\log\eps}\leq s.$$
 Applying Proposition~\ref{prop:weed} gives
 $$\wasser_p^{'^p}(\mes\circ\pi^{-1},\diracn\circ\pi^{-1})\leq C_1'n^{-\frac{p}{s}}+C_2'n^{-\frac{1}{2}},$$
 where 
 $$C_1'=3^{\frac{3sp}{s-2p}+1}\left(\frac{1}{3^{\frac{s}{2}-p}-1}+3\right),$$
 and
 $$C_2'=\left(27/\eps'\right)^{\frac{s}{2}}.$$
 As such,
 $$\wasser_p^p(\mes\circ\pi^{-1},\diracn\circ\pi^{-1})\leq \diam\left(\man\right)^p\cdot C_1'n^{-\frac{p}{s}}+\diam\left(\man\right)^p\cdot C_2'n^{-\frac{1}{2}},$$
 which concludes the proof.
  \end{proof}

Note that the constants $C_1$ and $C_2$ in Corollary~\ref{cor:estim} are, up to a constant, the same as in Proposition 5 of~\citet{weed2019sharp}, and $\eps'$ in the expression of $C_2$ is the same as in the proof of Corollary~\ref{cor:estim}.  

 \subsection{The approximation error \texorpdfstring{$\appr_n$}{TEXT}}
 \subsubsection{Comparing Reeb graph and Mapper graph elements}
We focus here on the second term in the upper bound of Proposition~\ref{prop:reeb_mapper}:
$$\appr_n=\left(\frac{1}{n}\sum_{i=1}^n\haus([x_i]_{\sim_f},\DD(x_i))^p\right)^{\frac{1}{p}}.$$
Recall that given a cover $\mathbb{I}=\{\interv_1,\dots,\interv_r\}$ of the range of the filter $f(\sample_n)$, the {\em refinement} of $\mathbb{I}$ is defined as:
$$\mathbb{J}=\left\{\interv_1\setminus\interv_2,\interv_1\cap\interv_2,\interv_2\setminus(\interv_1\cup\interv_3),\interv_2\cap\interv_3,\dots\right\}.$$
For every critical point $c\in\critic(f)$, we know that there exists a Morse chart $\varphi\colon U\subseteq\man\rightarrow\R^d$ around $c$ by Lemma~\ref{lem:morse}. Recall that the Riemannian metric $\g$ can be represented in this chart as a symmetric positive-definite matrix $\G(c)$.

\begin{lemma}\label{lem:delta}
Let $x_i\in\sample_n$ and $\mathrm{J}\in\mathbb{J}$ such that $\mathrm{J}$ does not contain critical values and is not adjacent to another refined cover element $\mathrm{J}'$ that contains critical values. For any large enough resolution $r$, if $x_i\in\mathrm{J}$, then for every $x_j\in\DD(x_i)$:
$$\dis(x_j,[x_i]_{\sim_f})\leq\max_{c\in\critic(f)}\mu(c)\cdot\sqrt{\frac{1-g}{g}}\cdot\sqrt{\coverwidth}$$
where $\mu(c)$ is the square root of the spectral radius of the metric tensor $\G(c)$, see Appendix~\ref{apdx:elemRiem}.
\end{lemma}

\begin{proof}
First, notice that if $\mathrm{J}$ does not correspond to an intersection of two intervals, i.e., it is of the form $\mathrm{J}=\interv_j\setminus(\interv_{j-1}\cup\interv_{j+1})$, then $\DD(x_i)$ is a subset of the connected component of $x_i$ in $\preim{f}{\mathrm{J}}$.

Second, when $\mathrm{J}$ is an intersection, of the form $\mathrm{J}=\interv_j\cap\interv_{j+1}$, our assumption that $\mathrm{J}$ does not contain critical values and is not adjacent to another refined cover element $\mathrm{J}'$ that contains critical values implies that $\preim{f}{\interv_j\cup\interv_{j+1}}$ is homeomorphic to a finite collection of cylinders whose heights are given by the function values (see Theorem~\ref{thm:morse}). As such, the connected components of  $\preim{f}{\mathrm{J}}$ are exactly the intersections of the connected components of $\preim{f}{\interv_j}$ and $\preim{f}{\interv_{j+1}}$. Therefore, $\DD(x_i)$ is also a subset of the connected component of $x_i$ in $\preim{f}{\mathrm{J}}$ in this case.

Let $x_j\in\DD(x_i)$. Then  $x_j$ is in the same connected component of $\preim{f}{\mathrm{J}}$ than $x_i$, as demonstrated above. Since $\mathrm{J}$ does not contain critical values, we can therefore send $x_j$ to a point $\psi_{f(x_i)-f(x_j)}(x_j)$ in $[x_i]_{\sim_f}$ using the gradient flow of $f$. See the proof of Theorem~\ref{thm:morse} for more details. We can therefore write:

\begin{align*}
\dis(x_j,\psi_{f(x_i)-f(x_j)}(x_j))&\leq\left|\int_{0}^{f(x_i)-f(x_j)}\left\Vert\frac{d\psi_{u}(p)}{du}\right\Vert\,du \right| \\
&\leq \left|\int_{0}^{f(x_i)-f(x_j)}\frac{1}{\Vert\nabla f(\psi_{u}(p))\Vert}\,du \right|\\
&\leq\frac{|f(x_i)-f(x_j)|}{\inf_{p\in \preim{f}{\mathrm{J}}}\Vert\nabla f(p)\Vert}\\
&\leq\frac{\coverwidth}{\inf_{p\in \preim{f}{\mathrm{J}}}\Vert\nabla f(p)\Vert}.
\end{align*}

Now, if $\preim{f}{\mathrm{J}}$ is sufficiently close to a critical point $c$, in the sense that it intersects a Morse coordinate neighborhood $U\subseteq\man$ around $c$ in a non-empty region, we can use Lemma~\ref{lem:morse} to write:
 $$f(p)=f(c)-\sum_{j=1}^ix_j^2+\sum_{j=i+1}^dx_j^2,$$
where $x(p)=(x_1,\dots,x_d)$ is the representation of $p\in U$ in local coordinates. Proposition~\ref{prop:grad} therefore gives:
$$\Vert\nabla f(p)\Vert\geq\frac{2\Vert x(p)\Vert_0}{\mu(p)},$$
where $\Vert\cdot\Vert_0$ is the Euclidean norm in local coordinates. 
However, these local coordinates also induce:
$$|f(p)-f(c)|\leq \Vert x(p)\Vert_0^2,$$
and therefore 
$$\Vert\nabla f(p)\Vert\geq\frac{2\sqrt{|f(p)-f(c)|}}{\mu(p)}.$$ 
Since $\mathrm{J}$ does not contain critical values and is not adjacent to another refined cover element $\mathrm{J}'$ that contains critical values, this means that in $U\cap \preim{f}{\mathrm{J}}$, we have 
$$|f(p)-f(c)|\geq \frac{g}{1-g}\cdot\coverwidth,$$
the latter being the minimal width of a refined cover element. 
As such:
$$\frac{1}{\inf_{p\in U\cap \preim{f}{\mathrm{J}}}\Vert\nabla f(p)\Vert}\leq \frac{\sup_{p\in U}\mu(p)}{2}\cdot\sqrt{\frac{1-g}{g}}\cdot\frac{1}{\sqrt{\coverwidth}}.$$
We see that $\sup_{p\in U}\mu(p)\rightarrow\mu(c)$ as $p\rightarrow c$. Hence, upon shrinking $U$ further if necessary, we can guarantee that
$$\sup_{p\in U}\mu(p)\leq 2\mu(c).$$
Notice that the Morse charts around critical points (that we used so far in our proof) are fixed beforehand and do not depend on the resolution $r$. Moreover, away from critical points, $\Vert\nabla f(p)\Vert$ is lower bounded by a constant $C>0$ depending only on $f$ and $\man$. 

Therefore,
$$\dis(x_j,\psi_{f(x_i)-f(x_j)}(x_j))\leq \max\left(\frac{1}{C},\max_{c\in\critic(f)}\mu(c)\cdot\sqrt{\frac{1-g}{g}}\cdot\frac{1}{\sqrt{\coverwidth}}\right)\cdot\coverwidth.$$
Finally, since $\coverwidth\underset{r\rightarrow\infty}{\longrightarrow}0$, we proved our result.
\end{proof}

We now define the modulus of continuity associated to $f$.
\begin{definition}[\change{Modulus of continuity}]
Let $g\colon\man\rightarrow\R$ be a uniformly continuous function. We call modulus of continuity of $g$ the function:
 \begin{align*}
     \omega_g\colon\R^+&\longrightarrow\R^+\\
     \lambda&\longmapsto \sup_{\dis(x,y)\leq\lambda}|g(x)-g(y)| .
 \end{align*}
\end{definition}

In our case $f$ is a Morse function and hence continuously differentiable. This means that the modulus of continuity $\omega_f$ of $f$ is well defined.

\begin{lemma}\label{lem:slice}
Let $\lambda>0$ such that $\haus(\sample_n,\man)\leq \lambda$ and
$$\omega_f\left(\lambda\right)\leq\frac{g}{1-g} \frac{\coverwidth}{4}.$$
Let $x_i\in\sample_n$ and $\mathrm{J}\in\mathbb{J}$ such that $\mathrm{J}$ does not contain critical values and is not adjacent to another refined cover element $\mathrm{J}'$ that contains critical values. For any large enough resolution $r$, if $x_i\in\mathrm{J}$ then for every $y\in[x_i]_{\sim_f}$:
$$\dis(y,\DD(x_i))\leq \lambda+\frac{1}{2}\cdot\max_{c\in\critic(f)}\mu(c)\cdot\sqrt{\frac{1-g}{g}}\cdot\sqrt{\coverwidth}.$$
\end{lemma}

\begin{proof}
Let $\mathrm{J}:=[v,w]$ and let $S$ be the connected component of $x_i$ inside $\preim{f}{\mathrm{J}}$. 
We assumed that $\mathrm{J}$ does not contain critical values. As such, $f\left(S\right)=\mathrm{J}$, otherwise $f$ would reach either a local maximum or a local minimum in $S$. 
We will denote $\Tilde{S_\lambda}=S\setminus (\preim{f}{\{v\}}\cup \preim{f}{\{w\}})^{\lambda}$, i.e., $\Tilde{S_\lambda}$ is comprised of the points of $S$ that 
are at a distance at least $\lambda$ from $\preim{f}{\{v\}}$ or $\preim{f}{\{w\}}$.

The assumption $\omega_f\left(\lambda\right)\leq\frac{g}{1-g}\cdot\frac{\coverwidth}{4}$ guarantees that $\Tilde{S}_\lambda$ is not empty, as the set
$$S\cap \preim{f}{ \left(v+\frac{g}{1-g}\cdot\frac{\coverwidth}{4},w-\frac{g}{1-g}\cdot\frac{\coverwidth}{4}\right)}$$
is included in $\Tilde{S}_\lambda$, and is non-empty  
because $f\left(S\right)=\mathrm{J}$, like mentioned above. 

We can also immediately see that $\haus(\sample_n\cap S,\Tilde{S}_\lambda)\leq\lambda$, since we know that $\haus(\sample_n,\man)\leq \lambda$.
Therefore, for every $p\in\Tilde{S}_\lambda$, there exists $x_j$ in $\sample_n\cap S$ and hence in $\DD(x_i)$, such that $\dis(p,x_j)\leq\lambda$.

Now, let $q$ be a point that is at a distance less than $\lambda$ from either $\preim{f}{\{v\}}$ or $\preim{f}{\{w\}}$. Since we assumed that $\mathrm{J}$ does not contain critical values, we can send $q$ to a point $p=\psi_{\frac{v+w}{2}-f(q)}(q)$ in $S\cap \preim{f}{\left\{\frac{v+w}{2}\right\}}$ using the gradient flow of $f$ (see the proof of Theorem~\ref{thm:morse}). Notice that $p\in\Tilde{S}_\lambda$ because $f(p)=\frac{v+w}{2}$. 
We also have:
\begin{align*}
\dis(p,q)&\leq\left|\int_{0}^{f(p)-f(q)}\left\Vert\frac{d\psi_{u}(q)}{du}\right\Vert\,du \right| \\
&\leq\frac{\coverwidth}{2\inf_{p\in \preim{f}{\mathrm{J}}}\Vert\nabla f(p)\Vert}
\end{align*}
since $$\left|f(p)-f(q)\right|\leq \frac{\coverwidth}{2}.$$

We further assumed that $\mathrm{J}$ does not contain critical values and is not adjacent to another refined cover element $\mathrm{J}'$ that contains critical values. The same argument made in the proof of Lemma~\ref{lem:delta} shows that
$$\frac{1}{\inf_{p\in \preim{f}{\mathrm{J}}}\Vert\nabla f(p)\Vert}\leq \max\left(\frac{1}{C},\max_{c\in\critic(f)}\mu(c)\cdot\sqrt{\frac{1-g}{g}}\cdot\frac{1}{\sqrt{\coverwidth}}\right)$$
where $C$ is a constant depending only on $f$ and $\man$.\\
Taking a large enough resolution $r$, we have that
$$\dis(p,q)\leq\frac{1}{2}\cdot\max_{c\in\critic(f)}\mu(c)\cdot\sqrt{\frac{1-g}{g}}\cdot\sqrt{\coverwidth}.$$
Since $p\in\Tilde{S}_\lambda$ and $\haus(\sample_n\cap S,\Tilde{S}_\lambda)\leq\lambda$, there exists $x_j\in\sample_n\cap S$ such that 
\begin{align*}
\dis(x_j,q)&\leq\dis(x_j,p)+\dis(p,q)\\
&\leq \lambda+\frac{1}{2}\cdot\max_{c\in\critic(f)}\mu(c)\cdot\sqrt{\frac{1-g}{g}}\cdot\sqrt{\coverwidth}.
\end{align*}
Now, $x_j\in\DD(x_i)$ because $x_j\in S$ which proves our result.
\end{proof}

\subsubsection{Hausdorff convergence rate for the sample}
We now focus on the convergence rate in Hausdorff distance between $\sample_n$ and $\man$. Following the works of~\citet{chazal2014convergence} and~\citet{carriere2018statistical}, we recall that the measure $\mes$, with respect to which the points are sampled, is $(a,b)$-standard.  

\begin{definition}[\change{$(a,b)$-standard measure}]
Let $\nu$ be a Borel probability measure on $\man$. Let $a>0$ and $b>0$.
We say that $\nu$ is $(a,b)$-standard if for every $x\in\man$ and $r>0$:
$$\nu\left(\ball\left(x,r\right)\right)\geq \min\left(1,ar^b\right).$$
\end{definition}

 The $(a,b)$-standard assumption is fairly popular in the set estimation literature (see for example~\citet{cuevas2009set}) and is verified for example when $\mes$ is absolutely continuous with respect to the volume measure and $b=d$, the dimension of $\man$.
 
We recall Theorem 2 of~\citet{chazal2014convergence} that links the $(a,b)$-standard assumption to the Hausdorff convergence rate between a sample and the support of $\mes$.
 
 \begin{theorem}\label{thm:ab}
 Let $\nu$ be a Borel probability measure on $\man$. Let $\sample_n$ be a sample of $n$ points taken from $\nu$. Denote $\spa_\nu$ the support of $\nu$. 
 If $\nu$ is $(a,b)$-standard then for every $\eps>0$:
 $$\mathbb{P}\left(\haus(\sample_n,\spa_\nu)>2\eps\right)\leq\min\left(\frac{2^b}{a\eps^b}\exp(-na\eps^b),1\right).$$
 \end{theorem}

\subsubsection{Convergence rate for the volume of a cover element preimage}
We are interested here in an upper bound on the convergence rate of the volume of preimages of refined cover elements:
$$\max_{\mathrm{J}\in\mathbb{J}}\vol\left( \preim{f}{\mathrm{J}}\right).$$
\begin{lemma}\label{lem:meas}
For any large enough resolution $r$, we have:
$$\max_{\mathrm{J}\in\mathbb{J}}\vol\left( \preim{f}{\mathrm{J}}\right)\leq \eta\cdot\coverwidth^{\frac{d}{d+1}},$$
where
$$\eta=\max_{c\in\critic(f)}\mu(c)^2\cdot\sup_{v\notin f(\critic(f))}\vol_{d-1}\left(\preim{f}{\{v\}}\right)+2\alpha_d\left|\critic(f)\right|,$$
and where $\vol_{d-1}\left(\preim{f}{\{v\}}\right)$ is the $(d-1)$-dimensional volume of $\preim{f}{\{v\}}$ seen as a submanifold of $\man$ and $\alpha_d$ is the volume of the unit ball in $\R^d$.

\end{lemma}
\begin{proof}
We will use the Coarea Formula (see Chapter 2, Theorem 5.8 in~\citet{sakai1996riemannian}) which states that for every integrable function $u\colon\man\rightarrow\R$, we have:

$$\int_\man u\Vert\nabla f\Vert\,d\vol=\int_{\R}\int_{\preim{f}{\{v\}}}u\,d\vol_v\,dv,$$

where $\vol_v$ is the volume measure induced on the $(d-1)$-dimensional manifold $\preim{f}{\{v\}}$, $v$ being a non-critical value. Note that in the integral in the right hand side of the equation above, we only consider non-critical values $v$ of $f$. This is possible because the set of critical values of $f$ has measure 0 in $\R$ by Sard's Theorem.

Let $\mathrm{J}$ be a refined cover element. Let $\eps>0$ and consider $$\man_\eps=\man\setminus \bigcup_{c \in \critic(f)}\ball\left(c,\eps\right).$$
Applying the Coarea Formula to $u=\ind_{\preim{f}{\mathrm{J}}\cap\man_\eps},$ gives
\begin{align*}
\vol\left(\preim{f}{\mathrm{J}}\cap\man_\eps\right)=\int_\man u\,d\vol&\leq\frac{1}{\inf_{p\in\man_\eps}\Vert\nabla f(p)\Vert}\cdot \int_\man u\Vert\nabla f\Vert\,d\vol\\
&=\frac{1}{\inf_{p\in\man_\eps}\Vert\nabla f(p)\Vert}\cdot\int_{\R}\int_{\preim{f}{\{v\}}}u\,d\vol_v\,dv\\
&\leq\frac{1}{\inf_{p\in\man_\eps}\Vert\nabla f(p)\Vert}\cdot\int_{\mathrm{J}}\int_{\preim{f}{\{v\}}}\,d\vol_v\,dv\\
&\leq \frac{1}{\inf_{p\in\man_\eps}\Vert\nabla f(p)\Vert}\cdot\coverwidth\cdot\sup_{v\notin f(\critic(f))}\vol_{d-1}\left(\preim{f}{\{v\}}\right).
\end{align*}
Moreover, for small enough $\eps$, as $\inf_{p\in\man_\eps}\Vert\nabla f(p)\Vert$ is achieved in a neighborhood of a critical point $c$ at a point $p$ such that $\dis(c,p)\geq\eps$, we have (see the proof of Proposition~\ref{prop:covn}):
$$\inf_{p\in\man_\eps}\Vert\nabla f(p)\Vert \geq \frac{2\eps}{\mu_0^2},$$
where $\mu_0\rightarrow\mu(c)$ as $p\rightarrow c$. As such, for small enough $\eps$, we have
$$\mu_0\leq\sqrt{2}\mu(c),$$
and therefore
$$\inf_{p\in\man_\eps}\Vert\nabla f(p)\Vert \geq \frac{\eps}{\mu(c)^2}.$$
Accordingly,
$$\vol\left(\preim{f}{\mathrm{J}}\cap\man_\eps\right)\leq\frac{\max_{c\in\critic(f)}\mu(c)^2}{\eps}\cdot\coverwidth\cdot\sup_{v\notin f(\critic(f))}\vol_{d-1}\left(\preim{f}{\{v\}}\right).$$
Now, by Proposition~\ref{prop:geo_vol}, for every $c\in\critic(f)$, there exists $\eps'_c$ such that for every $\eps\leq\eps'_c$ we have
$$\vol\left(\ball\left(c,\eps\right)\right)\leq 2\alpha_d\cdot\eps^d.$$
Taking $\eps'=\min_{c\in\critic(f)}\eps'_c$, we have that for $\eps\leq\eps'$:
$$\sum_{c\in\critic(f)}\vol\left(\ball\left(c,\eps\right)\right)\leq 2\alpha_d\left|\critic(f)\right|\cdot\eps^d.$$
As such, for small enough $\eps>0$, we have
\begin{align*}
\vol\left(\preim{f}{\mathrm{J}}\right)&\leq\vol\left(\preim{f}{\mathrm{J}}\cap\man_\eps\right)+\sum_{c\in\critic(f)}\vol\left(\ball\left(c,\eps\right)\right)\\
&\leq\frac{\max_{c\in\critic(f)}\mu(c)^2}{\eps}\cdot\coverwidth\cdot\sup_{v\notin f(\critic(f))}\vol_{d-1}\left(\preim{f}{\{v\}}\right)+2\alpha_d\left|\critic(f)\right|\cdot\eps^d.
\end{align*}

Since $\coverwidth\underset{r\rightarrow\infty}{\longrightarrow}0$, we can choose $$\eps=\coverwidth^{\frac{1}{d+1}},$$
and we will have, for a large enough resolution $r$,
$$\vol\left(\preim{f}{\mathrm{J}}\right)\leq\left[\max_{c\in\critic(f)}\mu(c)^2\cdot\sup_{v\notin f(\critic(f))}\vol_{d-1}\left(\preim{f}{\{v\}}\right)+2\alpha_d\left|\critic(f)\right|\right]\cdot\coverwidth^{\frac{d}{d+1}}.$$
\end{proof}

\subsubsection{Convergence rate for the approximation error}

We now put together Lemmas~\ref{lem:delta}, \ref{lem:slice} and \ref{lem:meas} in the following proposition.

\begin{proposition}\label{prop:apprx}
Let $\lambda>0$ such that
$$\omega_f\left(\lambda\right)\leq\frac{g}{1-g}\cdot\frac{\coverwidth}{4}.$$
For any large enough resolution $r$, we have
$$\mathbb{E}\left(\appr_n^p\right)\leq \left(\lambda+\xi\cdot\sqrt{\coverwidth}\right)^p
    +\tilde{\eta}\cdot\coverwidth^{\frac{d}{d+1}}+\frac{\zeta}{\lambda^d}\cdot\exp\left(\frac{-na\lambda^d}{2^d}\right),$$
    where
    $$\xi=\max_{c\in\critic(f)}\mu(c)\cdot\sqrt{\frac{1-g}{g}},$$
    $$\tilde{\eta}=3 |\critic(f)| \left[\sup\frac{d\mes}{d\vol}\right]\cdot\diam(\man)^p\cdot\eta,$$
    $$\eta=\max_{c\in\critic(f)}\mu(c)^2\cdot\sup_{v\notin f(\critic(f))}\vol_{d-1}\left(f^{-1}\{v\}\right)+2\alpha_d\left|\critic(f)\right|,$$
    $$\zeta=\frac{4^d\diam(\man)^p}{a}.$$

\end{proposition}

\begin{proof}
By definition,
$$\appr_n^p=\frac{1}{n}\sum_{i=1}^n\haus([x_i]_{\sim_f},\DD(x_i))^p.$$
Since $\sample_n$ consists of i.i.d. samples:
$$\mathbb{E}\left(\appr_n^p\right)=\mathbb{E}\left(\haus([x_1]_{\sim_f},\DD(x_1))^p\right).$$

First, if $\haus(\man,\sample_n)>\lambda$, we use the loose bound $$\haus([x_1]_{\sim_f},\DD(x_1))\leq\diam(\man).$$
Now, assume that $\haus(\man,\sample_n)\leq\lambda$.
We will denote $\mathbb{J}_c\subseteq\mathbb{J}$ as the collection of elements $\mathrm{J}$ in $\mathbb{J}$ that either contain critical values, or are adjacent to elements containing critical values. We see that $|\mathbb{J}_c|\leq 3|\critic(f)|$.

\begin{itemize}
	\item If $f(x_1)\in \bigcup_{\mathrm{J}\in  \mathbb{J}_c} \mathrm{J} $, we use the bound $\haus([x_1]_{\sim_f},\DD(x_1))\leq\diam(\man)$.
    \item If $f(x_1) \notin \bigcup_{\mathrm{J}\in  \mathbb{J}_c} \mathrm{J}  $, we can use Lemmas~\ref{lem:delta} and \ref{lem:slice} to prove that
    $$\haus([x_1]_{\sim_f},\DD(x_1))\leq\lambda+\xi\cdot\sqrt{\coverwidth}.$$
\end{itemize}

As such,
$$\haus([x_1]_{\sim_f},\DD(x_1))^p\leq
\left(\lambda+\xi\cdot\sqrt{\coverwidth}\right)^p+
\ind_{f(x_1) \in \bigcup_{\mathrm{J}\in  \mathbb{J}_c} \mathrm{J} }\cdot \diam(\man)^p+\ind_{\haus(\man,\sample_n)>\lambda}\cdot \diam(\man)^p.$$
Accordingly,
\begin{align*}
    \mathbb{E}\left(\haus([x_1]_{\sim_f},\DD(x_1))^p\right)\leq&
    \left(\lambda+\xi\cdot\sqrt{\coverwidth}\right)^p+
    \mathbb{P}\left(f(x_1) \in \bigcup_{\mathrm{J}\in  \mathbb{J}_c} \mathrm{J} \right)\cdot \diam(\man)^p\\
    &+\mathbb{P}\left(\haus(\man,\sample_n)>\lambda\right)\cdot \diam(\man)^p,
\end{align*}
and by Lemma~\ref{lem:meas} we obtain
\begin{align*}
    \mathbb{P}\left(f(x_1) \in \bigcup_{\mathrm{J}\in  \mathbb{J}_c} \mathrm{J}  \right)&=\mes\left( \preim{f}{\bigcup_{\mathrm{J}\in\mathbb{J}_c}\mathrm{J}}\right)\\
    &\leq \sum_{\mathrm{J}\in\mathbb{J}_c} \max_{\mathrm{J}\in\mathbb{J}}\mes\left( \preim{f}{\mathrm{J}}\right)\\
    &\leq 3|\critic(f)|\cdot\max_{\mathrm{J}\in\mathbb{J}}\mes\left(\preim{f}{\mathrm{J}}\right)\\
    &\leq 3|\critic(f)|\left[\sup\frac{d\mes}{d\vol}\right]\cdot\max_{\mathrm{J}\in\mathbb{J}}\vol\left(\preim{f}{\mathrm{J}}\right)\\
    &\leq 3|\critic(f)|\left[\sup\frac{d\mes}{d\vol}\right]\eta\cdot\coverwidth^{\frac{d}{d+1}}.
\end{align*}
The measure $\mes$ satisfies the $(a,b)$-standard assumption for $b=d$ and it is fully supported on $\man$ by Assumption~\ref{hypdensity}. Using Theorem~\ref{thm:ab}, we have:

$$\mathbb{P}\left(\haus(\man,\sample_n)>\lambda\right)\leq\min\left(\frac{4^d}{a\lambda^d}\exp\left(\frac{-na\lambda^d}{2^d}\right),1\right).$$

Combining all three bounds leads to the result.

\end{proof}

\subsection{\change{Proof of Theorem~\ref{thm:main_result}}}\label{subsec:proof_thm}
\change{We are now ready to give the proof of Theorem~\ref{thm:main_result}.} From Proposition~\ref{prop:reeb_mapper}, we have the decomposition
$$\gromov_p(\overline{\reeb},\mapper)\leq \estim_n+\appr_n.$$

On one hand, since $\frac{d+\alpha}{\nu}>\max\{2d,2p\}$, Corollary~\ref{cor:estim} shows that
$$\mathbb{E}\left(\estim_n\right)\lesssim n^{-\frac{\nu}{d+\alpha}}.$$

On the other hand, Proposition~\ref{prop:apprx} gives
$$\mathbb{E}\left(\appr_n^p\right)\leq \left(\lambda+\xi\cdot\sqrt{\coverwidth}\right)^p
    +\tilde{\eta}\cdot\coverwidth^{\frac{d}{d+1}}+\frac{\zeta}{\lambda^d}\cdot\exp\left(\frac{-na\lambda^d}{2^d}\right),$$
    for every $\lambda>0$ such that
$$\omega_f\left(\lambda\right)\leq\frac{g}{1-g}\cdot\frac{\coverwidth}{4}.$$
where the constants $\xi$, $\tilde{\eta}$ and $\zeta$ are given in Proposition~\ref{prop:apprx}. The third term, that is given by 
$$\frac{\zeta}{\lambda^d}\cdot\exp\left(\frac{-na\lambda^d}{2^d}\right),$$
converges to zero for   $\lambda(n) \gtrsim n^{-\frac{1}{d+\alpha}}$.
Furthermore, since $f$ is smooth on a compact manifold, it is $L$-Lipschitz for some $L>0$. 
In Proposition~\ref{prop:apprx}, the assumption
$$\omega_f\left(\lambda\right)\leq\frac{g}{1-g}\cdot\frac{\coverwidth}{4},$$
where $\omega_f$ is the modulus of continuity of $f$, can be satisfied by choosing
$$\lambda=\frac{g}{1-g}\cdot\frac{\coverwidth}{4L}.$$
Therefore, in order to make the term above converge to zero (exponentially fast), a sufficient condition is
$$\coverwidth(n)\underset{n\rightarrow\infty}{\sim} n^{-\frac{1}{d+\alpha}},$$
or equivalently
$$r(n)\underset{n\rightarrow\infty}{\sim} n^{\frac{1}{d+\alpha}}.$$
Moreover, the two other terms in Proposition~\ref{prop:apprx} are dominated by $\coverwidth^{p\nu}$.
As such, 
$\mathbb{E}\left(\appr_n^p\right) \lesssim n^{-\frac{p\nu}{d+\alpha}},$
and therefore
$$\mathbb{E}\left(\appr_n\right)\lesssim n^{-\frac{\nu}{d+\alpha}}$$
and finally
$$\mathbb{E}\left(\gromov_p(\overline{\reeb},\mapper)\right)\lesssim n^{-\frac{\nu}{d+\alpha}}.$$

\label{sec:rates}

\section{Computing the Gromov-Wasserstein distance for Mapper graphs in practice}\label{sec:appl}

We conclude this article with a section providing some numerical experiments
 on both our metric measure space version of the Mapper graph $(\mapper,\haus,\dirac\circ\DD^{-1})$, 
 as well as the $\gromov_p$ distances between them. In order to achieve this,  
 we use optimal transport libraries for computing the $\gromov_p$ distances. 
 We emphasize that this section contains merely toy experiments; a more comprehensive study would be needed to explore the relevance of such considered transport metrics for Reeb inference with Mapper in the general context of  machine learning and data science.

Since it is finite, the  metric space structure of the Mapper $(\mapper,\haus,\dirac\circ\DD^{-1})$ can be summarized in a distance matrix $\mathrm{D}$ that contains the pairwise Hausdorff distances between the simplices of the Mapper. 
Furthermore, the measure $\dirac\circ\DD^{-1}$ being discrete, it can be represented with a vector $\mathrm{P}$ storing the measures associated to each simplex. 
Finally, we use the \texttt{POT} Python library~\citep{flamary2021pot} to compute the Gromov-Wasserstein distances between different Mapper graphs.\footnote{Note that this library actually computes \change{an estimate of} the $\gromov_p$ formulation of~\citet{memoli2011gromov} \change{using either a conditional gradient algorithm or stochastic solvers.}}\\
Our code is publicly available at the following repository~\citep{Oulhaj_Mapper_Gromov-Wasserstein_distance}.

\subsection{Change of filter function}
We first provide an experiment where Mapper graphs are computed on a $3$-dimensional point cloud $\sample_n$ representing a human shape \change{from the Princeton segmentation benchmark},\footnote{\change{\url{https://segeval.cs.princeton.edu/}}} see Figure~\ref{fig:human}.

\begin{figure}[t]
    \centering
    \includegraphics[width=.3\textwidth]{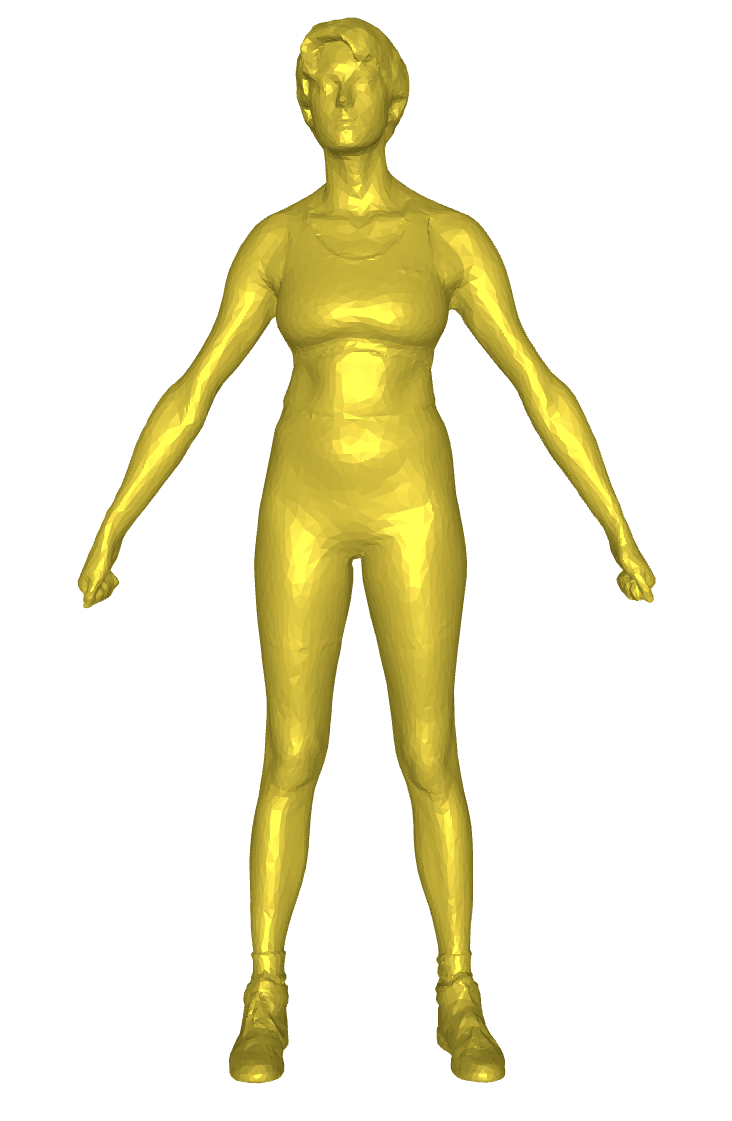}\hspace{0.05\textwidth}
    \includegraphics[width=.4\textwidth]{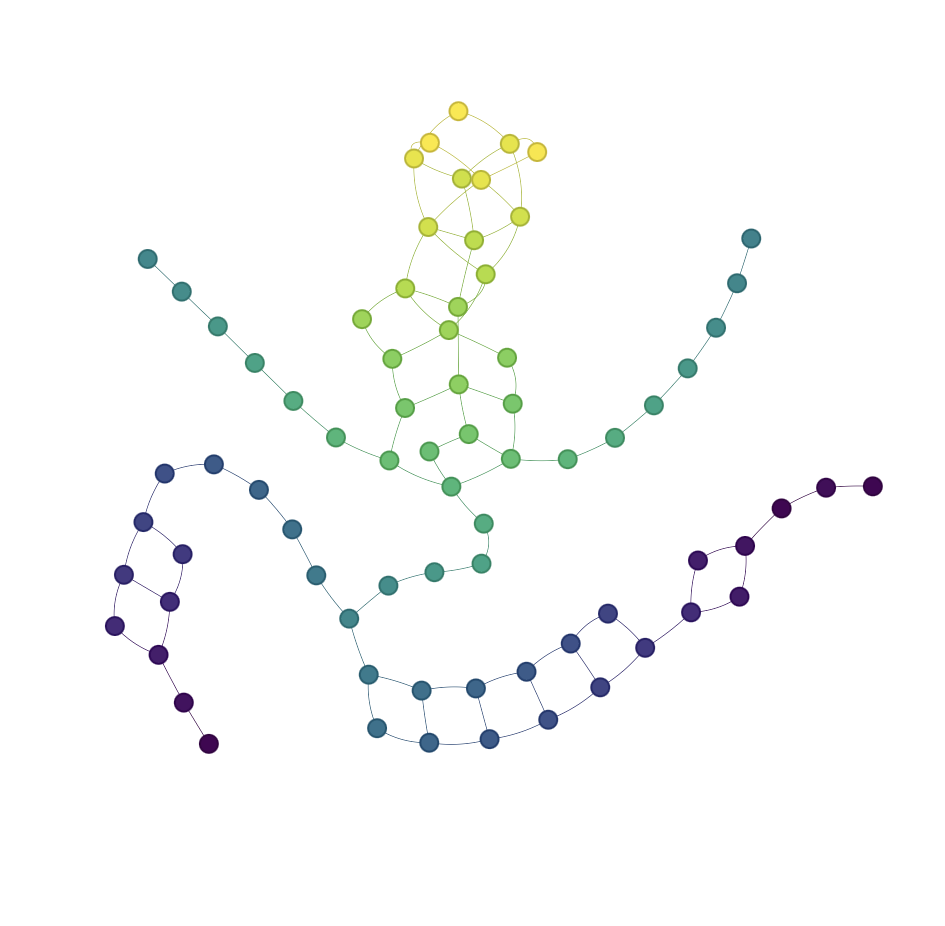}
    \caption{Left: Mesh of a $3$-dimensional point cloud representing a human shape. Right: Mapper graph computed on the point cloud using the height as filter function. Nodes are colored using the average value of the filter.}
    \label{fig:human}
\end{figure}

\change{
We use the same gain $g=0.3$ and resolution $r=25$ for all the Mapper graphs, but different filter functions. More precisely, we use filter functions from a parametrized family $\{f_t\}_t$ defined with
$$f_t(x)=\langle x,t\cdot u+(1-t)\cdot v\rangle,$$
with $u$ and $v$ being the unit vectors associated to the horizontal and vertical directions, respectively. Finally, for deriving the Mapper graphs, we use the same KMeans clustering algorithm with a fixed number of three clusters. For reference, the Mapper computed with filter function $f_0$ is displayed in Figure~\ref{fig:human}. 
}

\change{All the points in the point cloud have the same weights, meaning that we use a uniform probability measure for all the computations. We then explicitly compute 2-Gromov-Wasserstein distances between pairs of Mappers in the family.}


In Figure~\ref{fig:ot_filter}, we plot the resulting 2-Gromov-Wasserstein distances between the Mapper graph computed with $f_0$ and the Mapper graphs computed with $f_t$ with respect to $\Vert f_0-f_t\Vert_{\infty}=\sup_{x\in\mathbb{X}}|f_0(x)-f_t(x)|$. \change{The general shape of the curve is consistent with the stability result given by Theorem 3.3, as the Gromov-Wasserstein distance decreases to zero when the sup norm distance between $f_0$ and $f_t$ decreases to zero.}
\begin{figure}[t]
    \centering
    \includegraphics[width=.5\textwidth]{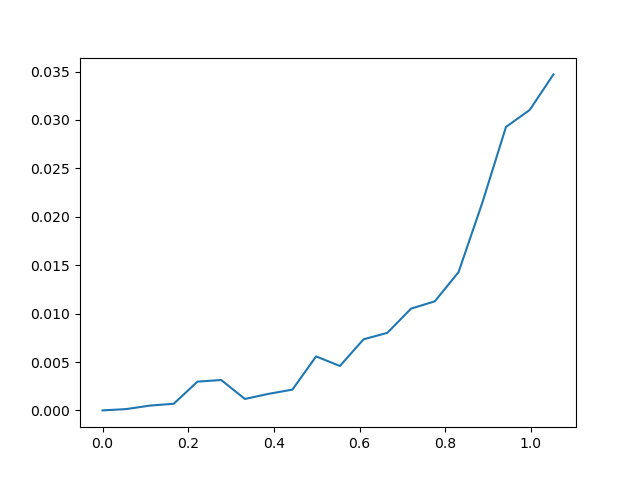}
    \caption{2-Gromov-Wasserstein distances between the Mapper graph computed with $f_0$ and the Mapper graphs computed with $f_t$ with respect to $\Vert f_0-f_t\Vert_{\infty}$.}
    \label{fig:ot_filter}
\end{figure}

\subsection{Change of measure}
In this section,
we illustrate the effect of varying the underlying probability measures, 
by computing Mapper graphs on point clouds sampled from a torus $\torus$ with different probability measures. The torus $\torus$ is a two dimensional manifold that can be parametrized with two angles $(\theta,\phi)\in[0,2\pi]^2$ and can be embedded in $\R^3$ with
\begin{align*}
\embd\colon[0,2\pi]^2&\longrightarrow\R^3\\
(\theta,\phi)&\longmapsto((a+b\cos(\theta))\cos(\phi),(a+b\cos(\theta))\sin(\phi),b\sin(\theta))
\end{align*}
where $a,b>0$ are two radius parameters.
We use the height function $f$ as filter, which is given by the projection of $\embd$ over its first coordinate, i.e., $f\colon(\theta,\phi)\mapsto(a+b\cos(\theta))\cos(\phi)$. It can easily be checked that $f$ is a Morse function, by differentiating twice over $\theta$ and $\phi$.
Moreover, 
the Riemannian metric $\g$ induced by the torus embedding in $\R^3$, 
is given at any point $(\theta,\phi)$ by the positive definite matrix
$$\G(p)=\begin{pmatrix}
b^2 & 0 \\
0 & (a+b\cos(\theta))^2
\end{pmatrix}.$$
We will consider several probability measures $\mes_{p,q}$ on $\torus$ that are parametrized by two values $0<p,q<1$, that give the proportion of points sampled in the region $\phi\in[5\pi/6,7\pi/6]$, and the proportion in the region $\phi\in[0,\pi/6]\cup[11\pi/6,2\pi]$, respectively. More explicitly, we define
$$d\mes_{p,q}(\theta,\phi)=\frac{u_{p,q}(\phi)(a+b\cos(\theta))d\theta d\phi}{2\pi a},$$

where
$$u_{p,q}(\phi)=p\frac{3}{\pi}\ind_{[5\pi/6,7\pi/6]}(\phi)+q\frac{3}{\pi}\ind_{[0,\pi/6]\cup[11\pi/6,2\pi]}(\phi)+(1-p-q)\frac{3}{4\pi}\ind_{[\pi/6,5\pi/6]\cup[7\pi/6,11\pi/6]}(\phi).$$
Notice that the special case $p=q=1/6$ corresponds to the normalized volume measure on $\torus$.
We fix $a=0.75$ and $b=0.25$, and we pick values for $p$ and $q$ in the $3\times 3$ grid $[1/9,1/6,1/3]^2$, resulting in nine different probability measures. Then, we sample $n_{p,q}=2\cdot10^5$ points from each probability measure $\mes_{p,q}$, and construct the corresponding Mapper graphs $\{\mapper_{p,q}\}_{p,q}$ using resolution $r=30$ and gain $g=0.3$ for all $(p,q)$ pairs. See Figure~\ref{fig:torus} for an example of three of these samples, and Figure~\ref{fig:mapper_torus} for the corresponding Mapper graphs. 

First, remark that the Mapper graphs $\{\mapper_{p,q}\}_{p,q}$ are all identical as (combinatorial) simplicial complexes. 
However, as they are different as metric measure spaces, when considering the $2$-Gromov-Wasserstein distances $\Hat{\gromov_2}$ between these metric measure spaces $\{(\mapper_{p,q},\haus,\mes_{p,q}\circ\DD^{-1})\}_{p,q}$, we  obtain a $9\times 9$ distance matrix 
with nonzero entries.
We perform multi-dimensional scaling (MDS) to visualize this distance matrix as a 2-dimensional point cloud in Figure~\ref{fig:mds}. The correspondence between the point labels in Figure~\ref{fig:mds} and the corresponding $(p,q)$ parameters  
is given in Table~\ref{tab:corr}. In Figure~\ref{fig:mds}, we can see three pairs of points, namely $(b,d)$, $(c,g)$ and $(f,h)$, that are close in the MDS representation, illustrating the fact that 
each of these pairs corresponds to two symmetric parameter pairs $(p,q)$
that induce two symmetrical underlying probability measures on $\torus$ with respect to the vertical axis. One can also see that the points labeled $a$ and $i$, which correspond to those two measures that assign the least and the most mass in the two regions of the torus, respectively, are the ones that are the farthest away from the remaining points, indicating that the corresponding Mapper graphs computed using these two measures are the most different from the others. 

\begin{figure}[t]
    \centering
    \includegraphics[width=.25\textwidth]{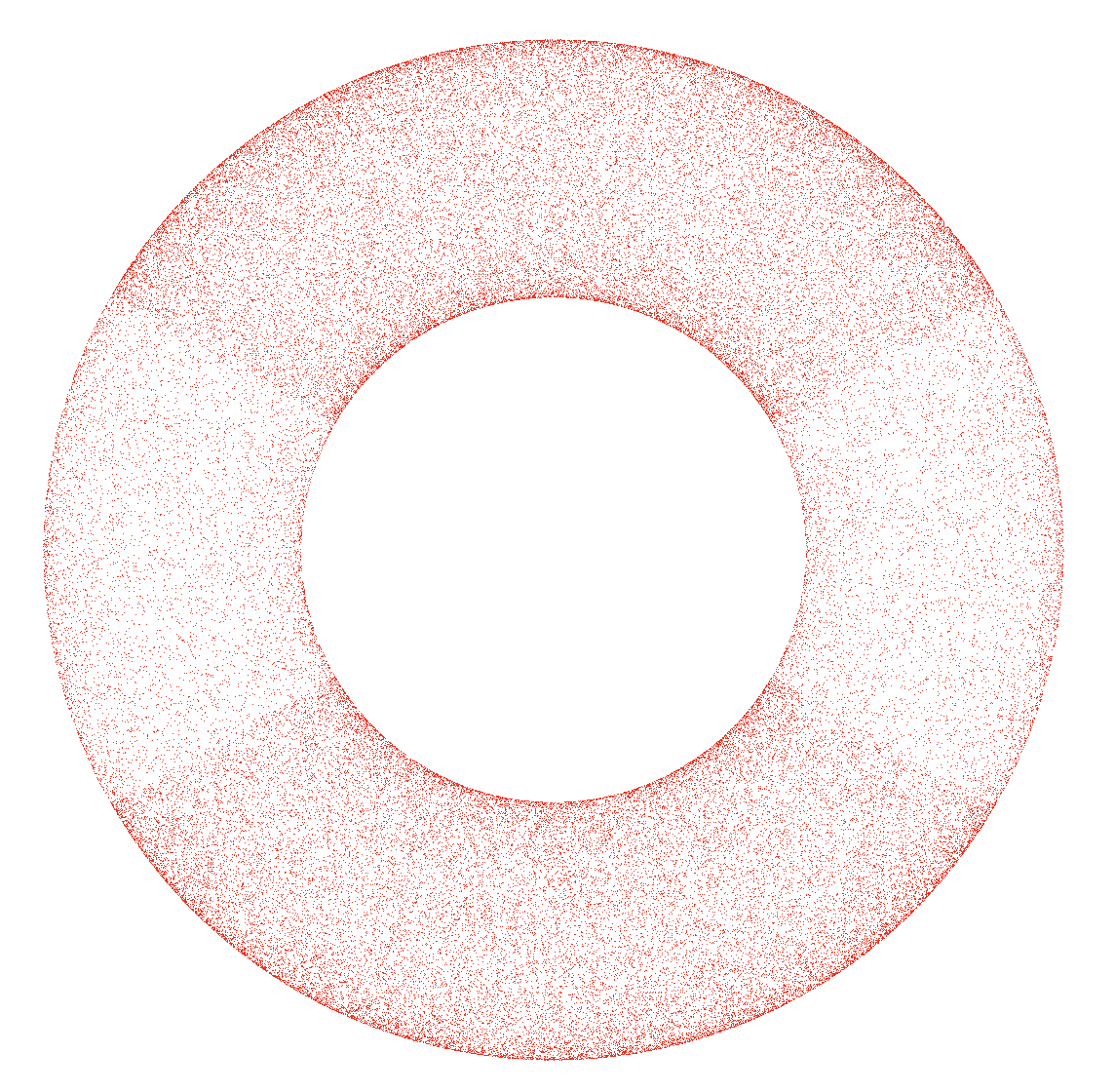}\hspace{0.05\textwidth}
    \includegraphics[width=.25\textwidth]{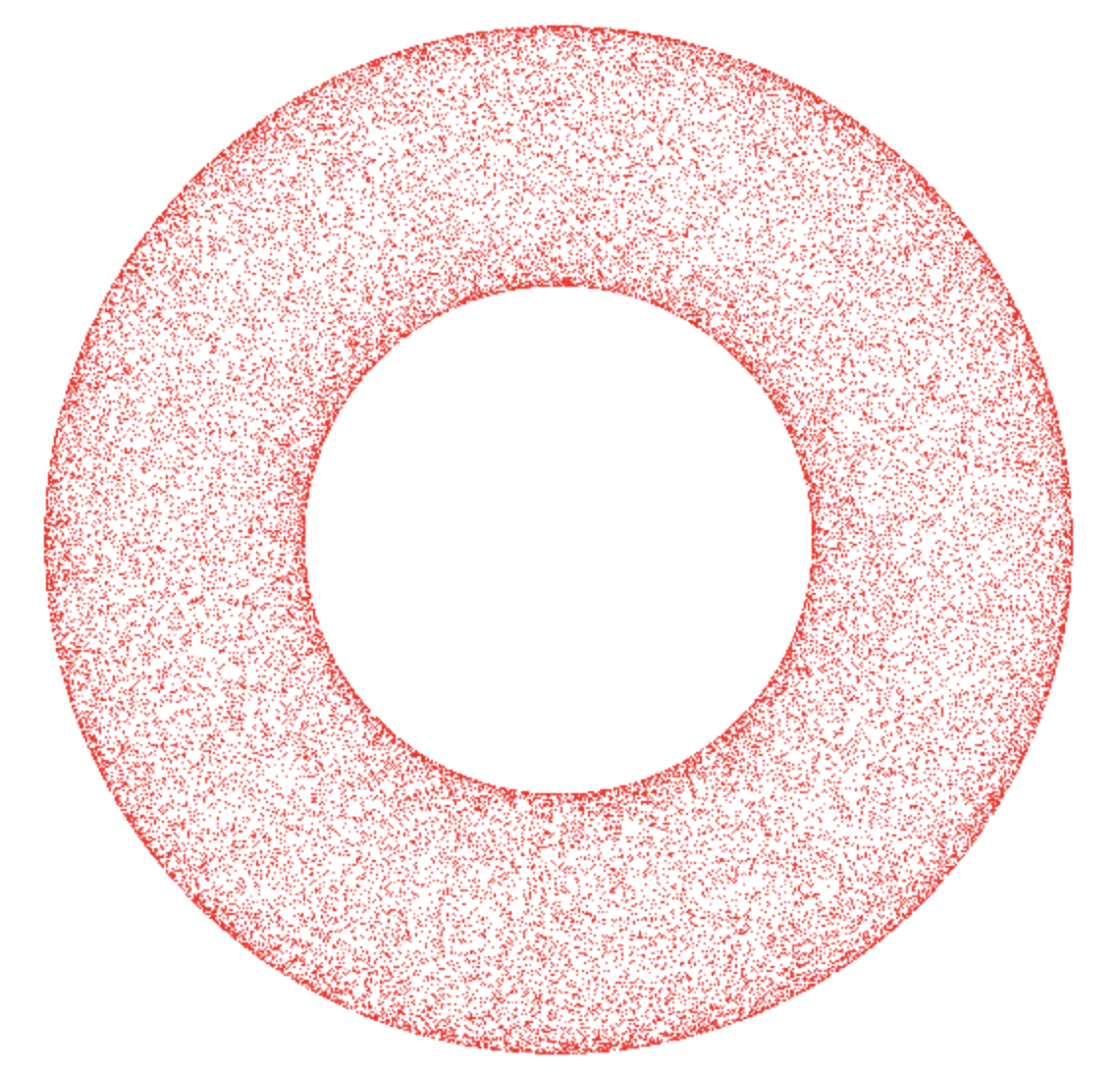}
    \hspace{0.05\textwidth}
    \includegraphics[width=.25\textwidth]{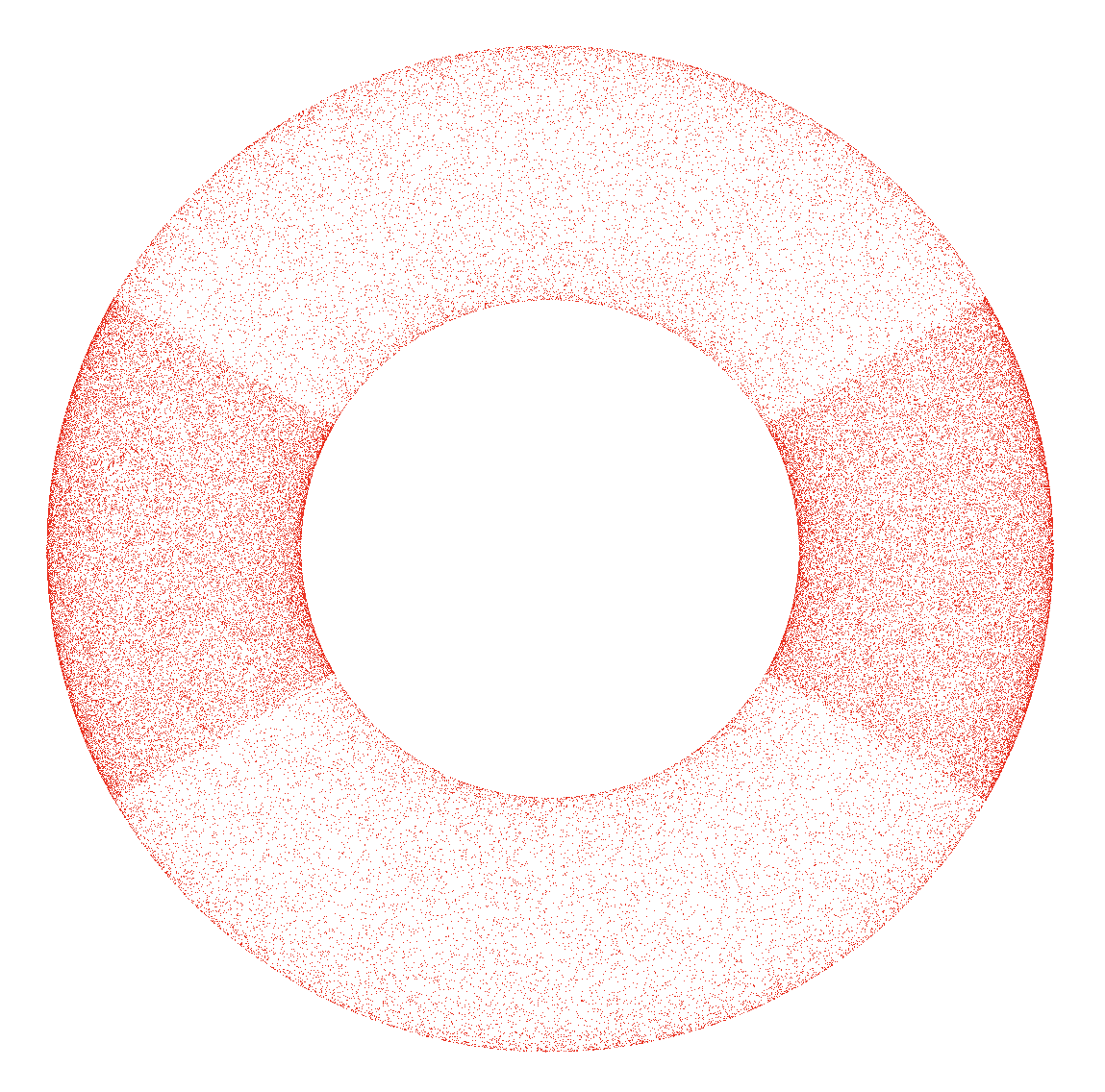}
    \caption{Point clouds sampled from $\mes_{p,q}$ for different values of $p$ and $q$. Left: $p=q=1/12$. Middle: $p=q=1/6$ (uniform measure). Right: $p=q=1/3$.}
    \label{fig:torus}
\end{figure}

\begin{figure}[t]
    \centering
    \includegraphics[width=.25\textwidth]{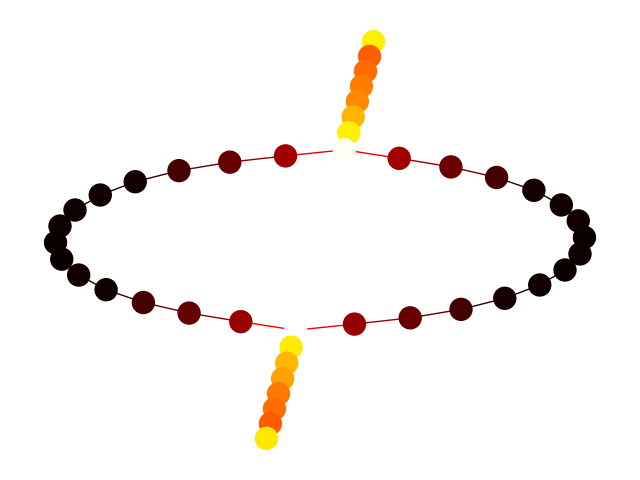}\hspace{0.05\textwidth}
    \includegraphics[width=.25\textwidth]{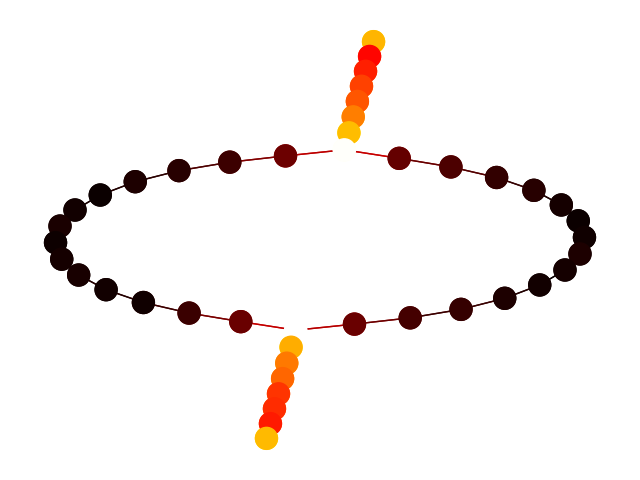}
    \hspace{0.05\textwidth}
    \includegraphics[width=.25\textwidth]{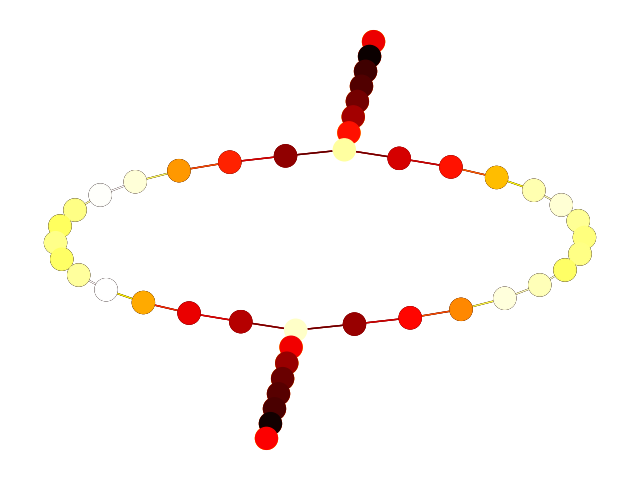}
    \caption{Mapper graphs computed on the samples shown in Figure~\ref{fig:torus}. Vertices and edges are colored using their associated masses. Left: $p=q=1/12$. Middle: $p=q=1/6$ (uniform measure). Right: $p=q=1/3$.}
    \label{fig:mapper_torus}
\end{figure}
\begin{figure}[H]
    \centering
    \includegraphics[width=.5\textwidth]{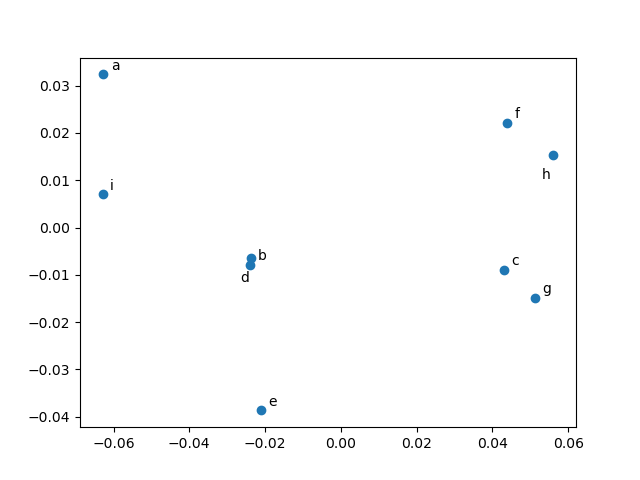}
    \caption{MDS visualization of the Gromov-Wasserstein distance matrix between the different Mapper graphs. \change{The pairs $(b,d)$, $(c,g)$ and $(f,h)$ that each correspond to symmetrical probability measures are close.}}
    \label{fig:mds}
\end{figure}

\begin{table}[t]
\begin{tabular}{ |c|c|c| } 
 \hline
 Label & p & q \\
 \hline
 a & 1/12 & 1/12 \\ 
 \hline
 b & 1/12 & 1/6 \\ 
 \hline
 c & 1/12 & 1/3 \\ 
 \hline
 d & 1/6 & 1/12 \\ 
 \hline
 e & 1/6 & 1/6 \\ 
 \hline
 f & 1/6 & 1/3 \\ 
 \hline
 g & 1/3 & 1/12 \\ 
 \hline
 h & 1/3 & 1/6 \\ 
 \hline
 i & 1/3 & 1/3 \\ 
 \hline
\end{tabular}
\caption{Correspondence between the labels of Figure~\ref{fig:mds} and the parameters of $\mes_{p,q}$.}
\label{tab:corr}
\end{table}

\section{Discussion}
\change{
In this work, we have proposed both a metric measure space framework to describe Reeb and Mapper graphs
as well as control of the rates of convergence of the Mapper graph towards the Reeb graph under the Gromov-Wasserstein metric. Several research directions naturally arise from this work.
}

\change{
First, it would be of interest to characterize convergence rates within a minimax framework, in the spirit of \citet{carriere2018statistical}. This would require, in particular, specifying appropriate classes of Reeb graphs (viewed as metric measure spaces), as well as classes of sampling distributions. One would then need to establish a lower bound and compare it with the upper bound obtained in this work.
}

\change{
A second research direction concerns the filter function \(f\). It would be of interest to extend our results to the setting where the filter function is itself estimated. The stability statement given in Theorem~\ref{thm:stability} suggests a natural path for this extension. This of course requires specifying the underlying statistical framework, in particular how the point cloud is sampled from the ground space.
Another possible extension concerns the multi-parameter setting, i.e., the case where the filter is vector-valued; see~\citet{carriere2022statistical}.
}

\change{
Finally, adapting standard statistical tools, such as confidence regions and permutation tests, to the Gromov-Wasserstein setting is a promising direction for future work. This would, however, require efficient methods for computing Gromov-Wasserstein distances between Mapper graphs. Fortunately, the fast computation of Gromov-Wasserstein distances is a rapidly growing area of research in optimal transport, and several tools are available, see~\citet{kerdoncuff2021sampled}.
} 

\backmatter

\bmhead{Acknowledgements}

The authors would like to thank Gilles Carron and St\'ephane Guillermou for their helpful insight and valuable discussions.

\section*{Declarations}
\bmhead{Conflict}

The authors have no conflicts of interest.

\bmhead{Funding}

The research was supported by two grants from Agence Nationale de la Recherche: ANR JCJC TopModel ANR-23-CE23-0014 and ANR GeoDSIC ANR-22-CE40-0007.
M.C. was also supported by the French government, through the 3IA Cote d'Azur Investments in the project
managed by the National Research Agency (ANR) with the reference number ANR-23-IACL-0001.

\newpage
\begin{appendices}

\section{Elements of Riemannian geometry}
\label{apdx:elemRiem}

A Riemannian manifold is \change{a $C^{\infty}$-manifold $\man$ together with a Riemannian metric $\g$}. The Riemannian metric consists of an \change{inner product} $\g_p$ on each of the tangent spaces $\tangent_p\man$ of $\man$, which satisfies that $p\mapsto\g_p(X_p,Y_p)$ is smooth whenever $X$ and $Y$ are two smooth vector fields defined on $\man$. 

Given some local coordinates $x(p)=(x_1,...,x_d)$ on an open set $U$ of $\man$, 
they induce a basis of the tangent space $\tangent_p\man$. 
We denote the associated {\it coordinate vector fields} as $(\partial_i)_{i=1}^d$, and their dual $1$-forms as $(dx_i)_{i=1}^d$.
Under these notations, we can write the Riemannian metric $\g$ as:
$$\g=\sum_{i,j}\g(\partial_i,\partial_j)\cdot dx_i\cdot dx_j,$$
where $dx_i\cdot dx_j$ is the bilinear form: $(v,w)\mapsto dx_i(v)\cdot dx_j(w).$
As such, $\g$ can be represented in local coordinates as a symmetric positive-definite matrix $\G(p)$ with entries parametrized over $U$.

We will denote
$$\lambda(p)=\min\{\sqrt{\rho},\,\rho\in\mathrm{Sp}(\G(p))\},$$
and
$$\mu(p)=\max\{\sqrt{\rho},\,\rho\in\mathrm{Sp}(\G(p))\},$$
where $\mathrm{Sp}(\G(p))$ is the spectrum of $\G(p)$. Notice that since the Riemannian metric is continuous, both $\lambda$ and $\mu$ are continuous as well.

The following proposition states that the Riemannian metric is bi-Lipschitz equivalent to the Euclidean metric in local coordinates.

\begin{proposition}\label{prop:g_equiv}
Let $U$ be an open set in $\man$, on which we are given coordinates $x(p)=(x_1,...,x_d)$. 
For every $p\in\man$ and for every $v\in\tangent_p\man$:

$$\lambda(p)\cdot \sqrt{\g_0(v,v)}\leq \sqrt{\g(v,v)}\leq \mu(p)\cdot \sqrt{\g_0(v,v)},$$

where $\g_0=\sum_{i}dx_i\cdot dx_i$ is the Euclidean metric in the local coordinates.
\end{proposition}
\begin{proof}
Denoting $w=(dx_1(v),...,dx_d(v))$, we have:

$$\g(v,v)=w^T\cdot \G(p)\cdot w,$$
and 

$$\g_0(v,v)=w^T\cdot w,$$

where $\G(p):=\left[\g_{i,j}\right]_{i,j}$ is the symmetric positive-definite matrix given by $\g_{i,j}=\g(\partial_{i_{|p}},\partial_{j_{|p}})$.\\

Hence,
$$\lambda(p)\cdot \sqrt{\g_0(v,v)}\leq \sqrt{\g(v,v)}\leq \mu(p)\cdot \sqrt{\g_0(v,v)}$$

\end{proof}

Riemannian manifolds can be given a metric space structure as it is possible to measure the ``length" of piecewise smooth curves using the Riemannian metric. This allows to define the {\it geodesic} distance $\dis$ on $\man$. 
The geodesic distance is locally bi-Lipschitz equivalent to the Euclidean distance in local coordinates.

\begin{proposition}\label{prop:dis_equiv}Let $p\in\man$. For every small enough neighborhood $U$ of $p$, we have that for every $q\in U$:
$$\lambda_0\cdot\dis_0(p,q)\leq \dis(p,q)\leq\mu_0\cdot\dis_0(p,q),$$
where $\dis_0(p,q)$ is the Euclidean distance between the representations of $p$ and $q$ in local coordinates and where
$$\lambda_0=\inf_{r\in U}\lambda(r) \textrm{ and }
\mu_0=\sup_{r\in U}\mu(r).$$
In particular $\lambda_0$ tends to $ \lambda(p)$ and
$\mu_0$  tends to  $\mu(p)$ as $\diam(U)$ tends to $0$.
\end{proposition}

\begin{proof}
Let $p\in\man$ and let $U$ be a coordinate neighborhood of $p$.
Up to shrinking $U$, we can assume that:
$$U=\{q\in\man\, ,\, \dis_0(p,q)\leq\eps\},$$
for a given, small enough $\eps>0$, and where $\dis_0$ is the distance associated to the metric $\g_0$, given by 
$$\dis_0(p,q)=\sqrt{x_1(q)^2+\dots+x_d(q)^2},$$
where $q$ is given in local coordinates by $x(q)=(x_1(q),...,x_d(q))$.

From Proposition~\ref{prop:g_equiv}, there exist continuous functions $\lambda,\mu\colon U\rightarrow(0,\infty)$ such that for every $q\in U$ and $v\in\tangent_q\man$, one has
$$\lambda(q)\cdot \sqrt{\g_0(v,v)}\leq \sqrt{\g(v,v)}\leq \mu(q)\cdot \sqrt{\g_0(v,v)}.$$
Furthermore, $\lambda(q),\mu(q)\rightarrow \lambda(p),\mu(p)$ as $q\rightarrow p$ because $\lambda$ and $\mu$ are continuous.

Now, let $\gamma\colon[0,1]\rightarrow\man$ be a piecewise-$C^{\infty}$ curve such that $\gamma(0)=p$ and $\gamma(1)=q\in U$.

\begin{itemize}
\item If $\gamma$ is a segment for $\dis_0$, then it lies in $U$ (as $\dis_0(p,q)\leq\eps$). We also have:
\begin{align*}
\dis_0(p,q)&=\int_0^1\sqrt{\g_0\left(\frac{d\gamma(t)}{dt},\frac{d\gamma(t)}{dt}\right)}\, dt\\
&\geq \frac{1}{\sup_{t\in[0,1]}\ \mu(\gamma(t))}\int_0^1\sqrt{\g\left(\frac{d\gamma(t)}{dt},\frac{d\gamma(t)}{dt}\right)}\, dt\\
&=\frac{\length(\gamma)}{\sup_{t\in[0,1]}\ \mu(\gamma(t))}\\
&\geq \frac{\dis(p,q)}{\sup_{t\in[0,1]}\ \mu(\gamma(t))}.
\end{align*}
\item If $\gamma$ lies inside $U$, then:
\begin{align*}
\length(\gamma)&=\int_0^1\sqrt{\g\left(\frac{d\gamma(t)}{dt},\frac{d\gamma(t)}{dt}\right)}\, dt\\
&\geq \inf_{t\in[0,1]}\ \lambda(\gamma(t)) \cdot \int_0^1\sqrt{\g_0\left(\frac{d\gamma(t)}{dt},\frac{d\gamma(t)}{dt}\right)}\, dt\\
&\geq \inf_{t\in[0,1]}\ \lambda(\gamma(t))\cdot\dis_0(p,q)
\end{align*}
\item If $\gamma$ leaves $U$, then there exists a smallest $t_0$ such that $\gamma(t_0)\notin U$. We have:
\begin{align*}
\length(\gamma)&\geq\int_0^{t_0}\sqrt{\g\left(\frac{d\gamma(t)}{dt},\frac{d\gamma(t)}{dt}\right)}\, dt\\
&\geq \inf_{t\in[0,1]}\ \lambda(\gamma(t))\cdot\int_0^{t_0}\sqrt{\g_0\left(\frac{d\gamma(t)}{dt},\frac{d\gamma(t)}{dt}\right)}\, dt\\
&\geq \inf_{t\in[0,1]}\ \lambda(\gamma(t))\cdot\eps\\
&\geq \inf_{t\in[0,1]}\ \lambda(\gamma(t))\cdot\dis_0(p,q)
\end{align*}
Like mentioned above, $\lambda(q),\mu(q)\rightarrow \lambda(p),\mu(p)$ as $q\rightarrow p$, therefore (and by shrinking $U$ again if necessary) we can ensure that $\lambda_0=\inf_{r\in U}\lambda(r)>0$ and that $\mu_0=\sup_{r\in U}\mu(r)<\infty$.

We have proved that for every $q\in U$, one has
$$\lambda_0\cdot\dis_0(p,q)\leq \dis(p,q)\leq\mu_0\cdot\dis_0(p,q),$$
and we finally see that, as $q\rightarrow p$:
$$\lambda_0,\mu_0\rightarrow \lambda(p),\mu(p).$$
\end{itemize}

\end{proof}

The Riemannian metric also induces a Borel measure $\vol$ called the {\it volume measure}. It can be first defined by specifying the expected value of a function $f\colon\man\rightarrow\mathbb{C}$ compactly supported on a single chart $\varphi\colon U\subseteq\man\rightarrow V\subseteq\mathbb{R}^d$:

$$\int_{\man}f\,d\vol=\int_{V}\left(f\cdot\sqrt{\deter(\G)}\right)\circ\varphi^{-1}\,d\lambda,$$

where \begin{align*}
\G\colon U&\longrightarrow\mathbb{R}^{d\times d}\\
p&\longmapsto(\g(\partial_{i_{|p}},\partial_{j_{|p}}))_{i,j}
\end{align*}
and $\lambda$ is the Lebesgue measure on $\mathbb{R}^d$.\\
It can be checked, with the help of the substitution rule, that the above definition does not depend on the choice of the coordinate neighborhood.\\
This definition can be extended to general compactly supported functions by using a partition of unity of a coordinate neighborhood cover of $\man$: we simply sum the expected values over the elements of the partition.\\
Subsequently, as the expected value of compactly supported functions is well defined, the Riesz representation theorem allows to consider the unique associated positive Borel measure on $\man$: this is the volume measure $\vol$.\\\\
 We have the following asymptotic equivalence result for the volume of small geodesic balls.
\begin{proposition}[\change{Volume of small geodesic balls}]\label{prop:geo_vol}
Let $\ball(p,\eps)$ be the geodesic ball around $p\in\man$ of radius $\eps$. We have that:
$$\vol\left(\ball(p,\eps)\right)\underset{\eps\rightarrow 0}{\sim} \alpha_d\cdot\eps^d,$$
where $\alpha_d$ is the volume of the unit ball in $\R^d$.
%
\end{proposition}
\begin{proof}
Let us consider a chart $\varphi\colon U\rightarrow\R^d$ around $p$ such that $\varphi(p)=(0,...,0)$ and 

$$\g(\partial_{i_{|p}},\partial_{j_{|p}})=
\begin{cases}
1, \text{ if } i=j\\
0, \text{ otherwise}
\end{cases}$$

so as to have $\g_p=\g_0$, where $\g_0$ is the Euclidean metric in local coordinates.

For a small enough $\eps>0$ and by Proposition~\ref{prop:dis_equiv}, we know that 
$$\ball_0(p,\eps/\mu_0)\subseteq\ball(p,\eps)\subseteq \ball_0(p,\eps/\lambda_0)\subseteq U,$$
for $\lambda_0,\mu_0>0$, where $\ball_0(p,\cdot)$ stands for the Euclidean ball using the distance $\dis_0$ in local coordinates.

As such,
$$\vol\left(\ball_0(p,\eps/\mu_0)\right)\leq \vol\left(\ball(p,\eps)\right)\leq \vol\left(\ball_0(p,\eps/\lambda_0)\right).$$
Furthermore, $\lambda_0,\mu_0\rightarrow 1$ as $\eps\rightarrow 0$ because $\g_p=\g_0$.\\
Denote
$$c=\inf\{\deter(\G(q)),q\in\ball_0(p,\eps/\mu_0)\},$$
$$C=\sup\{\deter(\G(q)),q\in\ball_0(p,\eps/\lambda_0)\},$$
where $\G=\left[\g(\partial_i,\partial_j)\right]_{i,j}$.
Since $\deter(\G(p))=1$ (as $\g_p=\g_0$) and $\G$ is smooth, we have $c,C\rightarrow 1$ as $\eps\rightarrow 0$.

Now, denoting the Lebesgue measure in $\R^d$ as $l$, we have
\begin{itemize}
\item One one hand: \begin{align*}
\vol\left(\ball_0(p,\eps/\mu_0)\right)&=\int_{\{x\in\R^d,\,\Vert x\Vert\leq \eps/\mu_0\}}\sqrt{\deter(\G)}\circ\varphi^{-1}\,dl\\
&\geq \sqrt{c}\int_{\{x\in\R^d,\,\Vert x\Vert\leq \eps/\mu_0\}}\,dl\\
&\geq \sqrt{c}\cdot\alpha_d\cdot \left(\frac{\eps}{\mu_0}\right)^d
\end{align*}

\item On the other hand:\begin{align*}
\vol\left(\ball_0(p,\eps/\lambda_0)\right)&=\int_{\{x\in\R^d,\,\Vert x\Vert\leq \eps/\lambda_0\}}\sqrt{\deter(\G)}\circ\varphi^{-1}\,dl\\
&\leq \sqrt{C}\int_{\{x\in\R^d,\,\Vert x\Vert\leq \eps/\lambda_0\}}\,dl\\
&\leq \sqrt{C}\cdot \alpha_d\cdot \left(\frac{\eps}{\lambda_0}\right)^d
\end{align*}
\end{itemize}

As such,
$$\vol\left(\ball(p,\eps)\right)\underset{\eps\rightarrow 0}{\sim} \alpha_d\cdot\eps^d$$
\end{proof}

For the following result we will assume a lower bound on the {\it Ricci curvature} of $\man$. The Ricci curvature $\ricci$ is a symmetric bilinear form on the tangent spaces and can be thought of as the Laplacian of the metric $\g$, see for instance Chapter 9 in~\citet{petersen2006riemannian}. We will adopt the convention that $\ricci\geq k$ means that all eigenvalues $\rho$ of $\ricci$ satisfy $\rho\geq k$. 

The next theorem allows to compare the volume of geodesic balls in $\man$ to the volume of geodesic balls in certain model manifolds $\spaceform^d_k$ called {\it constant-curvature space forms}, $d$ being the dimension of both $\man$ and $\spaceform^d_k$. \change{The space $\spaceform^d_k$ is either $\mathbb{R}^d$, a rescaled $d$-dimensional sphere or a rescaled hyperbolic space of dimension $d$ for $k=0$, $k>0$ and $k<0$, respectively} (see for example \change{Subsection 9.1} in~\citet{petersen2006riemannian}).
Denote $v(d,k,r)$ as the volume of a ball of radius $r$ in $\spaceform^d_k$. 

\begin{theorem}[\change{Bishop-Cheeger-Gromov Theorem, see for instance Lemma~36 in~\citet{petersen2006riemannian}}]
\label{thm:bishop_gromov}
Suppose that $\man$ is a $d$-dimensional complete Riemannian manifold such that $\ricci\geq(d-1)\cdot k$, for some $k \in \R$. Then  for any $p \in \man$, 
$$\eps\longmapsto\frac{\vol\left(\ball(p,\eps)\right)}{v(d,k,\eps)} $$
is a nonincreasing function whose limit is 1 as $\eps\rightarrow 0$. 
\end{theorem}
 
Note that in the above theorem, $v(d,k,\eps)$ is independent of the base point $p$. 
This is very convenient for providing bounds on the volume of geodesic balls in $\man$ that are uniform with respect to the choice of base points. Notice also that since $\spaceform^d_k$ is a $d-$dimensional Riemannian manifold, Proposition~\ref{prop:geo_vol} applies for $v(d,k,\eps)$, i.e.,
$$v(d,k,\eps)\underset{\eps\rightarrow 0}{\sim} \alpha_d\cdot \eps^d.$$

We now define the gradient of a smooth real valued function on a Riemannian manifold. 
Let $x(p)=(x_1,...,x_d)$ be some local coordinates on an open set $U$ of $\man$
with associated coordinate vector fields $(\partial_i)_{i=1}^d$.
Let $f:\man\rightarrow\R$ be a smooth function. We can define its differential $df$ as the $1$-form that satisfies:
$df(\partial_i)=\frac{\partial f}{\partial x_i}$, 
for all $1\leq i \leq d.$
The gradient $\nabla f$ of $f$ is then defined as the vector field satisfying:
$\g(\nabla f,v)=df(v),$
for all vector fields $v$. It can also be seen as the Riesz representative of $df$.
Given our local coordinates, we can express the gradient vectors in their associated coordinate vector field bases.
\begin{proposition}\label{prop:grad}
In the local coordinates associated to the open set $U$, we have that:
$$\nabla f=\sum_{i=1}^d a_i\cdot\partial_i,$$
where:
$$(a_1,...,a_d)=\left(\frac{\partial f}{\partial x_1},...,\frac{\partial f}{\partial x_d}\right)\cdot\left[\g(\partial_i,\partial_j)\right]_{i,j}^{-1}.$$
In particular, we have:
$$\frac{1}{\mu(p)}\cdot \sqrt{\sum_{i=1}^d\frac{\partial f(p)}{\partial x_i}^2}\leq \sqrt{\g\left(\nabla f(p),\nabla f(p)\right)}\leq \frac{1}{\lambda(p)}\cdot \sqrt{\sum_{i=1}^d\frac{\partial f(p)}{\partial x_i}^2}.$$

\end{proposition}

\begin{proof}
The coordinate vector fields are a basis of the tangent spaces $\tangent_p\man$ for all $p\in U$. As such, in $U$, the gradient $\nabla f$ can be expressed as a linear combination:
$$\nabla f=\sum_{i=1}^d a_i\cdot\partial_i.$$
Now, we know that for all $i$:
$$\g\left(\nabla f,\partial_i\right)=df(\partial_i)=\frac{\partial f}{\partial x_i},$$
and as such
$$\sum_{j=1}^d a_j\cdot \g\left(\partial_j,\partial_i\right)=\frac{\partial f}{\partial x_i}.$$
This gives
$$(a_1,...,a_d)\cdot\left[\g(\partial_i,\partial_j)\right]_{i,j}=\left(\frac{\partial f}{\partial x_1},...,\frac{\partial f}{\partial x_d}\right),$$
and
$$(a_1,...,a_d)=\left(\frac{\partial f}{\partial x_1},...,\frac{\partial f}{\partial x_d}\right)\cdot\left[\g(\partial_i,\partial_j)\right]_{i,j}^{-1}.$$

Furthermore,
\begin{align*}
\g\left(\nabla f,\nabla f\right)&=(a_1,...,a_d)\cdot\left[\g(\partial_i,\partial_j)\right]_{i,j}\cdot(a_1,...,a_d)^T\\
&=\left(\frac{\partial f}{\partial x_1},...,\frac{\partial f}{\partial x_d}\right)\cdot\left[\g(\partial_i,\partial_j)\right]_{i,j}^{-1}\cdot\left[\g(\partial_i,\partial_j)\right]_{i,j}\cdot \left[\g(\partial_i,\partial_j)\right]_{i,j}^{-1}\cdot \left(\frac{\partial f}{\partial x_1},...,\frac{\partial f}{\partial x_d}\right)^T\\
&=\left(\frac{\partial f}{\partial x_1},...,\frac{\partial f}{\partial x_d}\right)\cdot \left[\g(\partial_i,\partial_j)\right]_{i,j}^{-1}\cdot \left(\frac{\partial f}{\partial x_1},...,\frac{\partial f}{\partial x_d}\right)^T.
\end{align*}
As such,
$$\frac{1}{\mu(p)}\cdot \sqrt{\sum_{i=1}^d\frac{\partial f(p)}{\partial x_i}^2}\leq \sqrt{\g\left(\nabla f(p),\nabla f(p)\right)}\leq \frac{1}{\lambda(p)}\cdot \sqrt{\sum_{i=1}^d\frac{\partial f(p)}{\partial x_i}^2}.$$
\end{proof}

\section{Proofs for Section~\ref{sec:background}}\label{apdx:proofs}

\subsection{Proof of Lemma~\ref{lem:flow}}

Since $\preim{f}{[a,b]}$ does not contain critical points, the following function is well defined:
\change{\begin{align*}
\rho\colon\man  &\longrightarrow \R\\
q &\longmapsto \begin{cases}
\frac{1}{\Vert\nabla f(q)\Vert^2}& \text{if } f(q)\in [a,b],\\
0& \text{otherwise.}
\end{cases}
\end{align*}}

Now, consider the vector field defined by $X_q=\rho(q)\cdot\nabla f(q)$ and the flow $(\psi_t)_{t\in\R}$ associated to $X_q$. In other words, $\psi_t$ is the solution of the differential equation: 
$$\frac{d\psi_t(p)}{dt}=X_{\psi_t(p)},\,\psi_0(p)=p.$$
For details on the existence and uniqueness of $\psi_t$, see Lemma~$2.4$ of \citet{milnor1963morse}. The main assumption made is that $X_q$ vanishes outside a compact subset of $\man$, which is true in our case since $\man$ is compact and $\preim{f}{[a,b]}$ is closed. Note that $\forall t\in\R$, $\psi_t$ is a diffeomorphism. Also, $\psi_t\circ\psi_s=\psi_{t+s}$.

Notice that we constructed $\rho$ so that for every $p\in \preim{f}{[a,b]}$, with the notation $q=\psi_t(p)$, we have:
\begin{align*}
\frac{d(f\circ\psi_t)(p)}{dt}&=\g\left(\frac{d\psi_t(p)}{dt},\nabla f(q)\right)\\
&=\g\left(\frac{\nabla f(q)}{\Vert\nabla f(q)\Vert^2},\nabla f(q)\right)\\
&=1.
\end{align*}
Consequently, $f\circ\psi_t(p)\colon t\mapsto t+f(p)$, and therefore $$\psi_{b-a}\left(\preim{f}{\{a\}}\right)= \preim{f}{\{b\}},$$
and 
$$\psi_{a-b}\left(\preim{f}{\{b\}}\right)= \preim{f}{\{a\}}.$$
Finally, $\psi_{b-a}$ is a diffeomorphism between $\preim{f}{\{a\}}$ and $\preim{f}{\{b\}}$.

\subsection{Proof of Theorem~\ref{thm:morse}}
First, notice that an $\eps>0$ satisfying the assumption of the theorem always exists. By Lemma \ref{lem:morse}, $\critic(f)$ is finite. As such, there exists $\eps>0$ such that $\preim{f}{[v-\eps,v+\eps]}$ contains no critical points. Hence, let $\eps>0$ be such a positive real number. Then, for every $v'\in[v-\eps,v+\eps]$, Lemma~\ref{lem:flow} ensures that $\preim{f}{\{v\}}$ and $\preim{f}{\{v'\}}$ are diffeomorphic, under the map $\psi_{v'-v}$ corresponding to the gradient flow of $f$.\\
Define 
\begin{align*}
\psi\colon \preim{f}{\{v\}}\times[-\eps,\eps] &\longrightarrow \preim{f}{[v-\eps,v+\eps]}  \\
(q,t)&\longmapsto \psi_{t}(q)
\end{align*}

The function $\psi$ is continuous because it is locally Lipschitz: 
\begin{itemize}
    \item for a fixed $q\in \preim{f}{\{v\}}$: $t\mapsto\psi_{t}(q)$ is locally Lipschitz as it is continuously differentiable,
    \item for a fixed $t\in [-\eps,\eps]$: $q\mapsto\psi_{t}(q)$ is locally Lipschitz as it is a diffeomorphism.
\end{itemize}
The inverse of $\psi$ is given by $\psi^{-1}\colon p \longmapsto (\psi_{v-f(p)}(p),f(p)-v)$, which we now prove is continuous. Let $p\in \preim{f}{[v-\eps,v+\eps]}$ and $(p_n)_{n\in\N}$ a sequence in $\preim{f}{[v-\eps,v+\eps]}$ such that $p_n\converge p$.\\
We have:
$$\dis(\psi_{v-f(p)}(p),\psi_{v-f(p_n)}(p_n))\leq \dis(\psi_{v-f(p)}(p),\psi_{v-f(p)}(p_n))+\dis(\psi_{v-f(p)}(p_n),\psi_{v-f(p_n)}(p_n)).$$
Now, as mentioned above, the function $q\mapsto\psi_{v-f(p)}(q)$ is continuous because it is a diffeomorphism. Hence:
$$\dis(\psi_{v-f(p)}(p),\psi_{v-f(p)}(p_n))\converge 0.$$
Moreover,

\begin{align*}
 \dis(\psi_{v-f(p)}(p_n),\psi_{v-f(p_n)}(p_n))&\leq \left|\int_{v-f(p)}^{v-f(p_n)}\left\Vert\frac{d\psi_{u}(p_n)}{du}\right\Vert\,du \right| \\
 &\leq \left|\int_{v-f(p)}^{v-f(p_n)}\frac{1}{\Vert\nabla f(\psi_{u}(p_n))\Vert}\,du \right|.
\end{align*}

As $f$ is smooth and $\preim{f}{[v-\eps,v+\eps]}\cap\critic(f)=\varnothing$, and  denoting:
$$L:=\sup_{q\in \preim{f}{[v-\eps,v+\eps]}}\frac{1}{\Vert\nabla f(q)\Vert},$$
we have $L<\infty$. Therefore,
$$\dis(\psi_{v-f(p)}(p_n),\psi_{v-f(p_n)}(p_n)\leq L\cdot |f(p)-f(p_n)|\converge 0.$$

Finally, we showed that $$\dis(\psi_{v-f(p)}(p),\psi_{v-f(p_n)}(p_n))\converge 0,$$ and thus that $\psi^{-1}$ is continuous.

\subsection{Proof of Proposition~\ref{prop:local_path_con}}

When $v\in f(\man)$ is not critical, $\preim{f}{\{v\}}$ is a submanifold of $\man$ and as such is locally path connected.\\
Now, let $v\in f(\man)$ be a critical value.\\
\change{First, $\preim{f}{\{v\}}$ is locally path connected around critical points. Let $c\in\preim{f}{\{v\}}$ be a critical point.} By Lemma~\ref{lem:morse}, there exists an open neighborhood $U\subseteq\man$ of $c$ and a chart $\varphi\,:\,U\subseteq\man\rightarrow \mathbb{R}^d$ containing $c$ such that for every $p\in U$:
$$f(p)=f(c)-\sum_{j=1}^ix_j^2+\sum_{j=i+1}^dx_j^2,$$
where $x=\varphi(p)$ and the integer $i$ depends only on the signature of the Hessian at the critical point $c$.\\
Notice that the level set $\preim{f}{\{v\}}$ is given in local coordinates by 
$$\preim{f}{\{v\}}\cap U=\left\{\varphi^{-1}(x):\,\sum_{j=1}^ix_j^2=\sum_{j=i+1}^dx_j^2\right\},$$
and that $\varphi(c)=0$. For every $p\in\preim{f}{\{v\}}\cap U$, the path $\gamma\colon t\mapsto \varphi^{-1}(t\cdot\varphi(p))$ between $c$ and $p$ stays in $\preim{f}{\{v\}}\cap U$. This is because for every $t\in[0,1]$
\begin{align*}
\sum_{j=1}^i\varphi(\gamma(t))_j^2-\sum_{j=i+1}^d\varphi(\gamma(t))_j^2&=t^2\cdot\left(\sum_{j=1}^i\varphi(p)_j^2-\sum_{j=i+1}^d\varphi(p)_j^2\right)\\
&=0.
\end{align*}
\change{Second, $\preim{f}{\{v\}}$ is locally path connected away from critical points. The set $\critic(f)$ is finite and hence closed.} Therefore, $\man'=\man\setminus\critic(f)$ is an open submanifold of $\man$. $v$ is no longer a critical value for the restriction of $f$ on $\man'$, as such $\preim{f}{\{v\}}\cap\man'$ is locally path connected.

\section{Proofs for Section~\ref{sec:reeb_top}}\label{apdx:reeb_geom_proofs}
\subsection{Proof of Proposition~\ref{prop:reeb_adherence}}

Let $a\in\comp(\man)$ such that there exists $(a_n)_{n\in\N}\in\left(\reeb\right)^{\N}$ that converges to $a$.

$\bullet$ We first show that $a$ is connected. 
    Suppose that $a$ is not connected. Since it is compact, this means that it can be written as the disjoint union of two compact sets $a_0\sqcup a_1$.\\ \change{Recall the notation $c^\eps=\bigcup_{z\in c}\ball(z,\eps)$ for every subset $c$ of $\man$.}\\  Now, $a_0$ and $a_1$ being disjoint and both compact, there exists $\eps>0$ such that $a_0^\eps$ and $a_1^\eps$ are disjoint. To see this, consider for example $\eps :=\frac{1}{2}\ \inf_{x\in a_0,y\in a_1}\dis(x,y)$  and notice that $\inf_{x\in a_0,y\in a_1}\dis(x,y)>0$ by compactness of $a_0$ and $a_1$. \\
    Next, there exists $N\in\N$ such that $\haus(a_N,a)\leq\frac{\eps}{2}$. This means that 
     $a_N\subseteq a^{\frac{\eps}{2}}=a_0^{\frac{\eps}{2}}\sqcup a_1^{\frac{\eps}{2}},$ 
    and since $a_N$ is connected, it is included in one of $a_0^{\frac{\eps}{2}}$ and $a_1^{\frac{\eps}{2}}$ and not in the other. W.l.o.g.,
    we suppose that it is included in $a_0^{\frac{\eps}{2}}$. 
    Let $x\in a_1$, there exists $z\in a_N$ such that $\dis(x,z)\leq\frac{\eps}{2}$ because $a\subseteq a_N^{\frac{\eps}{2}}$. 
    In the same way, there exists $y\in a_1$ that satisfies $\dis(y,z)\leq\frac{\eps}{2}$, since $a_N\subseteq a_0^{\frac{\eps}{2}}$. 
    However, this means that there exist two points $(x,y)\in a_0\times a_1$ such that $\dis(x,y)\leq \eps$ and this contradicts the fact that $a_0^\eps$ and $a_1^\eps$ are disjoint. Note that we have shown more generally that every limit in the Hausdorff distance  of a sequence of connected compact sets is connected.

$\bullet$ We now show that $f$ is constant on $a$. 
    Let $x,y\in a$. By Proposition~\ref{prop:haus_limit}, we know that there exist two sequences $(x_n)_{n\in\N}$ and $(y_n)_{n\in\N}$ in $(a_n)_{n\in\N}$ such that $x_n\converge x$ and $y_n\converge y$. 
    Now, for every $n$, since  $x_n,y_n\in a_n$, then $f(x_n)=f(y_n)$ as $f$ is constant on $a_n$. Using this fact, we have:
    \begin{align*}
        | f(x)-f(y) |&\leq | f(x)-f(x_n) | + | f(x_n)-f(y_n) |+ | f(y_n)-f(y) |\\
        &\leq | f(x)-f(x_n) | + | f(y_n)-f(y) |.
    \end{align*}
    By continuity of $f$, it yields that $f(x)=f(y)$.
\subsection{Proof of Lemma~\ref{lem:haus_cont}}

Let $(a_n=[x_n]_{\sim_f})_{n\in\N} \in \reeb^{\N}$ be such that $(x_n)_{n\in\N}$ admits a limit $x \in \man$. As $\overline{\reeb}$ is compact, it converges if and only if it has one and only one adherent point. It admits an adherent point by compactness of $\overline{\reeb}$ and we have to show that it is unique and equal to $[x]_{\sim_f}$. Let $(a_{\phi(n)})_{n\in\N}$ be a subsequence of $a_n$ that converges to an element $a\in\overline{\reeb}$. 

$\bullet$ We first show that $a\subseteq[x]_{\sim_f}$. By Proposition~\ref{prop:haus_limit}, we know that:
$$a=\left\{y \in \man \, , \textrm{ there exists a sequence } (y_n)_n  \in \prod_n a_{\phi(n)} 
\textrm{ such that } y_n \converge y \right\}.$$
Now, $(x_{\phi(n)})_{n\in\N}$ is a subsequence of $(x_n)_{n\in\N}$ and as such $x_{\phi(n)}\converge x$. Therefore, $x\in a$, and furthermore by Proposition~\ref{prop:reeb_adherence}, $a$ is connected and included in $\preim{f}{\{v\}}$, where $v=f(x)$. These three points prove that $a\subseteq[x]_{\sim_f}$, since by definition $[x]_{\sim_f}$ is the maximal subset that satisfies these three exact requirements. 

$\bullet$ We now show that $[x]_{\sim_f}\subseteq a$. Let $y\in[x]_{\sim_f}$. As $v=f(x)$ is not a critical value, by Theorem~\ref{thm:morse}, there exists $\eps>0$ and a homeomorphism $\psi:\preim{f}{\{v\}}\times[-\eps,\eps] \rightarrow \preim{f}{[v-\eps,v+\eps]}$. 
Since $f$ is continuous, there exists $N\in\N$ such that $\forall n\geq N$: $|f(x_{\phi(n)})-v|\leq \eps$. Therefore, the following sequence is well defined:
$$\forall n\in\mathbb{N}:\,y_n=\begin{cases}
    x_{\phi(n)}&\text{ if } n<N,\\
    \psi\left(y,f(x_{\phi(n)})-v\right)&\text{otherwise}.
\end{cases}$$
By continuity of $\psi$ and $f$, $y_n\converge y$ (since $\psi(\cdot,0)$ is the identity by construction of $\psi$, see the proofs of Lemma~\ref{lem:flow} and Theorem~\ref{thm:morse}). It only remains to show that for a large enough $M\in\N$, one has $y_n\in a_{\phi(n)}$,
to conclude that $y\in a$.\\\\
\change{By Proposition~\ref{prop:local_path_con}, }the set $\preim{f}{\{v\}}$ is locally path-connected. Thus, there exists $\delta>0$ such that for all  $z\in \preim{f}{\{v\}}$, $\dis(x,z)\leq\delta$ implies that $z\in[x]_{\sim_f}$.\\\\
On the other hand, by continuity, we have $\preim{\psi}{x_{\phi(n)}}\converge \preim{\psi}{x}=(x,0)$. From the expression of $\psi^{-1}$ given in the proof of Theorem~\ref{thm:morse}, we find that   $\preim{\psi}{x_{\phi(n)}}=\left(z_n,w_n\right)$ where $z_n \in \preim{f}{\{v\}}$ and  $w_n=f(x_{\phi(n)})-v$. Thus, there exists $M\in\N$ such that $\forall n\geq M$: $\dis(z_n,x)\leq\delta$ and then 
 $z_n$ and $y$ are in the same connected set $[x]_{\sim_f}$. Moreover, by continuity of the map $\psi\left(\cdot,w_n\right) : \preim{f}{\{v\}} \rightarrow  \preim{f}{\{w_n\}}$, the images  $x_{\phi(n)} = \psi\left(z_n,w_n\right)$ and $y_n = \psi\left(y,w_n\right)$  are carried to the same connected component of $\preim{f}{\{w_n\}}$. Consequently $y_n\in[x_{\phi(n)}]_{\sim_f}$ and we know that $[x_{\phi(n)}]_{\sim_f}=a_{\phi(n)}$. See Figure \ref{fig:lemma_continuity} for an illustration.

Finally, $(a_n)_{n\in\N}$ admits one and only one adherent point, which is $[x]_{\sim_f}$.

\begin{figure}[t]
    \centering
    \includegraphics[width=0.4\textwidth]{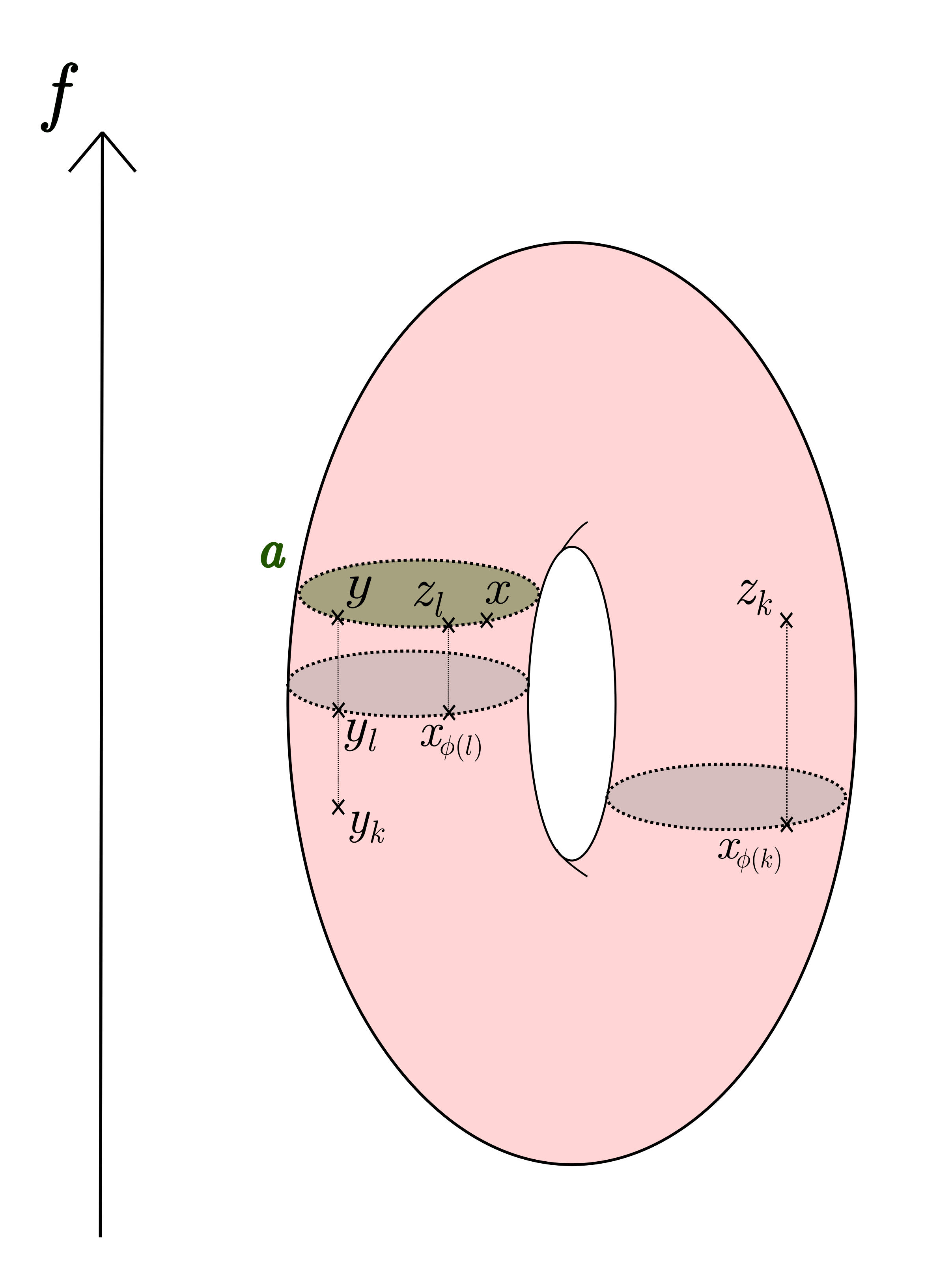}
    \caption{Illustration of the argument made in the proof of Lemma \ref{lem:haus_cont}. Two points $x_{\phi(k)}$ and $x_{\phi(l)}$ of the sequence $(x_{\phi(n)})_{n\in\N}$ are given as examples. The sequence $(z_n)_{n\in\N}$ is made of the projections of the elements of $(x_{\phi(n)})_{n\in\N}$ onto the level set where $a$ belongs using the gradient flow. Similarly, the sequence $(y_n)_{n\in\N}$ is made of the projections of $y$ onto the level sets $\preim{f}{\{x_{\phi(n)}\}}$. We see that, by continuity, $(z_n)_{n\in\N}$ must be in $a$ (for a large $n$). This means that, $y_n$ and $x_{\phi(n)}$ must be in the same element of the Reeb graph.}
    \label{fig:lemma_continuity}
\end{figure}

\subsection{Proof of Proposition~\ref{prop:reeb_finite}}

\change{
In all the section,  $\psi_{\cdot}$ denotes the gradient flow of $f$, as introduced in Lemma~\ref{lem:flow}. In order to prove Proposition~\ref{prop:reeb_finite}, we first show the following Lemmas~\ref{lem:3cases} and~\ref{lem:critical}. 
}

\change{
Lemma~\ref{lem:3cases} allows us to simplify the study of elements in $\overline{\reeb}$ that are associated to critical values by only considering the limits of sequences in $\reeb$ whose filter values are monotone. 
\begin{lemma} 
\label{lem:3cases}
Let $a\in\overline{\reeb}$ such that $v=\Tilde{f}(a)$ is a critical value. Let $v^-$ (resp. $v^+$) denote the closest critical value $v'$ to $v$ satisfying $v'<v$ (resp. $v'>v$) if it exists.  
Let also $(a_n)_{n\in\N}$ be a sequence of elements in $\reeb$ that converges to $a$. Then, we can extract a subsequence $(a_{\phi(n)})_{n\in\N}$ from $(a_n)_{n\in\N}$ such that one of the three following assertions is satisfied:
\begin{enumerate}
    \item $\forall n\in\mathbb{N}, $ $\Tilde{f}(a_{\phi(n)})=\Tilde{f}(a) = v$;
    \item there exists a connected component $c$ of $\preim{f}{\{\frac{v^-+v}{2}\}}$ such that $\forall n\in\N$, $a_{\phi(n)}=\psi_{v_n-(\frac{v^-+v}{2})}(c)$ where $v_n\in(\frac{v^-+v}{2},v) $, and $(v_n)_{n\in\N}$ converges to $v$;
    \item there exists a connected component $c$ of $\preim{f}{\{\frac{v+v^+}{2}\}}$  such that $\forall n\in\N$, $a_{\phi(n)}=\psi_{v_n-(\frac{v+v^+}{2})}(c)$ where $v_n\in(v,\frac{v+v^+}{2})$, and $(v_n)_{n\in\N}$ converges to $v$.
\end{enumerate}
\end{lemma}
}

\begin{proof}[Proof of Lemma~\ref{lem:3cases}]
If for every $n$ large enough the sequence $(a_n)_{n\in\N}$ satisfies $\Tilde{f}(a_n)=v$, then the first assertion is satisfied. In the following of the proof, we thus assume that $\forall n\in\N$, there exists $N_n\geq n$ such that $\Tilde{f}(a_{N_n})\neq v$. This allows us to extract a subsequence $(a_{\phi(n)})_{n\in\N}$ that satisfies $\Tilde{f}(a_{\phi(n)})\neq v$ for all $n\in\N$. W.l.o.g., by re-extracting, we can assume that $\forall n\in \N,\,\Tilde{f}(a_{\phi(n)})< v$.


Notice that $a_{\phi(n)}\converge a$. Let $x\in a$ and $(x_n)_{n\in\N}\in(a_{\phi(n)})_{n\in\N}$ such that $x_n\converge x$. 
By continuity of $f$, there exists $N\in\N$, such that $\forall n\geq N$, $f(x_n)\in(\frac{v^-+v}{2},v)$.
$[\frac{v^-+v}{2},f(x_n)]$ does not contain critical values. \change{According to Lemma~\ref{lem:flow}, we can construct a diffeomorphism $\psi_{f(x_n)-\frac{v^-+v}{2}}$ between $\preim{f}{\{\frac{v^-+v}{2}\}}$ and $\preim{f}{\{f(x_n)\}}$ using the gradient flow of $f$.} 

Consider now the connected components $(c_i)_i$ of $\preim{f}{\{\frac{v^-+v}{2}\}}$. We will prove that: $$ \exists i,\,\forall n\in\N,\,\exists k_n\geq n:\,a_{\phi(k_n)}=\psi_{f(x_{k_n})}(c_i) .$$
The opposite of this means that for all connected components $c_i$ there exists $n_i$ such that $\forall k\geq n_i$: $a_{\phi(k)}$ is not diffeomorphic to $c_i$. Since the number of $(c_i)_i$ is finite (see Proposition~\ref{prop:local_path_con}), taking $K=\max_i n_i$, we would have that $\forall k\geq K$, $a_{\phi(k)}$ is not diffeomorphic to any of the $c_i$ under the flow $\psi_{f(x_k)-\frac{v^-+v}{2}}$. This contradicts that $\forall n\geq N$: $\preim{f}{\{f(x_n)\}}$ and $\preim{f}{\{\frac{v^-+v}{2}\}}$ are diffeomorphic under the flow $\psi_{f(x_k)-\frac{v^-+v}{2}}$, and that this diffeomorphism induces a bijection between connected components. This can be understood as a application of the pigeonhole principle, see Figure \ref{fig:reeb_tirroir} for an illustration. 

This allows to extract a further subsequence $(a_{\Phi(n)})_{n\in\N}$ that converges to $a$, and that consists only of diffeomorphic images of a fixed connected component $c_i$, under well determined maps $(\psi_{v_n-\frac{v^-+v}{2}})_{n\in\N}$ such that $v_n\converge v$.
 
We therefore proved the lemma.

\begin{figure}[t]
    \centering
    \includegraphics[width=0.4\textwidth]{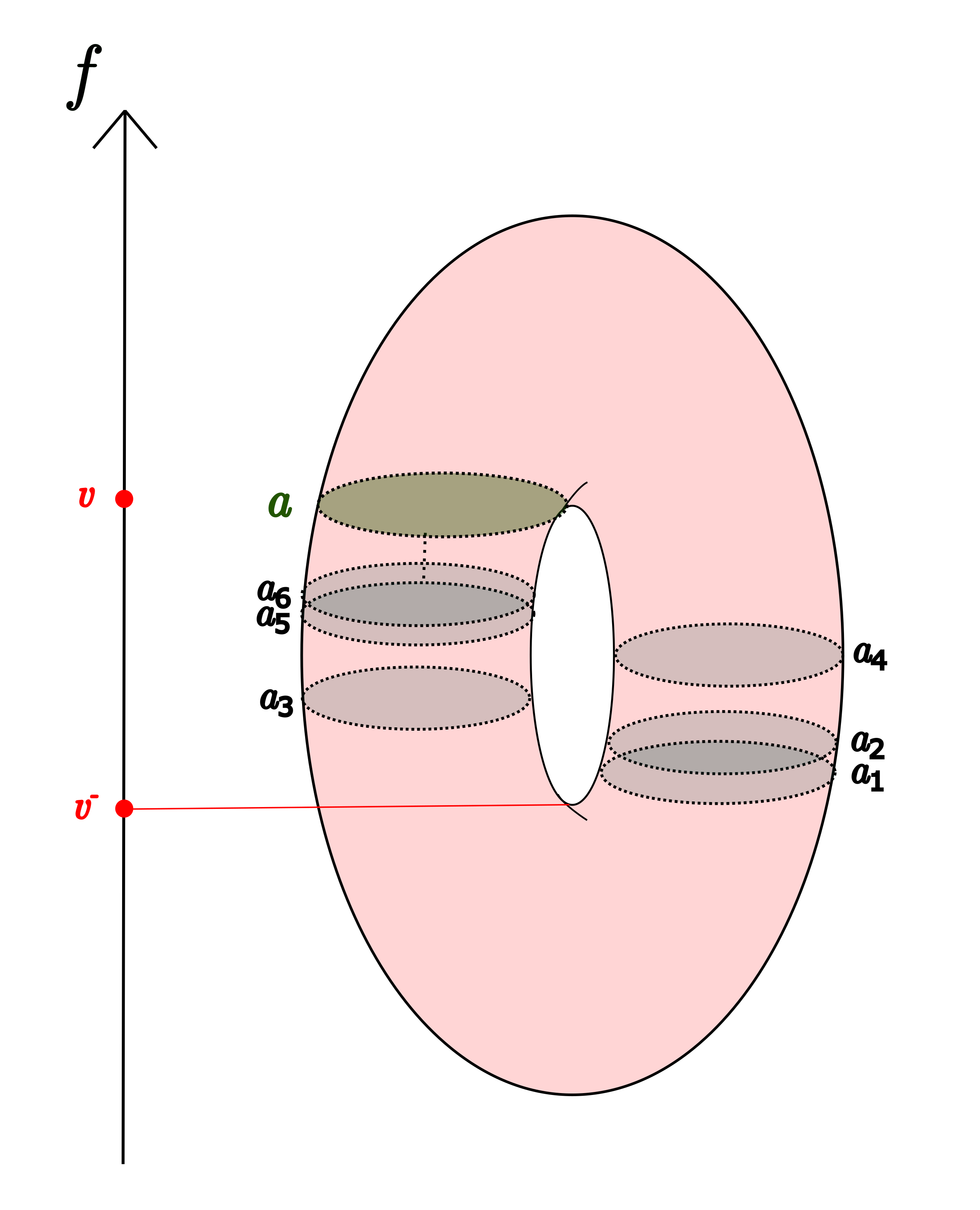}
    \caption{Example of a convergent sequence $(a_n)_{n\in\N}$ of Reeb graph elements  associated to filter values inside the open interval $(v^-,v)$. Since the number of cylinders in $\preim{f}{(v^-,v)}$ is finite, $(a_n)_{n\in\N}$ must have an infinite amount of terms inside one of them.}
    \label{fig:reeb_tirroir}
\end{figure}
\end{proof}

\change{Lemma~\ref{lem:critical} shows that two sequences in $\reeb$ that converge to elements associated to a critical value $v$, both diffeomorphic images of the same connected component and having filter values that are smaller than $v$, have the same limit.
\begin{lemma}\label{lem:critical}
Let $a\in\overline{\reeb}$ such that $v=\Tilde{f}(a)$ is a critical value, and $v^-$ and  $v^+$ defined as in Lemma~\ref{lem:3cases}. Let $(u_n)_{n\in\N}$ and $(v_n)_{n\in\N}$ be two sequences in $[\frac{v^-+v}{2},v)$ such that $u_n\converge v$ and $v_n\converge v$. Let $c$ be a connected component of $\preim{f}{\{\frac{v^-+v}{2}\}}$. Then we have: 
$$\haus(\psi_{u_n-\frac{v^-+v}{2}}(c),\psi_{v_n-\frac{v^-+v}{2}}(c))\converge 0$$
\end{lemma}}

\begin{proof}[Proof of Lemma~\ref{lem:critical}]
W.l.o.g., we assume that $u_n\leq v_n$. 
Let $\eps>0$. Denote $$\man_\eps=\man\setminus\bigcup_{c\in\critic(f)}\ball\left(c,\frac{\eps}{4\left|\critic(f)\right|}\right).$$ 
Since there are no critical values in $[u_n,v_n]$, we can use the gradient flow $(\psi_t)_t$ of $f$ (see Lemma~\ref{lem:flow}) to send any point $p$ in $\psi_{u_n-\frac{v^-+v}{2}}(c)$ to a point $q$ in $\psi_{v_n-\frac{v^-+v}{2}}(c)$ (and conversely) using the curve:
 \begin{align*}
 \gamma\colon[0,v_n-u_n]&\longrightarrow \man\\
 t&\longmapsto\psi_{t+u_n-\frac{v^-+v}{2}}(p),
 \end{align*}
where $\gamma(0)=p$ and $\gamma(v_n-u_n)=q$.\\
Now, for every $c\in\critic(f)$, if $\gamma$ enters $\ball\left(c,\frac{\eps}{4\left|\critic(f)\right|}\right)$, consider 
$$t^c_0=\inf\left\{t\in[0,v_n-u_n],\,\gamma(t)\in \ball\left(c,\frac{\eps}{4\left|\critic(f)\right|}\right)\right\},$$
and 
$$t^c_1=\sup\left\{t\in[0,v_n-u_n],\,\gamma(t)\in \ball\left(c,\frac{\eps}{4\left|\critic(f)\right|}\right)\right\}.$$
Otherwise, take $t^c_0=t^c_1=0$.
Notice that
$$\dis(\gamma(t^c_0),\gamma(t^c_1))\leq \frac{\eps}{2\left|\critic(f)\right|}.$$
 Denoting
 $$\mathcal{U}=[0,v_n-u_n]\setminus\bigcup_{c \in \critic(f)}[t^c_0,t^c_1],$$
 we have $\gamma\left(\mathcal{U}\right)\in\man_\eps$. We also  have:
\begin{align*}
\dis(p,q)&\leq\int_{t\in \mathcal{U}}\left\Vert\frac{d\psi_{t+u_n-\frac{v^-+v}{2}}(p)}{dt}\right\Vert\,dt + \sum_{c \in \critic(f)} \dis(\gamma(t^c_0),\gamma(t^c_1)) \\
&\leq\int_{t\in \mathcal{U}}\frac{dt}{\Vert\nabla f(\psi_{t+u_n-\frac{v^-+v}{2}}(p))\Vert} + \sum_{c \in \critic(f)} \frac{\eps}{2\left|\critic(f)\right|}\\
&\leq\frac{|v_n-u_n|}{\inf_{p\in\man_\eps}\Vert\nabla f(p)\Vert}+ \frac{\eps}{2}.
\end{align*}
Since $v_n-u_n\converge 0$ and $\inf_{p\in\man_\eps}\Vert\nabla f(p)\Vert > C$ for some $C>0$, for $n$ large enough we have $\dis(p,q)\leq\eps$, and finally $\haus(\psi_{u_n-\frac{v^-+v}{2}}(c),\psi_{v_n-\frac{v^-+v}{2}}(c))\converge 0$.

\end{proof}

\begin{proof}[Proof of Proposition~\ref{prop:reeb_finite}]

Recall that $\Tilde{f}\colon \overline{\reeb}\rightarrow\R$ is defined as the function that associates any adherent point of $\reeb$ to the constant value that $f$ takes on it.

 Let $a\in\overline{\reeb}$ be non-empty. First, assume that $v=\Tilde{f}(a)$ is not a critical value of $f$. Let $(a_n)_{n\in\N}$ be a sequence of elements in $\reeb$ that converges to $a$. Let $x\in a$ and choose a sequence $(x_n)_{n\in\N}$ such that $x_n\in a_n$ for every $n\in\N$ and $x_n\converge x$. By Lemma~\ref{lem:haus_cont}, $a_n\converge[x]_{\sim_f}$ and as such $a=[x]_{\sim_f}$ and in particular  $a\in\reeb$. We thus have showed that for any  $a\in\overline{\reeb}\setminus\reeb$ which is non empty, $\Tilde{f}(a)$ is a critical value.

We now consider a non-empty $a\in\overline{\reeb}\setminus\reeb$ and $v=\Tilde{f}(a)$ the corresponding critical value. By Lemma~\ref{lem:morse}, the set of critical values is finite because $\man$ is compact. It remains to check that the  preimage of a critical value is a finite set. 
 
Let $(a_n)_{n\in\N}$ be a sequence of elements in $\reeb$ that converges to $a$. If the first case of Lemma~\ref{lem:3cases} applies, then $(a_n)_{n\in\N}$ must be constant for $n$ large enough in order to converge because the number of connected components of $\preim{f}{\{v\}}$ is finite (see Proposition~\ref{prop:local_path_con}), which implies that $a\in\reeb$ leading to a contradiction because $a\in\overline{\reeb}\setminus\reeb$.

W.l.o.g. we assume that the second point of Lemma~\ref{lem:3cases} is satisfied for $a$, we thus have $a_{\phi(n)}=\psi_{v_n-\frac{v^-+v}{2}}(c)$ with $c$ being a connected component of $\preim{f}{\{\frac{v^-+v}{2}\}}$ and $v_n\converge v$. . According to Lemma ~\ref{lem:critical}, the limit of $(a_{\phi(n)})_{n\in\N}$ is the same as the limit of $(b_n)_{n\in\N}$ defined as:
$$\forall n\in\N:\,b_n=\begin{cases}
    c&\text{ if } v-\frac{1}{n}\leq\frac{v^-+v}{2},\\
    \psi_{v-\frac{1}{n}}(c)&\text{otherwise}.
\end{cases}$$
Finally, notice that $(b_n)_{n\in\N}$ depends only on $c$ and we showed that it converges to $a$. The number of connected components of $\preim{f}{\{\frac{v^-+v}{2}\}}$ and $\preim{f}{\{\frac{v^++v}{2}\}}$ being finite (see Proposition~\ref{prop:local_path_con}), the set $\left\{a\in\overline{\reeb},\,\Tilde{f}(a)=v\right\}$ is therefore also finite.\\\\
We conclude that $\overline{\reeb}\setminus\reeb$ is finite.
\end{proof}
\subsection{Proof of Proposition~\ref{prop:reeb_sep}}

On one hand, $f$ is constant on the elements of $\overline{\reeb}$. As such, there is a map $\Tilde{f}_{|\reeb}$, that assigns to each element of $\reeb$ the unique value that $f$ takes on it. 
For each critical value $v$, $\preim{\Tilde{f}_{|\reeb}}{\{v\}}$ is finite. This is because $\preim{f}{\{v\}}$ is compact  and locally path-connected, and a locally path-connected compact space has finitely many connected components. 
Now, there are finitely many critical values of $f$ by Lemma~\ref{lem:morse}. Hence, denoting $C$ \change{as} the union over all critical values $v$ of $\preim{\Tilde{f}_{|\reeb}}{\{v\}}$, one has that $C$ is finite. 

On the other hand, $\man$ is separable, and therefore there exists a countable subset $D$ of $\man$ such that $\man=\overline{D}$. We now show that the countable subset $E=C\cup \{[x]_{\sim_f},\,x\in D\}$ is dense in $\reeb$. Let $a\in\reeb$. 
\begin{itemize}
	\item If $\Tilde{f}_{|\reeb}(a)$ is a critical value, then $a\in C$.
	\item If $\Tilde{f}_{|\reeb}(a)$ is not a critical value, then let $x\in a$. Given that $x\in\man$ and $D$ is dense in $\man$, there exists $(x_n)_{n\in\N}$ in $D$ such that $x_n\converge x$. Consider the sequence $([x_n]_{\sim_f})_{n\in\N}$, by Lemma~\ref{lem:haus_cont} we have $[x_n]_{\sim_f}\converge a$.
\end{itemize}  

We showed that $E$ is dense in $\reeb$. Therefore, its closure in $\overline{\reeb}$ satisfies $\reeb\subseteq\overline{E}$. Hence, $\overline{\reeb}=\overline{E}$.

\section{Proof of Theorem~\ref{thm:stability}}\label{apdx:stab_proof}

The following proof uses similar arguments to the proofs of lemmas~\ref{lem:meas} and \ref{lem:delta}.

First,
$$\gromov_p^p(\overline{\reeb}_f,\overline{\reeb}_\ff)\leq\inf_{\eta}\int_{\overline{\reeb}_f\times\overline{\reeb}_\ff}\haus(a,b)^p\,d\eta(a,b),$$
where the infimum is taken over all coupling measures $\eta$ on $\overline{\reeb}_f\times\overline{\reeb}_\ff$ that have marginals $\vol\circ\pi_f^{-1}$ and $\vol\circ\pi_\ff^{-1}$. This is because the Hausdorff distance $\haus$ can be defined on $\overline{\reeb}_f\sqcup\overline{\reeb}_\ff$ via its inclusion in $\comp(\man)$, and is a metric coupling.\\
Now, notice that if we denote 
 \begin{align*}
     \tilde{\pi}\colon\man&\longrightarrow\overline{\reeb}_f\times\overline{\reeb}_\ff\\
     x&\longmapsto([x]_{\sim_f},[x]_{\sim_\ff}),
 \end{align*}
$\vol\circ\tilde{\pi}^{-1}$ is an example of a coupling measure having marginals $\vol\circ\pi_f^{-1}$ and $\vol\circ\pi_\ff^{-1}$.\\
As such,
\begin{align*}
\gromov_p^p(\overline{\reeb}_f,\overline{\reeb}_\ff)&\leq\int_{\overline{\reeb}_f\times\overline{\reeb}_\ff}\haus(a,b)^p\,d\vol\circ\tilde{\pi}^{-1}(a,b)\\
&=\int_\man\haus([x]_{\sim_f},[x]_{\sim_\ff})^p\,d\vol(x).
\end{align*}

Consider $I$ as the set of real values that are at least $3\inorm{f-\ff}$ away from the critical values of $f$ and $\ff$, i.e.
$$I=\left(f\left(\critic(f)\right)\cup \ff\left(\critic(\ff\right)\right)^{3\inorm{f-\ff}}.$$
Notice that the width of $I$, i.e. its Lebesgue measure, is at most $12\crib\inorm{f-\ff}$.\\
Let $x\in\man$. We will distinguish two cases.\\
First, assuming that $f(x)\in I$, we use the Coarea Formula (see Chapter 2, Theorem 5.8 in~\citet{sakai1996riemannian}) which gives for every integrable function $u\colon\man\rightarrow\R$ that:

$$\int_\man u\Vert\nabla f\Vert\,d\vol=\int_{\R}\int_{\preim{f}{\{v\}}}u\,d\vol_v\,dv,$$

where $\vol_v$ is the volume measure induced on the $(d-1)$-dimensional manifold $\preim{f}{\{v\}}$, $v$ being a non-critical value.

Let $\eps>0$ and consider $$\man_\eps=\man\setminus \bigcup_{c \in \critic(f)}\ball\left(c,\eps\right).$$
Denoting $u=\ind_{\preim{f}{I}\cap\man_\eps},$ we find
\begin{align*}
\vol\left(\preim{f}{I}\cap\man_\eps\right)=\int_\man u\,d\vol\leq&\frac{1}{\inf_{p\in\man_\eps}\Vert\nabla f(p)\Vert}\cdot \int_\man u\Vert\nabla f\Vert\,d\vol\\
=&\frac{1}{\inf_{p\in\man_\eps}\Vert\nabla f(p)\Vert}\cdot\int_{\R}\int_{\preim{f}{\{v\}}}u\,d\vol_v\,dv\\
\leq&\frac{1}{\inf_{p\in\man_\eps}\Vert\nabla f(p)\Vert}\cdot\int_{I}\int_{\preim{f}{\{v\}}}\,d\vol_v\,dv\\
\leq& \frac{1}{\inf_{p\in\man_\eps}\Vert\nabla f(p)\Vert}\cdot12\crib\inorm{f-\ff}\cdot\\
&\sup_{v\notin f(\critic(f))}\vol_{d-1}\left(\preim{f}{\{v\}}\right).
\end{align*}
Using the same argument as in the proof of Lemma~\ref{lem:meas}, we can bound $1/\inf_{p\in\man_\eps}\Vert\nabla f(p)\Vert$ for small enough $\eps$ by:
\begin{align*}
\frac{1}{\inf_{p\in\man_\eps}\Vert\nabla f(p)\Vert} &\leq \frac{\max_{c\in\critic(f)}\mu(c)^2}{\eps}\\
&\leq \frac{\mub^2}{\eps}.
\end{align*}

As such,
$$\vol\left(\preim{f}{I}\cap\man_\eps\right)\leq\frac{12\crib\mub^2}{\eps}\cdot\inorm{f-\ff}\cdot\sup_{v\notin f(\critic(f))}\vol_{d-1}\left(\preim{f}{\{v\}}\right).$$
Similarly, with Proposition~\ref{prop:geo_vol}, we have
$$\sum_{c\in\critic(f)}\vol\left(\ball\left(c,\eps\right)\right)\leq 2\alpha_d\left|\critic(f)\right|\cdot\eps^d.$$
For small enough $\eps>0$, we thus have
\begin{align*}
\vol\left(\preim{f}{I}\right)&\leq\vol\left(\preim{f}{I}\cap\man_\eps\right)+\sum_{c\in\critic(f)}\vol\left(\ball\left(c,\eps\right)\right)\\
&\leq\frac{12\crib\mub^2}{\eps}\cdot\inorm{f-\ff}\cdot\sup_{v\notin f(\critic(f))}\vol_{d-1}\left(\preim{f}{\{v\}}\right)+2\alpha_d\crib\cdot\eps^d.
\end{align*}

By considering $\inorm{f-\ff}$ small enough, we can fix $$\eps=\inorm{f-\ff}^{\frac{1}{d+1}},$$
and we find
$$\vol\left(\preim{f}{I}\right)\leq\left[12\mub^2\cdot\sup_{v\notin f(\critic(f))}\vol_{d-1}\left(\preim{f}{\{v\}}\right)+2\alpha_d\right]\crib\cdot\inorm{f-\ff}^{\frac{d}{d+1}}.$$

Second, if $f(x)\notin I$, we will bound $\haus([x]_{\sim_f},[x]_{\sim_\ff})$ explicitly.\\
Let $y\in[x]_{\sim_\ff}$. We have:

\begin{align*}
\vert f(x)-f(y) \vert &\leq \vert f(x)-\ff(x) \vert + \vert \ff(x)-f(y) \vert\\
&= \vert f(x)-\ff(x) \vert + \vert \ff(y)-f(y) \vert\\
&\leq 2\inorm{f-\ff}.
\end{align*}

Consequently, there are no critical values between $f(x)$ and $f(y)$ as $f(x)$ is at least $3\inorm{f-\ff}$ away from the closest critical value. Moreover, this argument being true for all points in $[x]_{\sim_\ff}$, we know that $[x]_{\sim_f}$ and $[x]_{\sim_f}$ are both inside $\preim{f}{(v^-,v^+)}$ where $v^-$ (respectively $v^+$) is the closest smaller (respectively larger) critical value than $f(x)$. $x$ being a common point to $[x]_{\sim_f}$ and $[x]_{\sim_f}$, they are inside the same connected component of $\preim{f}{(v^-,v^+)}$. We can therefore use the gradient flow of $f$ to send $y$ to $[x]_{\sim_f}$. Denoting $J=[f(x),f(y)]$, this gives:
$$\dis(y,[x]_{\sim_f})\leq \frac{2\inorm{f-\ff}}{\inf_{p\in \preim{f}{J}}\Vert\nabla f(p)\Vert}.$$

Since for every point $p\in\preim{f}{J}$, $f(p)$ is at least $\inorm{f-\ff}$ away from a critical value, the same argument made in the proof of Lemma~\ref{lem:delta} gives for a small enough $\inorm{f-\ff}$:

$$\frac{1}{\inf_{p\in \preim{f}{J}}\Vert\nabla f(p)\Vert}\leq \frac{\mub}{\sqrt{\inorm{f-\ff}}},$$

and 

$$\dis(y,[x]_{\sim_f})\leq 2\mub\cdot \sqrt{\inorm{f-\ff}}.$$

A symmetrical argument, made on $\ff$ instead of $f$, gives for all $z\in[x]_{\sim_f}$:

$$\dis(z,[x]_{\sim_\ff})\leq 2\mub\cdot \sqrt{\inorm{f-\ff}}.$$

We have proved that

$$\haus([x]_{\sim_f},[x]_{\sim_\ff})^p\leq \left(2\mub\cdot \sqrt{\inorm{f-\ff}}\right)^p\cdot\ind_{f(x)\notin I}+\diam(\man)^p\cdot\ind_{f(x)\in I}.$$

Hence,

\begin{align*}
\int_\man\haus([x]_{\sim_f},[x]_{\sim_\ff})^p\,d\vol(x)\leq& (2\mub)^p\cdot \inorm{f-\ff}^{\frac{p}{2}} + \diam(\man)^p\cdot\vol\left(\preim{f}{I}\right)\\
\leq& (2\mub)^p\cdot \inorm{f-\ff}^{\frac{p}{2}} +\diam(\man)^p\crib\cdot\\
&\left[12\mub^2\cdot\sup_{v\notin f(\critic(f))}\vol_{d-1}\left(\preim{f}{\{v\}}\right)+2\alpha_d\right]\cdot\inorm{f-\ff}^{\frac{d}{d+1}}.
\end{align*}

\newpage



\end{appendices}


\bibliography{sn-bibliography}

@misc{Oulhaj_Mapper_Gromov-Wasserstein_distance,
author = {Oulhaj, Ziyad},
title = {{Mapper Gromov-Wasserstein distance examples}},
url = {https://github.com/ZiyadOulhaj/Mapper-MMS},
year = {2025},
urldate = {2026-06-23}
}

@inproceedings{Naitzat2018,
author = {Naitzat, Gregory and Lokare, Namita and Silva, Jorge and Kaynar-Kabul, Ilknur},
booktitle = {KDD Workshop on Interactive Data Exploration and Analytics},
title = {{M-Boost: profiling and refining deep neural networks with topological data analysis}},
year = {2018},
organization={Association for Computing Machinery}
}

@inproceedings{Bruel-Gabrielsson2018,
  title={Exposition and interpretation of the topology of neural networks},
  author={Gabrielsson, Rickard Br{\"u}el and Carlsson, Gunnar},
  booktitle={18th ieee international conference on machine learning and applications},
  pages={1069--1076},
  year={2019},
  organization={IEEE}
}

@inproceedings{joseph2021topological,
  title={Topological Data Analysis in Conjunction with Traditional Machine Learning Techniques to Predict Future MDAP PM Ratings},
  booktitle={18th Annual Acquisition Research Symposium},
  author={Joseph, Brian B and Pham, Trami and Hastings, Christopher},
  year={2021},
  organization={Acquisition Research Program}
}

@INPROCEEDINGS{Mitra21,
  author={Mitra, Satanik and Rao JV, Kameshwar},
  booktitle={2021 IEEE International Conference on Quantum Computing and Engineering (QCE)}, 
  title={Experiments on Fraud Detection use case with QML and TDA Mapper}, 
  year={2021},
  volume={},
  number={},
  pages={471-472},

  organization={IEEE}
  }

@article{zechel2014topographical,
  title={Topographical transcriptome mapping of the mouse medial ganglionic eminence by spatially resolved RNA-seq},
  author={Zechel, Sabrina and Zajac, Pawel and L{\"o}nnerberg, Peter and Ib{\'a}{\~n}ez, Carlos F and Linnarsson, Sten},
  journal={Genome Biol.},
  volume={15},
  pages={1--12},
  year={2014},
  publisher={Springer}
}

@inproceedings{wang2018topological,
  title={Topological methods for visualization and analysis of high dimensional single-cell RNA sequencing data},
  author={Wang, Tongxin and Johnson, Travis and Zhang, Jie and Huang, Kun},
  booktitle={Pacific symposium on biocomputing. Pacific symposium on biocomputing},
  volume={24},
  pages={350},
  year={2019},
  organization={World Scientific}
}

@inproceedings{wang2020exploration,
  title={Exploration of Topological Data Analysis In 3D Printing},
  author={Wang, Ziqi},
  booktitle={2020 International Conference on Information Science, Parallel and Distributed Systems (ISPDS)},
  pages={150--153},
  year={2020},
  organization={IEEE}
}

@article{rosen2018inferring,
  title={Inferring Quality in Point Cloud-based 3D Printed Objects using Topological Data Analysis},
  author={Rosen, Paul and Hajij, Mustafa and Tu, Junyi and Arafin, Tanvirul and Piegl, Les},
  journal={Computer-Aided Design and Applications},
  volume={16},
  pages={519--527},
  year={2018},
  publisher={CAD Solutions, LLC}
}

@inproceedings{xu2019gromov,
  title={Gromov-wasserstein learning for graph matching and node embedding},
  author={Xu, Hongteng and Luo, Dixin and Zha, Hongyuan and Duke, Lawrence Carin},
  booktitle={International conference on machine learning},
  pages={6932--6941},
  year={2019},
  organization={PMLR}
}

@article{koehl2023computing,
  title={Computing the Gromov-Wasserstein distance between two surface meshes using optimal transport},
  author={Koehl, Patrice and Delarue, Marc and Orland, Henri},
  journal={Algorithms},
  volume={16},
  number={3},
  pages={131},
  year={2023},
  publisher={MDPI}
}

@article{han2023covariance,
  title={Covariance alignment: from maximum likelihood estimation to Gromov--Wasserstein},
  author={Han, Yanjun and Rigollet, Philippe and Stepaniants, George},
  journal={SIAM Journal on Mathematics of Data Science},
  volume={7},
  number={3},
  pages={1491--1513},
  year={2025},
  publisher={SIAM}
}

@article{zhang2024gromov,
  title={Gromov--Wasserstein distances: Entropic regularization, duality and sample complexity},
  author={Zhang, Zhengxin and Goldfeld, Ziv and Mroueh, Youssef and Sriperumbudur, Bharath K},
  journal={Ann. Stat.},
  volume={52},
  number={4},
  pages={1616--1645},
  year={2024},
  publisher={Institute of Mathematical Statistics}
}

@article{carriere2022statistical,
  title={Statistical analysis of Mapper for stochastic and multivariate filters},
  author={Carri{\`e}re, Mathieu and Michel, Bertrand},
  journal={J. Appl. Comput. Topol.},
  volume={6},
  number={3},
  pages={331--369},
  year={2022},
  publisher={Springer}
}

@article{ruscitti2024improving,
  title={Improving Mapper's Robustness by Varying Resolution According to Lens-Space Density},
  author={Ruscitti, Kaleb D and McInnes, Leland},
  journal={arXiv preprint arXiv:2410.03862},
  year={2024}
}

@article{brown2021probabilistic,
  title={Probabilistic convergence and stability of random mapper graphs},
  author={Brown, Adam and Bobrowski, Omer and Munch, Elizabeth and Wang, Bei},
  journal={J. Appl. Comput. Topol.},
  volume={5},
  number={1},
  pages={99--140},
  year={2021},
  publisher={Springer}
}

@article{bauer_reeb_2021,
	title = {The {Reeb} {Graph} {Edit} {Distance} is {Universal}},
	volume = {21},
	issn = {1615-3383},
	doi = {10.1007/s10208-020-09488-3},
	number = {5},
	journal = {Found. Comput. Math.},
	author = {Bauer, Ulrich and Landi, Claudia and Mémoli, Facundo},
	year = {2021},
	pages = {1441--1464},
}

@inproceedings{bauer_edit_2016,
  title={An Edit Distance for Reeb Graphs},
  author={Bauer, Ulrich and Fabio, Barbara Di and Landi, Claudia},
  booktitle={Eurographics Workshop on 3D Object Retrieval},
  pages={27--34},
  year={2016},
  organization={Eurographics Association}
}

@inproceedings{bauer_measuring_2014,
author = {Bauer, Ulrich and Ge, Xiaoyin and Wang, Yusu},
booktitle = {30th Annual Symposium on Computational Geometry},
pages = {464--473},
organization = {Association for Computing Machinery},
title = {{Measuring distance between Reeb graphs}},
year = {2014}
}

@article{liFlexibleProbabilisticTopology2025,
  title={Flexible and probabilistic topology tracking with partial optimal transport},
  author={Li, Mingzhe and Yan, Xinyuan and Yan, Lin and Needham, Tom and Wang, Bei},
  journal={IEEE Trans. Vis. Comput. Graph.},
  year={2025},
  publisher={IEEE}
}

@inproceedings{wangMeasureTheoreticReebGraphs2024,
	title = {Measure-{{Theoretic Reeb Graphs}} and {{Reeb Spaces}}},
	booktitle = {40th {{International Symposium}} on {{Computational Geometry}}},
	author = {Wang, Qingsong and Ma, Guanqun and Sridharamurthy, Raghavendra and Wang, Bei},
	year = {2024},
	volume = {293},
	pages = {80:1--80:18},
	issn = {1868-8969},
	isbn = {978-3-95977-316-4},
	organization = {Association for Computing Machinery},
}

@inproceedings{munch_convergence_2016,
	title = {Convergence between {Categorical} {Representations} of {Reeb} {Space} and {Mapper}},
	volume = {51},
	isbn = {978-3-95977-009-5},
	booktitle = {32nd {International} {Symposium} on {Computational} {Geometry}},
	author = {Munch, Elizabeth and Wang, Bei},
	year = {2016},
	pages = {53:1--53:16},
	organization = {Association for Computing Machinery},
}

@book{sakai1996riemannian,
  title={Riemannian geometry},
  author={Sakai, Takashi},
  volume={149},
  year={1996},
  publisher={American Mathematical Society},
  address={Providence}
}

@article{de2016categorified,
  title={Categorified reeb graphs},
  author={{d}e Silva, Vin and Munch, Elizabeth and Patel, Amit},
  journal={Discrete Comput. Geom.},
  volume={55},
  number={4},
  pages={854--906},
  year={2016},
  publisher={Springer}
}

@article{burgisser2018computing,
  title={Computing the homology of basic semialgebraic sets in weak exponential time},
  author={B{\"u}rgisser, Peter and Cucker, Felipe and Lairez, Pierre},
  journal={J. ACM},
  volume={66},
  number={1},
  pages={1--30},
  year={2018},
  publisher={ACM New York, NY, USA}
}

@article{aamari2019estimating,
  title={Estimating the Reach of a Manifold},
  author={Aamari, Eddie and Kim, Jisu and Chazal, Fr{\'e}d{\'e}ric and Michel, Bertrand and Rinaldo, Alessandro and Wasserman, Larry},
  journal={Electron. J. Stat.},
  year={2019}
}

@article{flamary2021pot,
  author  = {R{\'e}mi Flamary and Nicolas Courty and Alexandre Gramfort and Mokhtar Z. Alaya and Aur{\'e}lie Boisbunon and Stanislas Chambon and Laetitia Chapel and Adrien Corenflos and Kilian Fatras and Nemo Fournier and L{\'e}o Gautheron and Nathalie T.H. Gayraud and Hicham Janati and Alain Rakotomamonjy and Ievgen Redko and Antoine Rolet and Antony Schutz and Vivien Seguy and Danica J. Sutherland and Romain Tavenard and Alexander Tong and Titouan Vayer},
  title   = {POT: Python Optimal Transport},
  journal = {J. Mach. Learn. Res.},
  year    = {2021},
  volume  = {22},
  number  = {78},
  pages   = {1-8},
  url     = {http://jmlr.org/papers/v22/20-451.html}
}

@inproceedings{memoli2007use,
  title={On the use of Gromov-Hausdorff Distances for Shape Comparison},
  author={M{\'e}moli, Facundo},
  booktitle={4th Symposium on Point Based Graphics},
  pages={81--90},
  year={2007},
  organization={The Eurographics Association}
}

@inproceedings{memoli2009spectral,
  title={Spectral Gromov-Wasserstein distances for shape matching},
  author={M{\'e}moli, Facundo},
  booktitle={12th International Conference on Computer Vision Workshops},
  pages={256--263},
  year={2009},
  organization={IEEE}
}

@article{memoli2011gromov,
  title={Gromov--Wasserstein distances and the metric approach to object matching},
  author={M{\'e}moli, Facundo},
  journal={Found. Comput. Math.},
  volume={11},
  pages={417--487},
  year={2011},
  publisher={Springer}
}

@article{niyogi2008finding,
  title={Finding the homology of submanifolds with high confidence from random samples},
  author={Niyogi, Partha and Smale, Stephen and Weinberger, Shmuel},
  journal={Discrete Comput. Geom.},
  volume={39},
  pages={419--441},
  year={2008},
  publisher={Springer}
}

@article{carriere2018structure,
  title={Structure and stability of the one-dimensional mapper},
  author={Carriere, Mathieu and Oudot, Steve},
  journal={Found. Comput. Math.},
  volume={18},
  pages={1333--1396},
  year={2018},
  publisher={Springer}
}

@article{arias2011spectral,
  title={Spectral clustering based on local linear approximations},
  author={Arias-Castro, Ery and Chen, Guangliang and Lerman, Gilad},
  journal={Electron. J. Stat.},
  volume={5},
  pages={1537--1587},
  year={2011},
  publisher={Institute of Mathematical Statistics}
}

@article{weed2019sharp,
  title={Sharp asymptotic and finite-sample rates of convergence of empirical measures in Wasserstein distance},
  author={Weed, Jonathan and Bach, Francis},
  journal={Bernoulli},
  volume={25},
  number={4A},
  pages={2620--2648},
  year={2019},
  publisher={JSTOR}
}

@article{BoissardLeGouic2014,
author = {Emmanuel Boissard and Thibaut Le Gouic},
title = {{On the mean speed of convergence of empirical and occupation measures in Wasserstein distance}},
volume = {50},
journal = {Ann. Henri Poincare},
number = {2},
publisher = {Institut Henri Poincaré},
pages = {539 -- 563},
year = {2014},
doi = {10.1214/12-AIHP517},
URL = {https://doi.org/10.1214/12-AIHP517}
}

@article{sturm2006geometry,
  title={On the geometry of metric measure spaces},
  author={Sturm, Karl-Theodor},
  journal={Acta Math.},
  volume={196},
  number={1},
  pages={65--131},
  year={2006},
  publisher={International Press of Boston}
}

@book{milnor1963morse,
  title={Morse theory},
  author={Milnor, John Willard},
  year={1963},
  publisher={Princeton university press},
  address={Princeton}
}

@article{henrikson1999completeness,
  title={Completeness and total boundedness of the Hausdorff metric},
  author={Henrikson, Jeff},
  journal={MIT Undergraduate Journal of Mathematics},
  volume={1},
  number={69-80},
  pages={10},
  year={1999},
  publisher={Citeseer}
}

@inproceedings{singh2007topological,
  title={Topological Methods for the Analysis of High Dimensional Data Sets and 3D Object Recognition},
  author={Singh, Gurjeet and M{\'e}moli, Facundo and Carlsson, Gunnar},
  booktitle={4th Symposium on Point Based Graphics},
  pages={91--100},
  year={2007},
  organization={The Eurographics Association}
}

@book{petersen2006riemannian,
  title={Riemannian geometry},
  author={Petersen, Peter},
  year={2006},
  publisher={Springer},
  address={Cham}
}

@inproceedings{chazal2014convergence,
  title={Convergence rates for persistence diagram estimation in topological data analysis},
  author={Chazal, Fr{\'e}d{\'e}ric and Glisse, Marc and Labru{\`e}re, Catherine and Michel, Bertrand},
  booktitle={International Conference on Machine Learning},
  pages={163--171},
  year={2014},
  organization={PMLR}
}

@article{carriere2018statistical,
  title={Statistical analysis and parameter selection for mapper},
  author={Carriere, Mathieu and Michel, Bertrand and Oudot, Steve},
  journal={J. Mach. Learn. Res.},
  volume={19},
  number={12},
  pages={1--39},
  year={2018}
}

@article{cuevas2009set,
  title={Set estimation: Another bridge between statistics and geometry},
  author={Cuevas, Antonio},
  journal={Bol. Estad. Investig. Oper.},
  volume={25},
  number={2},
  pages={71--85},
  year={2009}
}

@book{tu2011manifolds,
  author={Tu, Loring W},
  title={An Introduction to Manifolds},
  year={2008},
  publisher={Springer},
  address={New York}
}

@article{kerdoncuff2021sampled,
  title={Sampled gromov wasserstein},
  author={Kerdoncuff, Tanguy and Emonet, Remi and Sebban, Marc},
  journal={Mach. Learn.},
  volume={110},
  number={8},
  pages={2151--2186},
  year={2021},
  publisher={Springer}
}

\end{document}